%% file: main.tex
\documentclass[]{amsart}
\input{librairies_and_commands}

\title{Lipschitz solutions to mean field games with a major player and applications}
\author{Meynard Charles}
\begin{document}
\begin{abstract}
This paper introduces a notion of weak solution for the coupled system of master equations in mean field games with a major player. It extends the previously introduced notion of Lipschitz solutions in mean field games. By relying on a probabilistic representation of the system of master equations, we prove that there can exist at most one sufficiently smooth solution and that it is consistent with the associated Nash equilibrium. In this approach, coefficients are only required to be Lipschitz, in particular, no differentiability assumption with respect to probability measures is needed. 
In a second part, we apply this notion of solution to prove the existence and uniqueness of solutions to MFGs with a major player on intervals of arbitrary length. Our argument relies on assuming that the intensity of the Brownian common noise driving the state of the major player is sufficiently large, as well as a joint displacement monotonicity assumption between the coefficients of minor players and those of the major player. Most notably, this joint monotonicity allows us to prove that the threshold of volatility can be taken independently of the horizon of the game considered, without any long time decoupling assumption on the dependence between minor players and the major player. Finally, inspired by recent extragradient methods for mean field games, we present an algorithm that converges exponentially fast to the solution of the major-minor probabilistic system under this monotonicity assumption. 
Thanks to the generality of our approach, all results presented in this article hold for mean field games of controls with a major player.
\end{abstract}
\maketitle

\section{Introduction}
\subsection{General introduction}
Mean field games (MFGs) were introduced in \cite{Lions-college} and subsequently in\cite{LasryLionsMFG} by Lasry and Lions to study Nash equilibria in differential games with a large number of players. If players interact with each other only through aggregation terms, the mean field limit yields a system consisting of a Hamilton-Jacobi-Bellman (HJB) equation for the evolution of the players' decisions coupled with a Fokker-Planck equation for their distribution over time. This system is often referred to as the mean field game system \cite{LasryLionsMFG,cardaliaguet2010notes}. An alternative formulation of MFGs consists in introducing the master equation, a single, infinite dimensional equation on the space of probability measures that encompasses the information of the whole MFG system \cite{convergence-problem,Lions-college}. Throughout the years, the wellposedness of the MFG master equation has been studied extensively; see the books \cite{convergence-problem,carmona2018probabilistic} for the existence of smooth solutions using both PDEs and probabilistic arguments. The existence of solutions over intervals of arbitrary length usually relies on a monotonicity condition that ensures the uniqueness and wellposedness of the associated Nash equilibrium. The notion of flat monotonicity was used in the works \cite{convergence-problem,carmona2018probabilistic,Bertucci02122023}. Another standard notion of monotonicity is the Hilbertian monotonicity \cite{Lions-college,monotone-sol-meynard}, also referred to as displacement monotonicity \cite{disp-monotone-1,disp-monotone-2,globalwellposednessdisplacementmonotone,propagation-of-monotonicity}. In this article we rely on this latter notion of monotonicity.

Due to the difficulties in working with smooth solutions to the master equation, several notions of weak solutions have been introduced in the literature. For potential mean field games, we refer to \cite{CecchinDelarue2025}. For notions of weak solutions relying on an underlying notion of monotonicity, we refer to \cite{Bertucci02122023,cardaliaguet-monotone} in the flat monotone regime and to \cite{monotone-sol-meynard} under displacement monotonicity. Finally notions of weak solutions relying directly on the structure of the master equation were presented in \cite{lipschitz-sol,weaksol-dipmono}. Although none of these articles treat of MFGs with a major player, these last two works are the most similar to the notion of weak solution we introduce in the paper, which is largely inspired by \cite{lipschitz-sol}.

As for MFG with common noise, these models arise from the addition of a noise impacting the dynamics of all players in the same fashion \cite{mfg-with-common-noise,Lions-college}. Models in which such noise is purely additive have led to master equations of second order \cite{convergence-problem} and have been studied extensively in the literature regarding both MFGs \cite{convergence-problem,Bertucci02122023,globalwellposednessdisplacementmonotone,doi:10.1137/21M1450008} and mean field type FBSDEs \cite{probabilistic-mfg,monotone-sol-meynard}. On the other hand, common noise can also arise from an additional stochastic process, as in \cite{noise-add-variable,common-noise-in-MFG}. If, furthermore, this common noise process is controlled by another player, we refer to the problem as a mean field game with a major player \cite{MFG-MAJOR-LIONS}. In this article, we are interested in the latter problem, and rely extensively on ideas introduced in \cite{noise-add-variable,common-noise-in-MFG} for the notion of monotonicity that yields the wellposedness of our problem.

In MFGs with a major player introduced in \cite{Huang2010,Huang2012,NourianCaines2013}, each individual competes not only with the crowd of other players but also with an additional player whose decisions affect all players, hence the terminology. This major player solves another control problem, which depends on the distribution of minor players. Consequently, the solvability of major-minor MFG is that of a system of coupled optimal control problems with a mean field element. In this framework, different notion of equilibrium have been proposed depending on how the crowd interacts with the major player. In \cite{Bensoussanetal2015,Bergaultetal2024} the authors consider Stackelberg equilibria whereas in the present work we focus on Nash equilibria. Even in this framework, solutions may differ depending on whether the problem is approached with open loop controls \cite{CarmonaWang2017,CarmonaZhu2016} or closed loop controls, which is the setting of interest throughout this article. Many works have already addressed this problem \cite{probabilistic-mfg}, in particular through its master equation. First in \cite{MFG-MAJOR-LIONS} and in a series of paper \cite{CardaliaguetCirantPorretta2020,CardaliaguetCirantPorretta2023}. While they may address short time existence and uniqueness \cite{MFG-MAJOR-LIONS,CardaliaguetCirantPorretta2023}, none of those articles treats this problem over time intervals of arbitrary length. To the extent of our knowledge, the only paper considering this problem is \cite{minor-major-delarue}. They rely on a flat monotonicity assumption \cite{Lions-college,convergence-problem} on the dynamics of minor players, combined with a non degeneracy assumption on the volatility of the major player. Then, they prove that, given a terminal time $T$, there exists a volatility threshold for the major player above which the major-minor system is well posed up until $T$. Under additional structural assumptions, they establish that this threshold can be chosen independently of the time horizon $T$. In spirit, our results are quite similar in that we also rely on a volatility threshold for the major player. However, our approach is very different. First, the notions of monotonicity considered in the two articles is quite different. We rely on a joint displacement monotonicity assumption between the major and the minors players instead of flat monotonicity restricted to minor players. While it imposes stronger conditions on the dynamics of coefficients, This allows us to considerably weaken the assumptions necessary to ensure that this volatility threshold can be taken independently of the horizon of game. This stands in contrast to to \cite{minor-major-delarue} which relies on a long time decoupling assumption between the dynamics of the major and minor players.

The terminology MFGs of controls is used to indicate a class of extended MFGs \cite{extended-mfg} in which the dynamics of each minor player depend on both the distribution of states and the distribution of controls of the crowd. This class of MFGs was introduced in \cite{CardaliaguetLehalle2018}. In general, since the control of an individual player now depends on its own law, there is a need for an additional fixed point in the definition of equilibrium. Those models have been studied in the literature \cite{propagation-of-monotonicity,Kobeissi2019,Kobeissi2022} under different notions of monotonicity. Most notably, it was observed in \cite{jackson2025quantitativeconvergencedisplacementmonotone} that in the displacement monotone regime, the wellposedness of this fixed point is natural. However, none of the above articles address MFGs of controls with a major player. Setting aside the wellposedness of this problem, to the extend of our knowledge this is the first work formulating Nash equilibria in MFGs of controls with a major player. In the presence of a major player, the fixed point characteristics of MFGs of controls also depend on its state and possibly even its control. We deal with this difficulty by relying on the approach of \cite{extragradient-in-mfg} in which the fixed point is made explicit in terms of the inverse of a function on a Hilbert space. This allows us to treat MFGs of controls in the same way as more standard extended MFGs. In particular both our notion of weak solution and our long time existence theory admit straightforward extensions to MFG of controls with a major player, which we detail.

\subsection{Setting of the study}
In this article, we are interested in the solvability of the coupled system of master equations arising in mean field games with a major player 
\begin{gather}
    \nonumber \left\{
        \begin{array}{c}
            \partial_t \mathcal{U}+H(x,q,\nabla_x \mathcal{U},m)-\sigma_x \Delta_x \mathcal{U}\\
            \displaystyle +\int_{\reels^d}D_p H(y,q,\nabla_x \mathcal{U}(t,y,q,m),m)\cdot D_m \mathcal{U}(t,x,q,m)(y)m(dy)-\sigma_x \int_{\reels^d} \text{div}_y\left(D_m\mathcal{U}(t,x,q,m)(y)\right)m(dy)\\
            +D_p H^0(q,m,\nabla_q \varphi)\cdot \nabla_q \mathcal{U}-\sigma^0\Delta_q \mathcal{U}=0\\
            \forall (t,x,q,m)\in (0,T)\times \reels^d\times \reels^{d^0}\times \mathcal{P}_2(\reels^d),\\
        \end{array}
    \right.\\
   \label{eq: master system}\\
    \nonumber \left\{
        \begin{array}{c}
            \partial_t \varphi+H^0(q,\nabla_q \varphi,m)-\sigma^0 \Delta_q \varphi\\
            \displaystyle +\int_{\reels^d}D_p H(y,q,\nabla_x \mathcal{U}(t,y,q,m),m)\cdot D_m \varphi(t,q,m)(y)m(dy)-\sigma_x \int_{\reels^d} \text{div}_y\left(D_m\varphi(t,q,m)(y)\right)m(dy)=0\\
             \forall (t,q,m)\in (0,T)\times \reels^{d^0}\times \mathcal{P}_2(\reels^d),\\
        \end{array}
    \right.
\end{gather}
and its link the with existence and uniqueness of Nash equilibria. Together, those two equations describes the dynamics of both minor players and the major player along Nash equilibria. In this framework, $\mathcal{U}(t,x,q,m)$ gives the value at time $t$ for a single minor player in state $x$, whenever the state of the major player is $q$ and the distribution of all other minor players is given by $m$. Meanwhile $\varphi(t,q,m)$ gives the value for the major player in the same conditions. The fact that $\varphi$ does not depend on $x$ is a key feature of MFGs with a major player: while the crowd of players is sensitive to any change in the state of the major player, the dynamics of the latter depend on minors players only through the movements of the crowd as a whole. The two Hamiltonians $H,H^0$ correspond respectively to the control problems solved by minor players and by the major player. The main difficulty of MFGs with a major player is that these problems are strongly coupled. As a consequence standard arguments for both mean field games \cite{convergence-problem,probabilistic-mfg} and for Hamilton-Jacobi-Bellman equations on the space of probability measures \cite{DaudinSeeger2024} fail in this setting. 

Due to the non-linear nature of the problem, it is unclear whether there exist smooth solutions to the master system system \eqref{eq: master system} in general. In this article, we introduce a notion of weak solution which relies on the characteristics of the problem to solve \eqref{eq: master system}. Namely we study 
\begin{equation}
\label{intro: fbsde mfg with major}
\left\{
\begin{array}{l}
     \displaystyle X_t=X_0-\int_0^t  D_u H(X_s,q_s,U_s,\mathcal{L}(X_s|\mathcal{F}^0_s))ds+\sqrt{2\sigma}B_t, \\
     \displaystyle q_t=q_0-\int_0^t D_z H^0(q_s,Z^\varphi_s,\mathcal{L}(X_s|\mathcal{F}^0_s))ds+\sqrt{2\sigma^0}W^0_t,\\
    \displaystyle  U_t= U_T+\int_t^T D_x H(X_s,q_s,U_s,\mathcal{L}(X_s|\mathcal{F}^0_s))ds-\int_t^T
    Z_sd(B,W^0)_s,\\
\displaystyle \varphi_t= \varphi_T+\int_t^T L^0_H(q_s,Z^\varphi_s,\mathcal{L}(X_s|\mathcal{F}^0_s))ds-\sqrt{2\sigma^0}\int_t^T Z^\varphi_sdW^0_s,\\
U_T=\nabla_x \mathcal{U}(T,X_T,q_T,\mathcal{L}(X_T|\mathcal{F}^0_T)),\\
\varphi_T=\varphi(T,q_T,\mathcal{L}(X_T|\mathcal{F}^0_T)),
\end{array}
\right.
\end{equation}
where $L^0_H$ is given by 
\[L^0_H:(q,z,m)\mapsto H^0(q,z,m)-D_z H^0\cdot z,\]
and $(\mathcal{F}^0_t)_{t\geq 0}$ indicates the filtration of common noise. 
Formally if there exists a smooth solution to \eqref{eq: master system} and the volatily associated to common noise is not degenerate 
\[\sigma^0>0,\]
then $(\nabla_x \mathcal{U}, \varphi)$ is a decoupling field for the forward backward stochastic differential equation (FBSDE) \eqref{intro: fbsde mfg with major}, that is for any initial condition $(X_0,q_0)$ 
\[ \forall t\in [0,T], \quad U_t=\nabla_x \mathcal{U}(T-t,X_t,q_t,\mathcal{L}(X_t|\mathcal{F}^0_t)), \quad \varphi_t= \varphi(T-t,q_t,\mathcal{L}(X_t|\mathcal{F}^0_t)) \quad a.s.\]
Conversely if there exists a decoupling field to the FBSDE \eqref{intro: fbsde mfg with major}, then it is natural to expect that it is a solution to \eqref{eq: master system}, this is the main idea behind the notion of solution adopted in this article. Compared to the master system, one of the main advantages of studying \eqref{intro: fbsde mfg with major}, is that this system does not require nearly as much regularity of the coefficients or the solution. In particular, we show that as soon as there exists a Lipschitz decoupling field to \eqref{intro: fbsde mfg with major}, then this weak solution to the master system \eqref{eq: master system} enjoys properties of stability and uniqueness, and is also consistent with the notion of Nash equilibrium in the setting of MFGs with a major player. Hence the terminology of Lipschitz solutions. At this point, it might seem surprising that for minor players, we rely on the equation satisfied by the gradient of the value function $\nabla_x\mathcal{U}$ instead of $\mathcal{U}$ directly. As highlighted in \cite{lipschitz-sol}, this is because for MFGs the key regularity of the problem is that of $\nabla_x \mathcal{U}$, which correlates to the regularity of feedback controls for minor players. In particular, from a probabilistic point of view, looking at the stochastic characteristics \eqref{intro: fbsde mfg with major} is very natural for minor players since this corresponds to the system obtained from Pontryagin's maximum principle.

In general, instead of focusing directly on the system \eqref{intro: fbsde mfg with major}, we analyze the existence and uniqueness of decoupling fields associated to systems of the form 
\begin{equation}
\label{intro: general system}
\left\{
\begin{array}{l}
     \displaystyle X_t=X_0-\int_0^t  F(X_s,q_s,U_s,\varphi_s,Z^\varphi_s,\mathcal{L}(X_s,U_s|\mathcal{F}^0_s))ds+\sqrt{2\sigma}B_t, \\
     \displaystyle q_t=q_0-\int_0^t H_z(q_s,\varphi_s,Z^\varphi_s,\mathcal{L}(X_s,U_s|\mathcal{F}^0_s))ds+\sqrt{2\sigma^0}W^0_t,\\
    \displaystyle  U_t= U_T+\int_t^T G(X_s,q_s,U_s,\varphi_s,Z^\varphi_s,\mathcal{L}(X_s,U_s|\mathcal{F}^0_s))ds-\int_t^T
    Z_sd(B,W^0)_s,\\
\displaystyle \varphi_t= \varphi_T+\int_t^T L_H(q_s,\varphi_s,Z^\varphi_s,\mathcal{L}(X_s,U_s|\mathcal{F}^0_s))ds-\sqrt{2\sigma^0}\int_t^T Z^\varphi_sdW^0_s,\\
U_T=g(X_T,q_T,\mathcal{L}(X_T|\mathcal{F}^0_T)), \quad \varphi_T=\psi(q_T,\mathcal{L}(X_T|\mathcal{F}^0_T)),
\end{array}
\right.
\end{equation}
for Lipschitz functions, locally in the argument associated to $(Z^\varphi_t)_{t\in [0,T]}$.  This approach does not make the analysis much more technical and allows us to study a wider class of models. In particular, \eqref{intro: general system} includes the class of extended MFGs \cite{extended-mfg}. This added generality allows us to treat the case of mean field games of controls with a major player, that is MFGs in which both the dynamics of minors players and the major player depend on the law of the optimal control of minor players. We can also consider situations in which the dynamics of minor players depend on the control of the major player directly. For \eqref{eq: master system}, this means that the Hamiltonian $H$ of minor players is a function of $\nabla_q \varphi$. Finally, we will also explain how \eqref{intro: general system} includes MFGs with a major player and additive common noise \cite{convergence-problem}, by which we MFGs in which the dynamics of minor players is given by 
\[dX_t=\alpha_tdt+\sqrt{2\sigma}dB_t+\sqrt{2v^0}dB^0_t,\]
for $(B^0_t)_{t\geq 0}$ a Brownian motion shared among all minor players and possibly correlated with $(W^0_t)_{t\geq 0}$, Brownian motion of the major player (up to a non degeneracy assumption on the correlation matrix). 
\subsection{Main contributions}
Extending on results presented in \cite{lipschitz-sol}, we introduce a new notion of weak solution to MFGs with a major player. In this class of weak solutions we prove stability, uniqueness and local in time existence. These results are quite general and do not require any monotonicity assumptions, instead they rely only on the Lipschitz regularity of coefficients. Moreover, we prove that these results are also valid for extended MFGs with a major player, and in particular MFGs of controls with major player. 

Then, we present new results on the long time existence of extended MFGs with a major player under $L^2-$monotonicity. Much like \cite{minor-major-delarue} we require the volatility of the common noise associated to the major player $\sigma^0$ to be sufficiently large. Under sufficiently strong assumptions, we demonstrate that there exists a volatility threshold independent of the horizon of the game $\sigma^0_*$ such that if $\sigma^0>\sigma^0_*$ there exists a Nash equilibrium \eqref{intro: fbsde mfg with major} for the minor-major problem on $[0,T]$ for any $T>0$. We are able to treat fully coupled systems where the dynamics of both minor players and the major player can depend on the law of controls of minor players. Moreover, the existence of a threshold $\sigma^0_*$ does not rely on a decoupling between the dynamics of the minors and the major player in long time but rather on a joint monotonicity assumption between the major and minor players. 

Finally, we extend the recently introduced extragradient methods in MFGs \cite{extragradient-in-mfg} to the setting of MFGs with a major player under the monotone setting we introduced in this article. This consists in an approximating scheme for the system \eqref{intro: fbsde mfg with major} by iterating on decoupled FBSDEs. Under appropriate regularity and monotonicity assumptions, we show that the scheme achieves geometric convergence to the solution of \eqref{intro: fbsde mfg with major}.
\subsection{Organization of the paper}
In section \ref{section: MFGMP} we introduce a notion of weak solution to MFGs with a major player which is used in the rest of the article. We prove general properties of uniqueness and stability within this class of solutions. This culminates in the main existence result, Theorem \ref{thm: existence mfgmp} for Lipschitz solutions. We also present in full detail the notion of Nash equilibrium in MFGs with a major player and prove that Lipschitz solution are consistent with this notion in the Subsection \ref{subsection: Nash}. 

Then, we present a joint monotonicity condition on the dynamics of minor players and the major player, and develop a framework for the long time existence of Lipschitz solutions under this notion of monotonicity. As presented in Theorem \ref{thm: existence for sigma T mfgmp}, we assume that the intensity of the common noise $\sigma^0$ must be sufficiently large relative to the time interval considered. Then, in Theorem \ref{thm sigma independent of T}, we give additional assumptions under which it is sufficient for this volatility to be larger than a threshold independent of the horizon of the game $T$. 

In a last short section, we present a numerical method to solve the FBSDE \eqref{intro: fbsde mfg with major} by iteratively solving decoupled systems of FBSDEs. This is based on extending the concept of extragradient methods in MFGs under the particular notion of monotonicity considered in this article. Theorem \ref{thm: exponential convergence mfg} gives assumptions under which the algorithm we present achieves geometric convergence. 
\subsection{Notation}
\begin{enumerate}
\item[-] let $k\in \mathbb{N}$, $k>0$, for the canonical product on $\reels^k$ we use the notation
\[x\cdot y=\sum_i x_iy_i,\]
and the following notation for the induced norm
\[|x|=\sqrt{x\cdot x}.\]
    \item[-] Let $\mathcal{P}(\reels^d)$ be the set of (Borel) probability measures on $\reels^d$, for $q\geq 0$, we use the usual notation 
\[\mathcal{P}_q(\reels^d)=\left\{\mu\in \mathcal{P}(\reels^d), \quad \int_{\reels^d} |x|^q \mu(dx)<+\infty\right\},\]
for the set of all probability measures with a finite $q$th moment.
\item[-] For two measures $\mu,\nu\in\mathcal{P}(\reels^d)$ we define $\Gamma(\mu,\nu)$ to be the set of all probability measures $\gamma\in\mathcal{
P}(\reels^{2d})$ satisfying 
\[\gamma(A\times \reels^d)=\mu(A) \quad \gamma(\reels^d\times A)=\nu(A),\]
for all Borel set $A$ on $\reels^d$. 
\item[-]The Wasserstein distance between two measures belonging to $\mathcal{P}_q(\reels^d)$ is defined as
\[\mathcal{W}_q(\mu,\nu)=\left(\underset{\gamma\in \Gamma(\mu,\nu)}{\inf}\int_{\mathbb{\reels}^{2d}} |x-y|^q\gamma(dx,dy)\right)^{\frac{1}{q}}.\]
In what follows $\mathcal{P}_q(\reels^d)$ is always endowed with the associated Wasserstein distance, $(\mathcal{P}_q,\mathcal{W}_q)$ being a complete metric space. 
\item[-] We say that a function $U:\mathcal{P}_q(\reels^d)\to \reels^d$ is Lipschitz if
\[\exists C\geq 0, \quad \forall(\mu,\nu)\in\left(\mathcal{P}_q(\reels^d)\right)^2, \quad |U(\mu)-U(\nu)|\leq C\mathcal{W}_q(\mu,\nu).\]
\item[-] Consider $(\Omega, \mathcal{F},\mathbb{P})$ a probability space,
\begin{itemize}
    \item[-] We define 
\[L^q(\Omega, \reels^d)=\left\{ X: \Omega\to \reels^d, \quad \esp{|X|^q}<+\infty\right\}.\]
\item[-] Whenever a random variable $X:\Omega\to \reels^d$ is distributed along $\mu\in \mathcal{P}(\reels^d)$ we use equivalently the notations $\mathcal{L}(X)=\mu$ or  $X\sim\mu$. 
\end{itemize}
\item[-] $\mathcal{M}_{d\times n}(\reels)$ is the set of all matrices of size $d\times n$ with reals coefficients, with the notation $\mathcal{M}_n(\reels)\equiv \mathcal{M}_{n\times n}(\reels)$.
\end{enumerate}
\subsection{Preliminary results}
In this article, we shall consider problems associated to MFGs with common noise coming from the dynamics of a major player. To that end, we borrow the probabilistic setting of \cite{probabilistic-mfg} Section 7. Namely, we fix a horizon $T>0$ and consider a complete probability space $(\Omega^0,\mathcal{F}^0,\mathbb{P}^0)$ such that $(\mathcal{F}^0_t)_{0\leq t\leq T}$ is a complete and right continuous filtration, and there exists $(W^0_t)_{0\leq t\leq T}$ a $d^0-$dimensional $\mathcal{F}^0-$Brownian motion. Similarly we are given $(\Omega^1,\mathcal{F}^1,\mathbb{P}^1)$ a complete probability space with a complete and right continuous filtration $(\mathcal{F}^1_t)_{0\leq t\leq T}$ and endowed with a $d-$dimensional $\mathcal{F}^1-$Brownian motion $(B_t)_{0\leq t\leq T}$. Then we work on $(\Omega,\mathcal{F},\mathbb{P})$ the completion of the product space $(\Omega^0\otimes \Omega^1,\mathcal{F}^0\otimes \mathcal{F}^1,\mathbb{P}^0\otimes \mathbb{P}^1)$ endowed with the filtration $(\mathcal{F}_t)_{0\leq t\leq T}$ obtained by augmenting the product filtration $(\mathcal{F}^0_t)_{0\leq t\leq T}\otimes(\mathcal{F}^1_t)_{0\leq t\leq T}$ to make it right continuous and completing it. In particular we recall that the space $L^2(\Omega,\reels)$ endowed with the inner product 
\[\langle \cdot,\cdot \rangle: (X,Y)\mapsto \esp{XY},\]
is a Hilbert space. We use the notation $\mathcal{H}=(L^2(\Omega,\reels),\langle \cdot,\cdot\rangle)$ and the induced norm is denoted by
\[\forall X\in \mathcal{H}, \|X\|=\sqrt{\langle XX\rangle}.\]
We also use the notation $\mathcal{H}^T=(L^1([0,T],L^2(\Omega,\reels)),\langle \cdot,\cdot\rangle^T)$ for the Hilbert space of square integrable random processes on $[0,T]$ endowed with the inner product 
\[\langle X,Y\rangle^T=\esp{\int_0^T X_tY_tdt}.\]
Given a filtration $(\mathcal{G}_t)_{t\in [0,T]}$, we use the notation
\[\mathcal{H}^T_\mathcal{G}=\left\{ (\alpha_t)_{t\in [0,T]} \in \mathcal{H}^T, (\alpha_t)_{t\in [0,T]} \text{ is }\mathcal{G}-\text{adapted} \right\}.\]
Finally, consider a forward backward system of conditional McKean-Vlasov type defined on $[t_0,T_m]$ for $0\leq t_0\leq T_m\leq  T$,
\begin{equation}
\label{fbsde conditional mckean vlasov}
\left\{
\begin{array}{l}
     \displaystyle X_t=X_{t_0}-\int_{t_0}^t  f_1(X_s,q_s,U_s,\varphi_s,Z^2_s,\mathcal{L}(X_s,U_s|\mathcal{F}^0_s))ds+\sqrt{2\sigma}B_t, \\
     \displaystyle q_t=q_{t_0}-\int_{t_0}^t f_2(q_s,\varphi_s,Z^2_s,\mathcal{L}(X_s,U_s|\mathcal{F}^0_s))ds+\sqrt{2\sigma^0}W^0_t,\\
    \displaystyle  U_{t}= U_T+\int_t^{T_m} b_1(X_s,q_s,U_s,\varphi_s,Z^2_s,\mathcal{L}(X_s,U_s|\mathcal{F}^0_s))ds-\int_t^{T_m}
    Z^1_s(B,W^0)_{s},\\
\displaystyle \varphi_t =\varphi_T+\int_t^{T_m} b_2(q_s,\varphi_s,Z^2_s,\mathcal{L}(X_s,U_s|\mathcal{F}^0_s))ds-\sqrt{2\sigma^0}\int_t^{T_m} Z^2_sdW^0_{s},\\
U_T=g_1(X_{T_m},q_{T_m},\mathcal{L}(X_{T_m}|\mathcal{F}^0_{T_m})), \quad \varphi_T=g_2(q_{T_m},\mathcal{L}(X_{T_m}|\mathcal{F}^0_{T_m})).
\end{array}
\right.
\end{equation}
\begin{definition}
Let $(g_1,g_2,f_1,f_2,b_1,b_2)$ be continuous functions in all of their arguments. We say that $(X_{t_0},q_{t_0})$ is an admissible initial condition if and only if 
\begin{enumerate}
    \item[-]$(X_{t_0},q_{t_0})\in \mathcal{H}^{d+d^0}$,
    \item[-] $X_{t_0}$ is $\mathcal{F}_{t_0}-$measurable,
    \item[-] $q_{t_0}$ is $\mathcal{F}^0_{t_0}-$measurable. 
\end{enumerate}
Moreover, a tuple $(X_t,q_t,U_t,\varphi_t,Z^1_t,Z^2_t)_{t\in [t_0,T_m]}$ is said to be a strong solution to \eqref{fbsde conditional mckean vlasov}, if and only if 
\begin{enumerate}
    \item[-] $(X_t,U_t,Z^1_t)_{t\in [t_0,T_m]}$ is $\mathcal{F}-$progressively adapted.
    \item[-] $(q_t,\varphi_t,Z^2_t)_{t\in [t_0,T_m]}$ is $\mathcal{F}^0-$ progressively adapted. 
    \item[-] the processes are such that 
    \[\esp{\sup_{s\in [t_0,T]} |X_s|^2+\sup_{s\in [t_0,T_m]} |q_s|^2+\sup_{s\in [t_0,T]} |\varphi_s|^2+\sup_{s\in [t_0,T_m]} |U_s|^2+\int_{t_0}^T |(Z^1_s,Z^2_s)|^2 ds}<+\infty.\] 
    \item[-] The equation \eqref{fbsde conditional mckean vlasov} holds almost surely for any $t\in [t_0,T_m]$.
\end{enumerate}
\end{definition}

let us end this introduction with the notion of Hilbertian lifting for functions of measure introduced by Lions in \cite{Lions-college}.
\begin{definition}
    For a function $F:\reels^d\times \mathcal{P}_2(\reels^d)\to \reels^d$, we define its lift on $\mathcal{H}$, $\tilde{F}$ to be 
    \[\tilde{F}:\left\{
\begin{array}{l}
     \mathcal{H}\to \mathcal{H},  \\
     X\mapsto F(X,\mathcal{L}(X)).
\end{array}
\right.
\]
\end{definition}
Under this lifting, the space of functions $C(\reels^d\times \mathcal{P}_2(\reels^d),\reels^d)$ can be endowed with the natural notion of monotonicity on the Hilbert space $\mathcal{H}^d$
\begin{definition}
    A function $F:\reels^d\times \mathcal{P}_2(\reels^d)\to \reels^d$ is said to be $L^2-$monotone if 
    \[\forall X,Y\in \mathcal{H}^d, \quad \langle\tilde{F}(X)-\tilde{F}(Y),X-Y\rangle \geq 0.\]
\end{definition} 
This notion of monotonicity has been widely used in MFGs, ever since Pierre-Louis Lions noticed that it was a natural condition for the wellposedness of non linear transport systems on Hilbert spaces and associated extended mean field games \cite{Lions-college}. For standard MFGs, it often takes the following form 
\begin{definition}
    A function $f:\reels^d\times\mathcal{P}_2(\reels^d)\to \reels$ is said to be displacement monotone if and only if its gradient in the space variable $\nabla_x f$ is well defined and $L^2-$monotone.
\end{definition} 
\section{Lipschitz solutions}
\label{section: MFGMP}
In this section, we present the notion of solution used in this paper to describe MFGs with a major player \cite{MFG-MAJOR-LIONS}. Although this article is concerned with solutions to the major-minor system of master equations, solutions are not defined through the system of PDEs but instead through their stochastic characteristics. This allows us to present a non smooth notion of solution and weaken the regularity assumptions needed, following the approach of \cite{lipschitz-sol}. Formally, we show that as soon as the value function of the major player is Lipschitz and minor players generate Lipschitz feedback controls, there exists a unique and well-defined Nash equilibrium for the associated MFG. 

We first consider the following forward backward system
\begin{equation}
\label{mfg with major}
\left\{
\begin{array}{l}
     \displaystyle X_t=X_0-\int_0^t  F(X_s,q_s,U_s,\varphi_s,Z^\varphi_s,\mathcal{L}(X_s,U_s|\mathcal{F}^0_s))ds+\sqrt{2\sigma}B_t, \\
     \displaystyle q_t=q_0-\int_0^t H_z(q_s,\varphi_s,Z^\varphi_s,\mathcal{L}(X_s,U_s|\mathcal{F}^0_s))ds+\sqrt{2\sigma^0}W^0_t,\\
    \displaystyle  U_t= U_T+\int_t^T G(X_s,q_s,U_s,\varphi_s,Z^\varphi_s,\mathcal{L}(X_s,U_s|\mathcal{F}^0_s))ds-\int_t^T
    Z_sd(B,W^0)_s,\\
\displaystyle \varphi_t= \varphi_T+\int_t^T L_H(q_s,\varphi_s,Z^\varphi_s,\mathcal{L}(X_s,U_s|\mathcal{F}^0_s))ds-\sqrt{2\sigma^0}\int_t^T Z^\varphi_sdW^0_s,\\
U_T=g(X_T,q_T,\mathcal{L}(X_T|\mathcal{F}^0_T)), \quad \varphi_T=\psi(q_T,\mathcal{L}(X_T|\mathcal{F}^0_T)),
\end{array}
\right.
\end{equation}
This is a slightly more general version of the major-minor system of interest. Whenever $F,G$ are the gradients of some Hamiltonian and $L_H,H_z$ denote the functions defined by 
\begin{gather*}\forall (q,\varphi,z,\mu)\in \reels^{2d^0+1}\times \mathcal{P}_2(\reels^{2d}), \\
     L_H(q,\varphi,z,\mu)=H(q,\varphi,z,\mu)-D_zH(q,\varphi,z,\mu)\cdot z,\\ 
      H_z(q,\varphi,z,\mu)=D_zH(q,\varphi,z,\mu),
\end{gather*}
for 
\[(q,\varphi,z,m)\mapsto H(q,\varphi,z,m),\]
 the Hamiltonian of the major player, this system is exactly the MFG system with a major player. In this setting, $(X_t,U_t)_{t\in [0,T]}$ corresponds to the forward backward problem associated to minor players while $(q_t)_{t\in [0,T]}$ models the state of the major player and $(\varphi_t)_{t\in [0,T]}$ its value through time. Allowing general coefficients $(F,G)$ instead of restricting them to the gradient of a Hamiltonian allows us to consider a larger class of problems, including MFGs of controls but also problems outside the consideration of MFGs. 
\begin{definition}
We say that a couple $(U,\varphi):[0,T]\times \reels^d\times \reels^{d^0}\times \mathcal{P}_2(\reels^d)$ is a solution to \linebreak MFGMP$(T,\psi,g,H_z,L_H,F,G)$ iff, for any initial condition $(X_0,q_0)\in \mathcal{H}^{d+d^0}$, there exists a strong solution $(X_t,q_t,U_t,\varphi_t,Z_t,Z^\varphi_t)_{t\in [0,T]}$ to \eqref{mfg with major} such that 
\[\forall t\in [0,T], \quad \varphi_t=\varphi(T-t,q_t,\mathcal{L}(X_t|\mathcal{F}^0_t)),\quad  U_t=U(T-t,X_t,q_t,\mathcal{L}(X_t|\mathcal{F}^0_t)).\]
\end{definition}
In particular, if there exists a smooth solution to the master system associated to the major-minor problem \cite{MFG-MAJOR-LIONS}, then we expect that it to be the decoupling field associated to \eqref{mfg with major}, and hence a solution to MFGMP$(T,\psi,g,H_z,L_H,F,G)$.
\subsection{Notion of solution}
We now present in full detail the notion of solution we use in this article for the system \eqref{eq: master system} and its main properties. Although this notion relies on the forward backward system \eqref{mfg with major}, our approach is not purely probabilistic. Namely we show that there exists a decoupling field $(U,\varphi)$ solving MFGMP$(T,\psi,g,H_z,L_h,F,G)$. This allows us to avoid requiring any differentiability with respect to probability measures for our solution $(U,\varphi)$ which would be difficult to prove in our setting. The fact that the volatility associated to the common noise $\sigma^0$ is non-degenerate is essential not only for our global existence, but also for the very notion of solution we present. Since this allows us to work with merely Lipschitz coefficients (regardless of the magnitude of $\sigma^0>0$), we believe that this constitutes a regularizing effect from the common noise. 
\begin{definition}
    Let $T_c\leq T$, a couple $(U,\varphi): [0,T_c]\times \reels^d\times \reels^{d^0}\times \mathcal{P}_2(\reels^d)\to \reels^d\times \reels$ is said to be a Lipschitz solution to MFGMP$(T,\psi,g,H_z,L,F,G)$ on $[0,T_c]$ iff, 
\begin{enumerate}
    \item[-] $(U(t,\cdot),\varphi(t,\cdot))$ are Lipschitz uniformly in $t\in [0,T_c]$. 
    \item[-]  for any $t\in [0,T_c]$, for any admissible initial condition $(X_0,q_0)$, there exists a strong solution $(X_s,q_s,U_s,\varphi_s,Z_s,Z^\varphi_s)_{s\in [T-t,T]}$ to
\begin{equation}
    \label{eq: def lip sol system}
\left\{
\begin{array}{l}
     \displaystyle X_u=X_0-\int_{T-t}^u  F(X_s,q_s,U_s,\varphi_s,Z^\varphi_s,\mathcal{L}(X_s,U_s|\mathcal{F}^0_s))ds+\sqrt{2\sigma}\int_{T-t}^udB_{s}, \\
     \displaystyle q_u=q_0-\int_{T-t}^u H_z(q_s,\varphi_s,Z^\varphi_s,\mathcal{L}(X_s,U_s|\mathcal{F}^0_s))ds+\sqrt{2\sigma^0}\int_{T-t}^udW^0_{s},\\
    \displaystyle  U_u= U_T+\int_{u}^T G(X_s,q_s,U_s,\varphi_s,Z^\varphi_s,\mathcal{L}(X_s,U_s|\mathcal{F}^0_s))ds-\int_u^T
    Z_sd(B,W^0)_{s},\\
\displaystyle \varphi_u= \varphi_T+\int_u^T L(q_s,\varphi_s,Z^\varphi_s,\mathcal{L}(X_s,U_s|\mathcal{F}^0_s))ds-\sqrt{2\sigma^0}\int_0^t Z^\varphi_sdW^0_{s},\\
U_T=g(X_T,q_T,\mathcal{L}(X_T|\mathcal{F}^0_T)), \quad \varphi_T=\psi(q_T,\mathcal{L}(X_T|\mathcal{F}^0_T)),
\end{array}
\right.
\end{equation}
such that
    \[\exists C>0, \quad |Z^\varphi_s|<C, a.e \text{ for } s\in [T-t,T]\quad  a.s.\]
and the following representation holds 
    \[\forall s\in [T-t,T],\quad \varphi(T-s,q_{s},\mathcal{L}(X_{s}|\mathcal{F}^0_s))=\varphi_{s}, U(T-s,X_s,q_s,\mathcal{L}(X_s|\mathcal{F}^0_s))=U_{s}, \quad a.s.\]
Moreover, for any $x\in \reels^d$ there exists a strong solution $(X^x_s,U^x_s,Z^x_s)_{s\in [T-t,T]}$  to
\begin{equation}
    \label{eq: caracteristics in x fbsde with major}
    \left\{\begin{array}{l}
    \displaystyle X^x_u= x-\int_{T-t}^u F(X^x_s,q_s,U^x_s,\varphi_s,Z^\varphi_s,\mathcal{L}(X_s,U_s|\mathcal{F}^0_s))ds+\sqrt{2\sigma}\int_u^TdB_{s},\\
\displaystyle U^{x}_u=U^x_T+\int_u^T G(X^x_s,q_s,U^x_s,\varphi_s,Z^\varphi_s,\mathcal{L}(X_s,U_s|\mathcal{F}^0_s))-\int_0^t Z^x_s d(B,W^0)_{s},\\
U^x_T=g(X^x_T,q_T,\mathcal{L}(X_T|\mathcal{F}^0_T)),
\end{array}\right.\end{equation}
and the following holds 
\[U^x_s=U(T-s,X^x_s,\mathcal{L}(X_s)), a.s.\]
\end{enumerate}
Finally, $(U,\varphi): [0,T_c)\times \reels^d\times \reels^{d^0}\times \mathcal{P}_2(\reels^d)\to \reels^d\times \reels$ is said to be a Lipschitz solution to \linebreak MFGMP$(T,\psi,g,H_z,L,F,G)$ on $[0,T_c)$ if it is a Lipschitz solution on $[0,t]$ for any $t<T_c$.
\end{definition}
\begin{remarque}
Let us emphasize that $x\in \reels^d$ does not necessarily belong to the support of $X_0$. As such, the additional requirement on the finite dimensional characteristics \eqref{eq: caracteristics in x fbsde with major} appears necessary to ensure that, for a given measure $\mu\in \mathcal{P}_2(\reels^d)$, the function $U(\cdot,\mu)$ is well defined, not just as a function on $[0,T]\times \text{Supp}(\mu)\times \reels^{d^0}$ but on the whole space $\reels^d$.
\end{remarque}
The fact that Lipschitz solutions do not have to be defined on the entire interval $[0,T]$ is quite important. As is standard in the theory of FBSDEs \cite{probabilistic-mfg}, we are going to construct a solution iteratively over small time intervals. To that end we first prove the following lemma
\begin{lemma}
\label{lemma: bootstraping lip sol}
Let $T_c\leq T$ and $(U,\varphi):[0,T_c)\times \reels^d\times \reels^{d^0}\times\mathcal{P}_2(\reels^d)\to \reels^d\times \reels$ be a couple of Lipschitz functions uniformly on $[0,t]$ for any $t<T_c$. Let $\tau\in (0,T_c)$, then $(U,\varphi)$ is a Lipschitz solution to MFGMP$(T,\psi,g,H_z,L,F,G)$ on $[0,T_c)$ if and only if the restriction of $(U,\varphi)$ to $[0,\tau]$ is a Lipschitz solution to MFGMP$(T,\psi,g,H_z,L,F,G)$ on $[0,\tau]$, and the function
\[(U,\varphi)|_{[\tau,T_c]}:(t,x,q,\mu)\mapsto (U,\varphi)(\tau+t,x,q,\mu),\]
is a Lipschitz solution to MFGMP$(T-\tau,\varphi(\tau,\cdot),U(\tau,\cdot),H_z,L,F,G)$ on $[0,T_c-\tau)$.
\end{lemma}
\begin{proof}
    By definition of Lipschitz solutions, it is straightforward that if $(U,\varphi)$ is a Lipschitz solution its restriction to $[\tau,T_c)$ for any $\tau\in (0,T_c)$, $(U,\varphi)|_{[\tau,T_c]}$ is a Lipschitz solution to MFGMP$(T-\tau,\varphi(\tau,\cdot),U(\tau,\cdot),H_z,L,F,G)$. Thus, the first implication follows naturally. 
    
    On the other hand, let us assume that for some $\tau\in (0,T_c)$, $(U,\varphi)$ is a Lipschitz solution to MFGMP$(\psi,g,H_z,L,F,G)$ on $[0,\tau]$ and $(U,\varphi)|_{[\tau,T_c]}$ is a Lipschitz solution to \linebreak MFGMP$(T-\tau,\varphi(\tau,\cdot),U(\tau,\cdot),H_z,L,F,G)$ on $[0,T_c-\tau]$. Let us fix $t\in (\tau,T_c)$,and  an admissible initial condition $(X_0,q_0)$. Using the fact that  $(U,\varphi)|_{[\tau,T_c]}$ is a Lipschitz solution, we can construct a strong solution to 
   \begin{equation*}
\left\{
\begin{array}{l}
     \displaystyle X_{u}=X_0-\int_{T-t}^{u}  F(X_s,q_s,U_s,\varphi_s,Z^\varphi_s,\mathcal{L}(X_s,U_s|\mathcal{F}^0_s))ds+\sqrt{2\sigma}B_u, \\
     \displaystyle q_{u}=q_0-\int_{T-t}^{u} H_z(q_s,\varphi_s,Z^\varphi_s,\mathcal{L}(X_s,U_s|\mathcal{F}^0_s))ds+\sqrt{2\sigma^0}W^0_u,\\
    \displaystyle  U_u= U(\tau,X_{T-\tau},q_{T-\tau},\mathcal{L}(X_{T-\tau}|\mathcal{F}^0_{T-\tau}))+\int_u^{T-\tau} G(X_s,q_s,U_s,\varphi_s,Z^\varphi_s,\mathcal{L}(X_s,U_s|\mathcal{F}^0_s))ds-\int_u^{T-\tau}
    Z_sd(B,W^0)_{s},\\
\displaystyle \varphi_u= \varphi(\tau, q_{T-\tau},\mathcal{L}(X_{T-\tau}|\mathcal{F}^0_{T-\tau}))+\int_u^{T-\tau} L(q_s,\varphi_s,Z^\varphi_s,\mathcal{L}(X_s,U_s|\mathcal{F}^0_s))ds-\sqrt{2\sigma^0}\int_u^{T-\tau} Z^\varphi_sdW^0_{s},
\end{array}
\right.
\end{equation*}
Using the fact that $(U,\varphi)$ is a Lipschitz solution to MFGMP$(T,\psi,g,H_z,L,F,G)$ on $[0,\tau]$, we can build a strong solution $(\tilde{X}_s,\tilde{q}_s,\tilde{\varphi}_s,\tilde{U}_s,\tilde{Z}_s,\tilde{Z^\varphi}_s)_{s\in [T-\tau,T]}$ to \eqref{eq: def lip sol system} with initial condition $(X_{T-\tau},q_{T-\tau})$. Indeed the admissibility of $(X_{T-\tau},q_{T-\tau})$ for \eqref{eq: def lip sol system} follows easily from the admissibility of $(X_0,q_0)$ and the definition of Lipschitz solutions. Letting 
\[\forall s\in [T-t,T], \quad \bar{X}_s=\left\{
\begin{array}{c}
    X_s \text{ if } s< T-\tau,\\
    \tilde{X}_{s} \text{ if } s\geq  T-\tau
\end{array}    
\right. ,\]
we define $(\bar{q}_s,\bar{\varphi}_s,\bar{U}_s,\bar{Z}_s,\bar{Z}^\varphi_s)_{s\in [T-t,T]}$ in a similar fashion. We now check easily that this is a strong solution to \eqref{eq: def lip sol system} on $[T-t,T]$ with initial condition $(X_0,q_0)$. for any $s\in [T-t,T]$, The compatibility condition 
\[(\bar{U}_{s^-},\bar{\varphi}_{s^-})=(\bar{U}_s,\bar{\varphi}_s) \quad a.s,\]
is satisfied since it is satisfied by $(\bar{X}_s,\bar{q}_s)$ and 
\[(\bar{U}_s,\bar{\varphi}_{s})=(U,\varphi)(T-s,\bar{X}_s,\bar{q}_s,\mathcal{L}(\bar{X}_s|\mathcal{F}^0_s)).\]
Naturally, the definition of Lipschitz solution is also satisfied. Since this is true for any initial condition and for any $t\in (\tau,T_c)$, $(U,\varphi)$ is indeed a Lipschitz solution to MFGMP$(T,\psi,g,H_z,L,F,G)$ on $[0,T_c)$.
\end{proof}
At this point we have not made any assumptions on the coefficients of the problem, as the existence of a strong solution to \eqref{eq: def lip sol system} is assumed directly in the definition of a Lipschitz solution. Obviously we are going to need appropriate regularity assumptions to prove that such a solution exists. 
\begin{hyp}
    \label{hyp: locally lipschitz in z}
    \begin{enumerate}
        \item[-] The volatility associated to common noise is such that $\sigma^0>0$
    \item[-] There exists a constant $C_{coef}>0$ such that
    \begin{enumerate}
        \item[-] $\psi,g$ are Lispchitz with 
        \[\|\psi\|_{Lip},\|g\|_{Lip}\leq C_{coef}.\]
        \item[-] There exists an increasing function $\omega:\reels\to \reels^+$ such that for $f=F,G,L,H_z$ the following holds
        \begin{gather*}
            \forall (x,q,u,z,\mu),(y,q',v,z',\nu)\in \reels^{2(d+d^0)}\times \mathcal{P}_2(\reels^{2d}),\\
            |f(x,q,u,z,\mu)-f(y,q',v,z',\nu)|\leq (C_{coef}+\omega(|z|\wedge |z'|))\left(|x-y|+|q-q'|+|u-v|+|z-z'|+\mathcal{W}_2(\mu,\nu)\right).
        \end{gather*}
    \end{enumerate}
\end{enumerate}
\end{hyp}
This assumption is quite natural for the coefficients of MFGs with a major player. Indeed, the Lagrangian of the major player should, at least, be allowed to have quadratic growth in $z$ to handle quadratic penalization on controls. However there is an additional difficulty arising from considering a backward SDE with quadratic growth or higher. For FBSDEs without a mean field component, a standard way of dealing with this nonlinearity is to use Malliavin calculus \cite{tangpi} to show that there exists a constant $M>0$ such that 
\[ \forall t\in [0,T], \quad |Z^\varphi_t|\leq M\quad a.s.\]
However due to structure of our problem this approach is not particularly suitable. Instead we maintain the idea of bounding $(Z^\varphi_t)_{t\in [0,T]}$ uniformly but we rely on the decoupling field of the problem. If there exists a smooth solution to the master system \cite{MFG-MAJOR-LIONS}, then we expect that 
\[\forall t\in [0,T], \quad Z^\varphi_t=\nabla_q \varphi(T-t,q_t,\mathcal{L}(X_t|\mathcal{F}^0_t)).\]
In particular, so long as $\varphi$ is Lipschitz then we expect $(Z^\varphi_t)_{t\in [0,T]}$ to be bounded by its Lipschitz constant. We now formalize this idea for merely Lipschitz decoupling fields
\begin{lemma}
    \label{lemma: Z bound by lip}
    Suppose Hypothesis \ref{hyp: locally lipschitz in z} holds and that there exists $(U,\varphi)$ a Lipschitz solution to \linebreak MFGMP$(T,\psi,g,H_z,L,F,G)$ on $[0,T_c]$ for some $T_c\leq T$, then for any $t\in [0,T_c]$ and for any admissible initial condition $(X_0,q_0)$, letting $(X_s,q_s,U_s,\varphi_s,Z_s,Z^\varphi_s)_{s\in [T-t,T]}$ be a strong solution of \eqref{eq: def lip sol system} associated to $(U,\varphi)$ then 
    \begin{enumerate}
        \item[-] $(Z^\varphi_s)_{s\in [T-t,T]}\in L^\infty([T-t,T], L^\infty(\Omega,\reels^{d^0}))$
        \item[-] $|Z^\varphi_s|\leq  \|\nabla_q \varphi(T-s,\cdot)\|_\infty \quad a.e \text{ for } s\in [T-t,T]\quad a.s. $
    \end{enumerate}
\end{lemma}
\begin{proof}
We may assume $T_c=T=t$, since the proof does not change significantly depending on the time interval considered. 
Consider $(A_t)_{t\in [0,T]}\in L^\infty([0,T]L^\infty(\Omega,\reels^{d^0}))$ a bounded $\mathcal{F}-$progressively adapted process on $[0,T]$. By definition of $(\varphi_s)_{s\in [0,T]}$, 
\[\forall t\leq T, \quad \langle \varphi_\cdot,\int_0^\cdot A_s dW^0_s\rangle_t=\sqrt{2\sigma^0}\int_0^t Z^\varphi_s A_s ds,\]
where $\langle \cdot,\cdot \rangle_t$ indicates the covariation on $[0,t]$. Let us now estimate 
\[I_n=\esp{\sum_{i<n} (\varphi_{t_{i+1}}-\varphi_{t_i})(A_{t_i} (W^0_{t_{i+1}}-W^0_{t_i}))},\]
for a given subdivision $(t_1,\cdots t_n)$ of $[0,T]$.
\begin{align*}
I_n&=\esp{\sum_{i<n} (\varphi(t_{i+1},q_{t_{i+1}},\mathcal{L}(X_{t_{i+1}}|\mathcal{F}_{t_{i+1}}))-\varphi(t_i,q_{t_i},\mathcal{L}(X_{t_i}|\mathcal{F}_{t_i})(A_{t_i} (W^0_{t_{i+1}}-W^0_{t_i}))}\\
&=\esp{\sum_{i<n} \espcond{(\varphi(t_{i+1},q_{t_{i+1}},\mathcal{L}(X_{t_{i+1}}|\mathcal{F}_{t_{i+1}}))-\varphi(t_{i+1},q_{t_i},\mathcal{L}(X_{t_i}|\mathcal{F}_{t_i}))(A_{t_i} (W^0_{t_{i+1}}-W^0_{t_i}))}{\mathcal{F}_{t_i}}},
\end{align*}
where we used the fact that $W^0_{t_{i+1}}-W^0_{t_i}$ is independent of $\mathcal{F}_{t_i}$ and that $(A_{t_i},q_{t_i})$ are $\mathcal{F}_{t_i}-$measurable. Next we observe the following
\[X_{t_{i+1}}=X_{t_i}+(t_{i+1}-t_i)\Delta X_i+\Delta B_i,\]
with 
\[\Delta B_i= \sqrt{2\sigma }(B_{t_{i+1}}-B_{t_i}), \quad \Delta X_i=-\frac{1}{t_{i+1}-t_i}\int_{t_i}^{t_{i+1}}F(X_s,q_s,U_s,Z^\varphi_s,\mathcal{L}(X_s,U_s|\mathcal{F}^0_s))ds.\]
Using the fact that 
\[\mu\mapsto \varphi(t,q,\mu),\]
is Lipschitz uniformly in the other variable with constant $C^\varphi_\mathcal{P}$ we deduce that
\begin{gather*}
    I_n\leq \left|\esp{\sum_{i<n} \espcond{(\varphi(t_{i+1},q_{t_{i+1}},\mathcal{L}(X_{t_{i}}+\Delta B_i|\mathcal{F}_{t_{i+1}}))-\varphi(t_{i+1},q_{t_i},\mathcal{L}(X_{t_i}|\mathcal{F}_{t_i}))(A_{t_i} (W^0_{t_{i+1}}-W^0_{t_i}))}{\mathcal{F}_{t_i}}}\right|\\
    +\underbrace{C^\varphi_\mathcal{P}\esp{\sum_{i<n}(t_{i+1}-t_i)|A_{t_i}|\espcond{|\Delta X_i||W^0_{t_{i+1}}-W^0_{t_i}|}{\mathcal{F}_{t_i}}}}_{J_n}
\end{gather*}
Letting $\mu_{t_i}=\mathcal{L}(X_{t_i}|\mathcal{F}_{t_i})$, the following holds
\[\mathcal{L}(X_{t_i}+\Delta B_i|\mathcal{F}_{t_{i+1}})=\mu_{t_i}\star g^\sigma_{t_{i+1}-t_i},\]
where for $h>0$,$g^\sigma_h$ indicates the density given by
\[g^\sigma_h(x)=\frac{1}{\sqrt{(4\pi \sigma)^d }}\exp\left( -\frac{1}{2\sqrt{2\sigma }}\|x\|^2\right).\]
In particular this shows that $\mathcal{L}(X_{t_i}+\Delta B_i|\mathcal{F}_{t_{i+1}})$ is $\mathcal{F}_{t_i}-$ measurable. Treating this dependency as we did for the time dependency, we deduce that 
\[I_n\leq \esp{\sum_{i<n} \espcond{|\varphi(t_{i+1},q_{t_{i+1}},\mu_{t_i}\star g^\sigma_{t_{i+1}-t_i})-\varphi(t_{i+1},q_{t_i},\mu_{t_i}\star g^\sigma_{t_{i+1}-t_i})||A_{t_i}| |W^0_{t_{i+1}}-W^0_{t_i}|}{\mathcal{F}_{t_i}}}+J_n.\]
Moreover, letting 
\[\|A\|_\infty= \inf \{C>0, \quad \forall s\in [0,T], \quad |A_s|\leq C \quad a.s.\},\]
which is finite by assumption the following holds independently of the subdivision considered
\begin{align*}
J_n&\leq C^\varphi_\mathcal{P}\esp{\sum_{i<n}(t_{i+1}-t_i)^{\frac{3}{2}}|A_{t_i}|\espcond{|\Delta X_i|\frac{|W^0_{t_{i+1}}-W^0_{t_i}|}{\sqrt{t_{i+1}-t_i}}}{\mathcal{F}_{t_i}}},\\
&\leq \|A\|_\infty C^\varphi_\mathcal{P}\sup_i \sqrt{t_{i+1}-t_i}\left(T+\sum_{i<n}(t_{i+1}-t_i)\esp{|\Delta X_i|^2}\right).\\
\end{align*}
It remains to estimate 
\[\esp{|\Delta X_i|^2}.\]
By definition of Lipschitz solution $(Z^\varphi_s)_{s\in [0,T]}\in L^\infty([0,T],L^\infty(\Omega,\reels^{d^0}))$, letting $\|Z^\varphi_\cdot\|_\infty$ indicates the norm of $(Z^\varphi_s)_{s\in [0,T]}$ in this space, by Hypothesis \ref{hyp: locally lipschitz in z} 
\begin{align*}
\esp{|\Delta X_i|^2}&\leq \frac{1}{(t_{i+1}-t_i)^2}\esp{\left(\int_{t_i}^{t_{i+1}}\left(C_{coef}+\omega(|Z^\varphi_s|)\right)\left(1+|X_s|+|q_s|+|U_s|+\espcond{\sqrt{|X_s|^2}+\sqrt{|U_s|^2}}{\mathcal{F}^0_s}\right)ds\right)^2}\\
&\leq \frac{\left(C_{coef}+\omega(\|Z^\varphi_\cdot\|_\infty)\right)^2}{(t_{i+1}-t_i)^2}\esp{\left(\int_{t_i}^{t_{i+1}}\left(1+|X_s|+|q_s|+|U_s|+\espcond{\sqrt{|X_s|^2}+\sqrt{|U_s|^2}}{\mathcal{F}^0_s}\right)ds\right)^2}\\
&\leq \left(C_{coef}+\omega(\|Z^\varphi_\cdot\|_\infty)\right)^2\esp{C(d,d^0)+\sup_{[0,T]}|X_s|^2+\sup_{[0,T]}|U_s|^2+\sup_{[0,T]}|q_s|^2},
\end{align*}
where $C(d,d^0)$ is a constant depending on $d$ and $d^0$ only. 

As a consequence, there exists a constant depending on the data of the problem such that but independent of the subdivision such that
\[|J_n|\leq C_{data}\sup_i\sqrt{t_{i+1}-t_i},\]
taking first the supremum on subdivision of $[0,T]$ with $n$ elements, and then the limit as $n$ tends to infinity we deduce that for any bounded $\mathcal{F}-$adapted process $(A_s)_{s\in [0,T]}$,
\[\esp{\int_0^T Z^\varphi_t \cdot A_t dt}\leq \esp{\int_0^T \|\nabla_q\varphi(T-s,\cdot)\|_\infty |A_s|ds}.\]
Letting 
\[\forall t\in [0,T], \for 1\leq i \leq d \quad A^i_t=\xi_t (\mathds{1}_{Z^\varphi_t\geq 0}-\mathds{1}_{Z^\varphi_t\leq0})\]
for some $(\xi_t)_{t\in [0,T]}\in \mathcal{H}^T_\mathcal{F}\cap L^\infty([0,T],L^\infty(\Omega,\reels^{d^0}))$, we get 
\[\esp{\int_0^T |Z^\varphi_t| \xi_t dt}\leq \esp{\int_0^T \|\nabla_q\varphi(T-s,\cdot)\|_\infty |\xi_s|ds}.\]
Taking $\xi_t=\mathds{1}_{|Z^\varphi_t|\geq \|\nabla_q\varphi(T-t,\cdot)\|_\infty }$ we deduce that 
\[\esp{\int_0^T (|Z^\varphi_t|-\|\nabla_q\varphi(T-t,\cdot)\|_\infty )\mathds{1}_{|Z^\varphi_t|\geq \|\nabla_q\varphi(T-t,\cdot)\|_\infty }dt}\leq 0,\]
which implies that 
\[\int_0^T (|Z^\varphi_t|-\|\nabla_q\varphi(T-t,\cdot)\|_\infty )\mathds{1}_{|Z^\varphi_t|\geq \|\nabla_q\varphi(T-t,\cdot)\|_\infty }dt=0 \quad a.s.\]
Finally we deduce that 
\[|Z^\varphi_t|\leq \|\nabla_q\varphi(T-t,\cdot)\|_\infty \quad a.e \text{ in } [0,T], \quad a.s.\]
\end{proof}
\begin{remarque}
Although it requires no differentiability of the coefficients, let us remark that our result is strictly weaker than the analogue obtained via Malliavin calculus since $(Z^\varphi_t)_{t\in [0,T]}$ is only bounded in $L^\infty(\Omega,L^\infty([0,T],\reels^{d^0}))$ and we do not have pointwise estimates on $Z^\varphi_t$ for a fixed $t\in [0,T]$.Nevertheless, we believe this is the optimal regularity obtainable under our assumptions. Indeed since $\varphi$ is only Lipschitz its gradient might no be defined at every point, hence the almost everywhere nature of the result. 
\end{remarque}
We now adapt a standard result from the theory of Markovian FBSDEs to our setting, namely the existence of a Lipschitz decoupling field to \eqref{eq: def lip sol system} implies the uniqueness of strong solutions
\begin{thm}
\label{lemma: uniqueness lip sol}
Under Hypothesis \ref{hyp: locally lipschitz in z}, if there exists a Lipschitz solution $(U,\varphi)$ to MFGMP$(T,\psi,g,H_z,L,F,G)$ on $[0,\tau]$ for some $\tau>0$ then there exists a unique strong solution $(X_s,q_s,U_s,\varphi_s,Z_s,Z^\varphi_s)_{s\in [T-t,T]}$ to \eqref{eq: def lip sol system}
satisfying 
    \[\exists C>0, \quad |Z^\varphi_s|<C, \quad a.e \text{ for }s\in [T-t,T] \quad  a.s.\]
for any $t\leq \tau$ and admissible initial condition $(X_0,q_0)$, moreover if all coefficients are Lipschitz then it is sufficient to require 
\[\esp{\int_{T-t}^T|Z^\varphi_s|^2ds}<+\infty,\]
instead of a bound in $L^\infty(\Omega,L^\infty([T-t,T],\reels^{d^0}))$. In particular, there can exist at most one Lipschitz solution on $[0,t]$ for any $t\leq T$. 
\end{thm}
\begin{proof}
Since the proof does not change depending on the time interval considered, we assume without loss of generality that $t=\tau=T$. We also assume in a first step that all coefficients and $(U,\varphi)$ are Lipschitz with a contant $c_l$. Let us consider an initial condition $(X_0,q_0)$ and assume that there exists two solutions $(X^1_s,q^1_s,U^1_s,\varphi^1_s,Z^1_s,Z^{\varphi,1}_s)_{s\in [0,T]}$ given by the Lipschitz solution and $(\tilde{X}_s,\tilde{q}_s,\tilde{U}_s,\tilde{\varphi}_s,\tilde{Z}_s,\tilde{Z}^\varphi_s)$. Since all coefficients are Lipschitz, by a standard contraction argument \cite{probabilistic-mfg} Theorem 5.4, there exists a time $t^*$ depending on $c_l$ only such that for $t\leq t^*$ there exists a unique solution to \eqref{eq: def lip sol system}. Let us consider a solution $(\bar{X}_s,\bar{q}_s,\bar{U}_s,\bar{\varphi}_s,\bar{Z}^\varphi_s,\bar{Z}_s)_{s\in [T,T-t^*]}$ of \eqref{eq: def lip sol system} with initial condition $\tilde{X}_{T-t^*},\tilde{q}_{T-t^*}$ and such that 
\[\forall s\in [T-t^*,T], \quad \bar{U}_s=U(T-s,\bar{X}_s,\bar{q}_s,\mathcal{L}(\bar{X}_s|\mathcal{F}^0_s)), \quad \bar{\varphi}_s=\varphi(T-s,\bar{q}_s,\mathcal{L}(\bar{X}_s|\mathcal{F}^0_s)).\]
The existence of such solution follows from the definition of a Lipschitz solution. On the other hand since it satisfies the same FBSDE as $(\tilde{X}_s,\tilde{q}_s,\tilde{U}_s,\tilde{\varphi}_s,\tilde{Z}_s,\tilde{Z}^\varphi_s)_{s\in [T-t^*,T]}$, we deduce that
\[\forall s\in [T-t^*,T], \quad \tilde{U}_s=U(T-s,\tilde{X}_s,\tilde{q}_s,\mathcal{L}(\tilde{X}_s|\mathcal{F}^0_s)), \quad \tilde{\varphi}_s=\varphi(T-s,\tilde{q}_s,\mathcal{L}(\tilde{X}_s|\mathcal{F}^0_s))\quad a.s.\]
Following this representation formula, $(\tilde{X}_s,\tilde{q}_s,\tilde{U}_s,\tilde{\varphi}_s,\tilde{Z}_s,\tilde{Z}^\varphi_s)_{s\in [T-2t^*,T-t^*]}$ is a solution to 
\begin{equation}
\left\{
\begin{array}{l}
     \displaystyle X_{u}=\tilde{X}_{T-2t^*}-\int_{T-2t^*}^{u}  F(X_s,q_s,U_s,\varphi_s,Z^\varphi_s,\mathcal{L}(X_s,U_s|\mathcal{F}^0_s))ds+\sqrt{2\sigma}\int_{T-2t^*}^{u}dB_{s}, \\[0.5em]
     \displaystyle q_u=\tilde{q}_{T-2t^*}-\int_{T-2t^*}^{u} H_z(q_s,\varphi_s,Z^\varphi_s,\mathcal{L}(X_s,U_s|\mathcal{F}^0_s))ds+\sqrt{2\sigma^0}\int_{T-2t^*}^{u}dW^0_{s},\\[0.5em]
    \displaystyle  U_u= U_{T-t^*}+\int_u^{T-t^*} G(X_s,q_s,U_s,\varphi_s,Z^\varphi_s,\mathcal{L}(X_s,U_s|\mathcal{F}^0_s))ds-\int_u^{T-t^*}
    Z_sd(B,W^0)_{s},\\[0.5em]
\displaystyle \varphi_u= \varphi_{T-t^*}+\int_{u}^{T-t^*} L(q_s,\varphi_s,Z^\varphi_s,\mathcal{L}(X_s,U_s|\mathcal{F}^0_s))ds-\sqrt{2\sigma^0}\int_u^{T-t^*} Z^\varphi_sdW^0_{s},\\[0.5em]
U_{T-t^*}=U(T-t^*,X_{T-t^*},q_{T-t^*},\mathcal{L}(X_{T-t^*}|\mathcal{F}^0_{T-t^*})),\\[0.5em]
\varphi_{T-t^*}=\varphi(T-t^*,q_{T-t^*},\mathcal{L}(X_{T-t^*}|\mathcal{F}^0_{T-t^*})).
\end{array}
\right.
\end{equation}
Since all coefficients of this FBSDE, are $c_l$ Lipschitz, there can exist at most one solution on $[T-t^*,T]$ and by repeating this argument across the whole interval $[0,T]$, we deduce that 
\[\forall s\in [0,T], \quad \tilde{U}_s=U(T-s,\tilde{X}_s,\tilde{q}_s,\mathcal{L}(\tilde{X}_s|\mathcal{F}^0_s)), \quad \tilde{\varphi}_s=\varphi(T-s,\tilde{q}_s,\mathcal{L}(\tilde{X}_s|\mathcal{F}^0_s)) \quad a.s.\]
We now use the same argument but forward in time. On $[0,t^*]$, $(X^1_s,q^1_s,U^1_s,\varphi^1_s,Z^1_s,Z^{\varphi,1}_s)_{s\in [0,t^*]}$  $(\tilde{X}_s,\tilde{q}_s,\tilde{U}_s,\tilde{\varphi}_s,\tilde{Z}_s,\tilde{Z}^\varphi_s)_{s\in [0,t^*]}$ solve the same FBSDE and consequently are equal by the small time uniqueness result Theorem 5.4 of \cite{probabilistic-mfg}. The uniqueness on $[0,T]$ follows by induction. 

We now relax the assumption that coefficients are $c_l$ Lipschitz and only assume that Hypothesis \ref{hyp: locally lipschitz in z} is in force. We let $C_{\varphi,U}=\sup_{s\in [0,T]}\|(U,\varphi)(s,\cdot)\|_{Lip}$. Let us assume that for an initial condition $X_0,q_0$ there exists two solutions to \eqref{eq: def lip sol system} $(X^i_s,q^i_s,U^i_s,\varphi^i_s,Z^i_s,Z^{\varphi,i}_s)_{s\in [0,T]}$ $i=1,2$. We may assume that for $i=1$ it is associated to $(U,\varphi)$ without loss of generality. We also assume that 
\[\exists C, \quad |Z^{\varphi,2}_s|\leq C \quad a.e \text{ for } s\in [0,T]\quad a.s. \]
Let $M=C\wedge C_{\varphi,U}$, since for any $s\leq T$, $Z^{\varphi,1}_s,Z^{\varphi,2}_s$ are bounded in $L^\infty(\Omega,L^\infty([0,T],\reels^{d^0}))$ by $M$, $(X^i_s,q^i_s,U^i_s,\varphi^i_s,Z^i_s,Z^{\varphi,i}_s)_{s\in [0,T]}$ is also a solution to \eqref{eq: def lip sol system} by replacing the coefficients with $G,F,H_z,L$ with 
\[(G^M,F^M,H_z^M,L^M):(x,q,u,\varphi,z,\mu)\mapsto (G,F,H_z,L)(x,q,u,\varphi,z\wedge M\vee(-M),\mu).\]
Since $(U,\varphi,G^M,F^M,H_z^M,L^M)$ are Lipschitz uniformly on $[0,T]$ for 
\[c_l=C_{coef}+M+\omega(M),\]
we may apply the previous result for Lipschitz coefficients. 
\end{proof}

Let us now introduce the following norm for a function $f:\reels^d\times\reels^{d^0}\times \reels^d\times \reels^\times \reels^{d^0}\times \mathcal{P}_2(\reels^{2d})\to \reels^k$
\[\|f\|_{\infty,c}=\sup_{\begin{array}{c}
    (x,q,u,\varphi,z,m)\in \reels^d\times\reels^{d^0}\times \reels^d\times \reels\times \reels^{d^0}\times \mathcal{P}_2(\reels^{2d})\\
    \| z\|_\infty \leq c
\end{array}
}
\frac{|f(x,q,u,\varphi,z,m)|}{1+|x|+|q|+|u|+|\varphi|+\sqrt{E_2(m)}}.
    \]
In particular, whenever $f$ does not depend not depend on $z$ this norm is denoted by $\|\cdot \|_{\infty,loc}$.
This norm is convenient to prove the following stability result for Lipschitz solutions
\begin{thm}
    \label{thm: stability lip sol}
Fix $T_c\leq T$ and let $(\varphi_n,U_n)_{n\in \mathbb{N}}:[0,T_c]\times \reels^d\times \reels^{d^0}\times\mathcal{P}_2(\reels^d)\to \reels\times \reels^d$ be a sequence of Lipschitz solutions to MFGMP$(T,\psi^n,g^n,H_z^n, L^n,F^n,G^n)$ on $[0,T_c]$. Suppose that there exists locally bounded functions $(F,H_z,L,G,\psi,g)$ such that 
\[\forall z\geq 0, \quad \lim_{n\to \infty } \|(F,H_z,L,G,\psi,g)-(F^n,H_z^n,L^n,G^n,\psi^n,g^n)\|_{\infty,z}=0.\]
If Hypothesis \ref{hyp: locally lipschitz in z} holds uniformly in $n\in \mathbb{N}$, i.e. there exists a constant $C_{coef}$, and an increasing function $\omega:\reels\to \reels^+$, both independent of $n\in \mathbb{N}$ such that 
\begin{enumerate}
    \item[-]$\forall n\in \mathbb{N}, \forall t\in [0,T_c],\quad  \|U(t,\cdot)\|_{Lip},\|\varphi(t,\cdot)\|_{Lip}\leq C_{coef},$
    \item[-] for $f_n= F^n,H_z^n, L^n,G^n,\psi^n,g^n$
    \begin{gather*}
        \forall n\in \mathbb{N}, \quad \forall (x,q,u,z,\mu),(y,q',v,z',\nu)\in \reels^{2(d+d^0)}\times \mathcal{P}_2(\reels^{2d}),\\
            |f_n(x,q,u,z,\mu)-f_n(y,q',v,z',\nu)|\\
            \leq (C_{coef}+\omega(|z|\wedge |z'|))\left(|x-y|+|q-q'|+|u-v|+|z-z'|+\mathcal{W}_2(\mu,\nu)\right),
    \end{gather*}
\end{enumerate}
then $(U_n,\varphi_n)_{n\in \mathbb{N}}$ converges locally uniformly to $(U,\varphi)$ a Lipschitz solution to MFGMP$(T,\psi,g,H_z,L,F,G)$ on $[0,T_c]$. 
\end{thm}
\begin{proof}
We may assume without loss of generality that $T=T_c$. Let us fix an initial condition $(X_0,q_0)$ and a horizon $t_h\leq T$ and consider the strong solution $(X^n_s,q^n_s,U^n_s,\varphi^n_s,Z^{\varphi,n}_s,Z^n_s)_{s\in [T-t_h,T]}$ of \eqref{eq: def lip sol system} associated to the Lipschitz solution $(U^n,\varphi^n)$. By Lemma \ref{lemma: Z bound by lip}, and thanks to our assumptions 
\begin{equation}
    \label{eq: bound on Z stability proof}
    \forall n\in \mathbb{N}\quad   |Z^{\varphi,n}_s|\leq C_{coef}, a.e\text{ in }[T-t_h,T]\quad  a.s.\end{equation} 
From there, classic Gronwall type estimates allow us to deduce that there exists a constant $C$ depending only on $C_{coef},\omega,T$ such that 
\begin{gather}
\label{eq: bound on process stability proof}
    \forall n \in \mathbb{N}, s,t\in [T-t_h,T],\quad,\\
    \nonumber \|X_s\|,\|q_s\|,\esp{\int_{T-t_h}^{T}|Z_s|^2ds}, \frac{\|X_t-X_s\|}{\sqrt{|t-s|}},\frac{\|q_t-q_s\|}{\sqrt{|t-s|}},\frac{\|\varphi_t-\varphi_s\|}{\sqrt{|t-s|}},\frac{\|U_t-U_s\|}{\sqrt{|t-s|}}\leq C(1+\|X_0\|+|q_0|).
\end{gather}
In particular for $T-t_h\leq s\leq t\leq T$ and $q_0\in \reels^{d^0}$, 
\begin{align*}
|\varphi^n (T-s,q_0,\mathcal{L}(X_0))-\varphi^n(T-t,q_0,\mathcal{L}(X_0))|&=|\varphi^n_{s}-\varphi^n_{t}+\varphi(T-t,q_{t},\mathcal{L}(X_{t}|\mathcal{F}^0_t))-\varphi_n(T-t,q_0,\mathcal{L}(X_0))|\\
&\leq (C+C_{coef})(1+\|X_0\|+|q_0|)\sqrt{t-s}.
\end{align*}
Since this is true for any $(t_h,X_0,q_0)\in [0,T]\times H\times \reels^{d^0}$, and given that a similar computation can be carried out for $U^n$, we deduce that the sequence $(U^n,\varphi^n)_{n\in \mathbb{N}}$ is locally Holder in time, uniformly in $n\in \mathbb{N}$. Moreover since $(\psi^n,g^n)_{n\in \mathbb{N}}$ converges to $(\psi,g)$ for the $\|\cdot \|_{\infty,loc}$ norm, we deduce that there exists a constant $C>0$ such that 
\[\forall n\in \mathbb{N}, \forall t\in [0,T], \quad \|(U^n,\varphi^n)(t,\cdot)\|_{\infty,loc}\leq C.\]
In particular, by Azerla-Ascoli theorem, along a subsequence, $(U^n,\varphi^n)_{n\in \mathbb{N}}$ converges for $\|\cdot\|_{\infty,loc}$ to a limit $(U,\varphi)$. We now work along one such subsequence, still denoted by $(U^n,\varphi^n)_{n\in \mathbb{N}}$. 
Fixing $t_h\in [0,T]$ and an initial condition $(X_0,q_0)$, in light of \eqref{eq: bound on process stability proof} and \eqref{eq: bound on Z stability proof}, and by an application of Ito's lemma, there exists a constant depending only on $\omega,C_{coef}$ and $\sigma^0$, $C_{x,q}$ such that for any $n,m\in \mathbb{N}$
\begin{gather*}
\forall s\in [T-t_h,T], \quad \|X^n_s-X^m_s\|^2+\|q_s^n-q_s^m\|^2\leq\int_{T-T_h}^s C_{x,q}\left(\|X^n_u-X^m_u\|^2+\|q_u^n-q_u^m\|^2\right)du\\
+\frac{\sigma^0}{2C_{coef}}\int_{T-t_h}^s \|Z^{\varphi,n}_u-Z^{\varphi,m}_u\|^2du+C_{x,q}(1+T+\|X_0\|^2+\|q_0\|^2)\|(F^n,H_z^n,U^n,\varphi^n)-(F^m,H_z^m,U^n,\varphi^n)\|^2_{\infty,C_{coef}}.
\end{gather*}
By an application of Gronwall's lemma, there exists a constant $C'_{x,q}$ depending on $C_{coef},\omega,\sigma^0$ and $T$ only such that for any $s\in [T-t_h,T]$
\begin{align*}
 \|X^n_s-X^m_s\|^2+\|q_s^n-q_s^m\|^2&\leq (\frac{\sigma^0}{2C_{Coef}}\!+\!C'_{x,q}t_h)\int_{T-t_h}^{T}\! \|Z^{\varphi,n}_u-Z^{\varphi,m}_u\|^2du\\
 &+C'_{x,q}(1+T+\|X_0\|^2+\|q_0\|^2)\|(F^n,H_z^n,U^n,\varphi^n)-(F^m,H_z^m,U^m,\varphi^m)\|^2_{\infty,C_{coef}}.
\end{align*}
We now turn to an estimate on 
\[\int_{T-t_h}^{T} \|Z^{\varphi,n}_s-Z^{\varphi,m}_s\|^2ds,\]
to that end, we apply Ito's lemma to the backward processes, yielding
\begin{gather*}
    \int_{T-t_h}^{T} \|Z^n_s-Z^m_s\|^2ds+(\sigma^0-Ct)\int_{T-t_h}^{T} \|Z^{\varphi,n}_s-Z^{\varphi,m}_s\|^2ds\leq C\int_{T-t_h}^{T}\left( \|X^n_s-X^m_s\|^2+\|q_s^n-q_s^m\|^2\right)ds\\
    +C(1+T+\|X_0\|^2+\|q_0\|^2)\|(F^n,H_z^n,U^n,\varphi^n)-(F^m,H_z^m,U^m,\varphi^m)\|^2_{\infty,C_{coef}}
\end{gather*}
for a constant $C$ depending on $\omega,C_{coef},\sigma^0$ and $T$. Plugging in the estimates we already have on the difference of the forward processes, we deduce that there exists a time $\tau$ depending only on $T,\omega,C_{coef},\sigma^0$, such that if the horizon $t_h$ is chosen such that $t_h\leq \tau$ then 
\[\int_{T-t_h}^{T} \|(Z^n,Z^{\varphi,n})_s-(Z^m,Z^{\varphi,m})_s\|^2ds\leq  C_\tau(1+T+\|X_0\|^2+\|q_0\|^2)\|(F^n,H_z^n,U^n,\varphi^n)-(F^m,H_z^m,U^m,\varphi^m)\|^2_{\infty,C_{coef}},\]
for a constant $C_\tau$ with the same dependencies as $\tau$. We now make the assumption $t_h<\tau$. Since $(F^n,H_z^n,U^n,\varphi^n)_{n\in \mathbb{N}}$ converges to $(F,H_z,U,\varphi)$ for $\|\cdot \|_{\infty,C_{coef}}$, it follows that $\left((Z^n_s,Z^{\varphi,n}_s)_{s\in [0,t_h]}\right)_{n\in \mathbb{N}}$ is a cauchy sequence in $\mathcal{M}_{(d+d^0)\times d}(\mathcal{H}^{t_h})\times \left(\mathcal{H}^{t_h}\right)^{d^0}$, and by completeness it converges in this space to a random process $(Z_t,Z^\varphi_t)_{t\in [T-t_h,T]}$. Similarly for any $t\in [T-t_h,T]$, $(X^n_t,q^n_t)_{n\in \mathbb{N}}$ is a cauchy sequence in $\mathcal{H}^{d+d^0}$, which converges to a couple $(X_t,q_t)\in \mathcal{H}^{d+d^0}$. We now show that for fixed $t\leq t_h$, $(U^n_t,\varphi^n_t)$ converges in $\mathcal{H}^{d+1}$ to $(U(T-t,X_t,q_t,\mathcal{L}(X_t|\mathcal{F}^0_t)), \varphi(T-t,q_t,\mathcal{L}(X_t)|\mathcal{F}^0_t))$. We show this is true only for $\varphi$, the argument for $U$ following in a similar fashion. By the definition of a Lipschitz solution, we already know that 
\[\forall n\in \mathbb{N}, t\in [T-t_h,T],\quad \varphi^n_t=\varphi^n(T-t,q^n_t,\mathcal{L}(X^n_t|\mathcal{F}^0_t)).\]
Consequently for fixed $n\in \mathbb{N}$ 
\[\|\varphi^n_t-\varphi(T-t,q_t,\mathcal{L}(X_t)|\mathcal{F}^0_t)\| \leq C_{coef}\|(X^n_t,q^n_t)-(X_t,q_t)\|+\underbrace{\| \varphi^n(T-t,q_t,\mathcal{L}(X_t|\mathcal{F}^0_t))-\varphi(T-t,q_t,\mathcal{L}(X_t|\mathcal{F}^0_t))\|}_{\leq \|\varphi-\varphi^n\|_{\infty,loc}(1+\|q_t\|+\|X_t\|)},\]
and the convergence follows from the convergence of $(\varphi^n(t,\cdot))_{n\in \mathbb{N}}$ and $(X^n_t,q^n_t)_{n\in \mathbb{N}}$. Let us remind that by Ito's isometry 
\[\forall t\leq t_h, \quad \left\|\int_{T-t}^T (Z^{\varphi,n}_s-Z^\varphi_s)dW^0_s\right\|^2=\esp{\int_{T-t}^T|Z^{\varphi,n}_s-Z^\varphi_s|^2ds }.\]
Finally using the convergence of all processes as well as the system satisfied by \linebreak $\left((X^n_t, q^n_t,\varphi^n_t,U^n_t,Z^n_t,Z^{\varphi,n}_t)_{t\in [T-T_h,T]}\right)_{n\in \mathbb{N}}$. we deduce that $(X_t, q_t,\varphi_t,U_t,Z_t,Z^{\varphi}_t)_{t\in [T-t_h,T]}$ is a strong solution of the forward backward system 
    \begin{equation}
    \label{eq: sol on 0 tau}
\left\{
\begin{array}{l}
     \displaystyle X_{u}=X_0-\int_{T-t_h}^{u}  F(X_s,q_s,U_s,\varphi_s,Z^\varphi_s,\mathcal{L}(X_s,U_s|\mathcal{F}^0_s))ds+\sqrt{2\sigma}B_u, \\
     \displaystyle q_{u}=q_0-\int_{T-t_h}^{u} H_z(q_s,\varphi_s,Z^\varphi_s,\mathcal{L}(X_s,U_s|\mathcal{F}^0_s))ds+\sqrt{2\sigma^0}W^0_u,\\
    \displaystyle  U_u= g(X_{T},q_{T},\mathcal{L}(X_{T}|\mathcal{F}^0_{T}))+\int_u^{T} G(X_s,q_s,U_s,\varphi_s,Z^\varphi_s,\mathcal{L}(X_s,U_s|\mathcal{F}^0_s))ds-\int_u^{T}
    Z_sd(B,W^0)_s,\\
\displaystyle \varphi_u= \psi(q_{T},\mathcal{L}(X_{T}|\mathcal{F}^0_{T}))+\int_u^{T} L(q_s,\varphi_s,Z^\varphi_s,\mathcal{L}(X_s,U_s|\mathcal{F}^0_s))ds-\sqrt{2\sigma^0}\int_0^t Z^\varphi_sdW^0_s,
\end{array}
\right.
\end{equation}
satisfying 
    \[\exists C>0, \quad |Z^\varphi_s|<C_{coef} \quad a.e \text{ for }s\in [T-t_h,T] \quad  a.s.\]
and that the following representation holds 
    \[\forall s\in [T-t_h,T],\quad \varphi(T-s,q_{s},\mathcal{L}(X_{s}|\mathcal{F}^0_s))=\varphi_{s}, U(T-s,X_s,q_s,\mathcal{L}(X_s|\mathcal{F}^0_s))=U_{s}, \quad a.s.\]
Clearly, this is true for any admissible initial condition $(X_0,q_0)$, and for any time $t_h\leq \tau$. It follows by definition that $(U,\varphi)$ is a Lipschitz solution to MFGMP$(T,\psi,g,H_z,L,F,G)$ on $[0,\tau]$. Let us now take another converging subsequence of $(U^n,\varphi^n)$, by the arguments we used above, it is also a Lipschitz solution on $[0,\tau]$. However by Theorem \ref{lemma: uniqueness lip sol}, there can be at most one Lipschitz solution on $[0,\tau]$. Since all converging subsequences of $(U^n,\varphi^n)_{n\in \mathbb{N}}$ must have the same limit $(U,\varphi)$ on $[0,\tau]$, we deduce that the sequence converges strongly to the unique Lipschitz solution to $MFGMP(T,\psi,g,H_z,L,F,G)$ on this interval. Clearly we can repeat this argument to show that on the interval $[\tau,2\tau]$, $(U^n,\varphi^n)_{n\in \mathbb{N}}$ converges to a Lipschitz solution of MFGMP$(T-\tau,\varphi(\tau,\cdot),U(\tau,\cdot),H_z,L,F,G)$ on $[0,\tau]$. By Lemma \ref{lemma: bootstraping lip sol} it follows that $(U^n,\varphi^n)$ converges on $[0,2\tau]$ to the Lipschitz solution of MFGMP$(T,\psi,g,H_z,L,F,G)$. By a bootstrapping argument, we can extend this result to the whole interval $[0,T]$ concluding the proof.
\end{proof}

We now introduce the following Lipschitz norm on the Hilbert space of square integrable random variable. For a function $U:\reels^d\times \reels^{d^0}\times \mathcal{P}_2(\reels^d)\to \reels$ we denote
\[\|U\|_{Lip,H}=\sup_{\begin{array}{cc}
    (X,Y)\in H^d\\
    (q,q')\in H^{d^0}
\end{array}}
  \frac{\|U(X,q,\mathcal{L}(X))-U(Y,q',\mathcal{L}(Y))\|}{\|(X,q)-(Y,q')\|}.\]
In particular, if $U$ is $C_x$ Lipschitz in $\reels^d$, $C_q$ Lipschitz in $\reels^{d^0}$ and $C_\mu$ Lipschitz on $\mathcal{P}_2(\reels^d)$, then 
\[\|U\|_{Lip,H}\leq C_x+C_q+C_\mu.\]
\

For Lipschitz solutions, we have the following local in time a priori estimate
\begin{lemma}
\label{lemma: local lipschitz estimate}
Suppose that there exists a Lipschitz solution to MFGMP$(T,\psi,g,H_z,L,F,G)$ on $[0,T_c]$ for some $T_c\leq T$. Under Hypothesis \ref{hyp: locally lipschitz in z}, there exists a time $t^*$ and a constant $C^*$ both depending only on $C_{coef},\omega$ and $\sigma^0$ such that 
\[\forall t< t^*\wedge T_c, \quad \|(U,\varphi)(t,\cdot)\|_{Lip}\leq C^*.\]
\end{lemma}
\begin{proof}
Let us first introduce the notation 
\[\forall t<T_c, \quad u(t)=\sup_{s\in [0,t]}\|(U,\varphi)(s,\cdot)\|_{Lip,H}.\]
we now fix $t\in [0,T_c)$ and consider two solutions $(X_s,q_s,U_s,\varphi_s,Z_s,Z^\varphi_s)_{s\in [T-t,T]}$, $(\tilde{X}_s,\tilde{q}_s,\tilde{U}_s,\tilde{\varphi}_s,\tilde{Z}_s,\tilde{Z}^\varphi_s)_{s\in [T-t,T]}$ to \eqref{eq: def lip sol system}, respectively with initial condition $(X_0,q_0)$ and $\tilde{X}_0,\tilde{q}_0$. 

Standard Gronwall type estimates yield 
\begin{gather*}\forall s\in [T-t,T],\\
     \|(X_s,q_s)-(\tilde{X}_s,\tilde{q}_s)\|^2\\
    \leq \left(\|(X_0,q_0)-(\tilde{X}_0,\tilde{q}_0)\|^2+\frac{\sigma^0}{2C_{coef}}\int_{T-t}^s \|Z^\varphi_u-\tilde{Z}^\varphi_u\|^2du \right)e^{ C_1(t+\int_{T-t}^s \omega^2(u(t-r))dr)},
\end{gather*}
for a constant $C_1$ depending on $\sigma^0$ and $C_{coef}$ only.
On the other hand, by Ito's Lemma we deduce that 
\begin{gather*}\|(U,\varphi)(T-t,X_0,q_0,\mathcal{L}(X_0))-(U,\varphi)(T-t,\tilde{X}_0,\tilde{q}_0,\mathcal{L}(\tilde{X}_0))\|^2+\sigma^0\int_{T-t}^T \|Z^\varphi_s-\tilde{Z}^\varphi_s\|^2ds\\
\leq C_{coef}\|(X_t,q_t)-(\tilde{X}_t,\tilde{q}_t)\|^2+C_2\int_{T-t}^T (1+\omega^2(u(t-s)))\|(X_s,q_s)-(\tilde{X}_s,\tilde{q}_s)\|^2ds,
\end{gather*}
for a constant $C_2$ with the same dependencies as $C_1$. Letting 
\[t_1=\sup \left\{t\in [0,T_c], \int_0^t \omega^2(u(s))ds\leq 1\right\},\]
By assumption on $(U,\varphi)$, we already know that $t_1>0$. Letting $T_n=\frac{nT_c}{n+1}$ for some $n\in \mathbb{N}, n\geq 1$, we deduce that for all $t\leq t_1\wedge T_n$ there exists a constant $C$ depending on $\sigma^0,C_{coef}$ only such that
\begin{gather*}\|(U,\varphi)(t,X_0,q_0,\mathcal{L}(X_0))-(U,\varphi)(t,\tilde{X}_0,\tilde{q}_0,\mathcal{L}(\tilde{X}_0))\|^2+(\sigma^0-Ct)\int_{T-t}^T \|Z^\varphi_s-\tilde{Z}^\varphi_s\|^2ds\\
\leq (C_{coef}+Ct)\left(\|(X_0,q_0)-(\tilde{X}_0,\tilde{q}_0)\|^2\right).
\end{gather*}

It follows that for $t\leq t_2=t_1\wedge T_n\wedge \frac{\sigma^0}{C}$
\[\frac{\|(U,\varphi)(t,X_0,q_0,\mathcal{L}(X_0))-(U,\varphi)(t,\tilde{X}_0,\tilde{q}_0,\mathcal{L}(\tilde{X}_0))\|^2}{\|(X_0,q_0)-(\tilde{X}_0,\tilde{q}_0)\|^2}\leq (C_{coef}+Ct).\]
Since this is true for any initial condition, we deduce that for $t\leq t_2$
\[\|(U,\varphi(t,\cdot))\|_{Lip,H}\leq C_{coef}+Ct.\]
Let us fix 
\[t_3=\sup \left\{t\in [0,T_c], \int_0^t \omega^2(C_{coef}+Cs)ds\leq 1\right\}, \]
clearly $t_3\leq t_1$ and it is independent of $\left(\|(U,\varphi(t,\cdot))\|_{Lip}\right)_{t\in [0,T_c)}$. Finally letting $t^*_n=t_3\wedge \frac{\sigma^0}{2C}\wedge T_n$, we get
\[\forall t\leq t^*_n,\quad \|(U,\varphi(t,\cdot))\|_{Lip}\leq C_{coef}+Ct,\]
and that for any $t\leq t^*_n$, and for any initial condition $(X_0,q_0), (\tilde{X}_0,\tilde{q_0})$
\[\frac{\sigma^0}{2} \int_{T-t}^T\|Z^\varphi_s-\tilde{Z}^{\varphi}_s\|^2ds\leq (C_{coef}+Ct)\|(X_0,q_0)-(\tilde{X}_0,\tilde{q}_0)\|^2.\]
In particular, by a direct Corollary of Lemma 2.3 in \cite{disp-monotone-1}, $U$ is Lipschitz in $x$. 
Fixing $x\in \reels^d$ and $q_0=\tilde{q}_0\in H$, we define $(X^x_s,U^x_s)_{s\in [T-t,T]}$ (respectively $(\tilde{X}^x_s,\tilde{U}^x_s)_{s\in [T-t,T]}$) to be the solution of \eqref{eq: caracteristics in x fbsde with major} with $(Z^\varphi_s,\mathcal{L}(X_s,U_s|\mathcal{F}^0_s))_{s\in [T-t,T]}$ (respectively $(\tilde{Z}^\varphi_s,\mathcal{L}(\tilde{X}_s,\tilde{U}_s|\mathcal{F}^0_s))_{s\in [T-t,T]}$).
Since $U$ is Lipschitz in $x$, this system can be viewed as a decoupled forward backward system, moreover we already know that $\|\nabla_q\varphi\|_\infty$ is bounded and hence have bounds in $L^\infty(\Omega, L^\infty([T-t,T],\reels^{d^0}))$ for $(Z^\varphi_s,\tilde{Z}^\varphi_s)_{s\in [T-t,T]}$ by Lemma \ref{lemma: Z bound by lip}. It follows by standard gronwall type argument as well as the above estimates that 
\[\forall s\in [T-t,T], \|(X^x_s,U^x_s)-(\tilde{X}^x_s,\tilde{U}^x_s)\|\leq (C_{coef}+Ct)\|X_0-\tilde{X}_0\|,\]
for possibly a larger constant $C$. Since 
\[U^x_0=U(T-t,x,\mathcal{L}(X_0)), \tilde{U}^x_0=U(T-t,x,\mathcal{L}(\tilde{X}_0)),\]
and that this is true for any initial condition $X_0,\tilde{X}_0$ and any $t\leq t^*_n$ we deduce that $U$ is also Lipschitz on $\mathcal{P}_2(\reels^d)$. The end results follows by letting $n$ goes to infinity.
\end{proof}
We now have all the tools at hands to state the main existence result on Lipschitz solutions. 
\begin{thm}
\label{thm: existence mfgmp}
Under Hypothesis \ref{hyp: locally lipschitz in z}
\begin{enumerate}
    \item[-] There exists a time $t>0$ for which there exists a Lipschitz solution to MFGMP$(T,\psi,g,H_z,L,F,G)$ on $[0,t]$
    \item[-] There exists a maximal time of existence $T_c$ and an associated Lipschitz solution, in the sense that for any Lipschitz solution to MFGMP$(T,\psi,g,H_z,L,F,G)$ defined on $[0,\tau)$ for some $\tau\in [0,T]$ 
    \[\tau \leq T_c.\]
    \item[-] Either $T_c=T$ or $T_c<T$ and
    \[\limsup_{t\to T_c}\|(U,\varphi)(t,\cdot)\|_{Lip}=+\infty.\] 
    \item[-] If $T_c=T$ and 
    \[\limsup_{t\to T}\|(U,\varphi)(t,\cdot)\|_{Lip}<+\infty,\]
    then there exists a Lipschitz solution to MFGMP$(T,\psi,g,H_z,L,F,G)$ on $[0,T]$. 
\end{enumerate}
\end{thm}
\begin{proof}
\noindent\textit{Step 1: all coefficients are Lipschitz}

We assume first that $\psi,g,H_z,L,f,g$ are Lipschitz with constant $L$. This assumption will be relaxed later. In this particular case, this result is a direct consequence of previous results from \cite{probabilistic-mfg} as we now explain.

First, the existence on a sufficiently short time interval follows from Theorem 5.4 of \cite{probabilistic-mfg} and the subsequent results Proposition 5.7 and 5.8. They first show that there exists a time $T_L$ depending only on $L$ such that for any $t\leq T_L$ there exists a unique solution to \eqref{eq: def lip sol system} for any initial condition $(X_0,q_0)$. This is done by a contraction argument which is quite standard in the theory of FBSDEs. Since there exists exactly one solution, the pair $(U,\varphi)$ is uniquely defined on $[0,T_L]$ by solving the FBSDE for any initial condition. It then remains to check that the pair thereby defined is indeed a decoupling field for the FBSDE \eqref{eq: def lip sol system} on $[0,T_L]$, which is the subject of Proposition 5.8. While they only show that for any initial condition $(X_0,q_0)$ and any $t\in [0,T_L]$ the solution of \eqref{eq: def lip sol system} satisfies 
\[\esp{\sup_{s\in [T-t,T]} |(X_s,q_s,U_s,\varphi_s)|+\int_{T-t}^T |(Z_s,Z^\varphi_s)|^2 ds}<+\infty,\]
by a variant of Lemma \ref{lemma: Z bound by lip}, this is sufficient to deduce that 
\[\exists C,  \quad |Z^\varphi_s|<C \quad a.e \text{ in }[T-t,T]\quad  a.s.\]

Next, we turn to the statement on the maximal time of existence. Let us assume that there exists a solution to MFGMP$(T,\psi,g,H_z,L,F,G)$ on $[0,t]$ for some $t<T$. If $\limsup_{s\to t}\|(U,\varphi)(s,\cdot)\|\leq C_{Lip}<+\infty$ then there exists a decreasing sequence $(\delta_n)_{n\in \mathbb{N}}\subset \reels^+$ such that 
\[\lim_{n\to \infty} \delta_n=0,\]
and 
\[\forall n\in \mathbb{N}, \quad \|(U,\varphi)(t-\delta_n,\cdot)\|_{Lip}\leq C_{Lip}.\]
By the above existence result, there exists a time $t^*$ depending only on $C_{Lip},T$ and $L$ such that for any $n\in \mathbb{N}$, there exists a Lipschitz solution to MFGMP$(T-t+\delta_n,\varphi(t-\delta_n,\cdot),U(t-\delta_n,\cdot),H_z,L,F,G)$ on $[0,t^*]$. In particular, for $n$ sufficiently large 
\[t-\delta_n+t^*>t+\frac{t^*}{2}\]
and by Lemma \ref{lemma: bootstraping lip sol}, we can construct a Lipschitz solution to MFGMP$(T,\psi,g,H_z,L,F,G)$ on $[0,t+\frac{t^*}{2}]$. Clearly so long as $t<T$ and $\liminf_{s\to t} \|(U,\varphi)(s,\cdot)\|_{Lip}<+\infty$ this argument can be repeated. In particular, if 
\[\limsup_{t\to T} \|(U\varphi)(t,\cdot)\|_{Lip}<+\infty,\]
then for $\delta$ sufficiently small, there exists a Lipschitz solution to MFGMP$(T-\delta,\varphi(T-\delta,\cdot),U(T-\delta),H_z,L,F,G)$ on $[0,\delta]$ and by a bootstraping argument, there exists a Lipschitz solution to \linebreak MFGMP$(T,\psi,g,H_z,L,F,G)$ on $[0,T]$.

\noindent\textit{Step 2: general coefficients}

Hypothesis \ref{hyp: locally lipschitz in z} is now in force. for $f=H_z,L,F,G$, we define 
\[f^n(x,q,u,z,\mu)=f(x,q,u,z\wedge n\vee (-n),\mu).\]
so that $\psi,g,H_z^n,L^n,F^n,G^n$ are Lipschitz with constant $L_n=C_{coef}+\omega(n)$, and satisfy Hypothesis \ref{hyp: locally lipschitz in z} uniformly in $n\in \mathbb{N}$. Moreover 
\[\forall z\geq 0, \quad \lim_{n\to \infty }\|(H_z,L,F,G)-(H_z^n,L^n,F^n,G^n)\|_{\infty,z}=0.\]
By the above results for Lipschitz coefficients, for fixed $n\in \mathbb{N}$ there exists a Lipschitz solution to MFGMP$(T,\psi,g,L^n_p,L^n,F^n,G^n)$ on $[0,t_n]$ for some $t_n$ depending on $L_n$. Moreover by Lemma \eqref{lemma: local lipschitz estimate}, there exists $t^*>0$ and a constant $C^*$ depending only on $C_{coef},\omega,\sigma^0$ such that 
\[\forall n \in \mathbb{N}, \quad \forall t<t^*\wedge t_n, \|(U,\varphi)(t,\cdot)\|_{lip}\leq C^*.\]
In particular, following step 1 this implies that 
\[\forall n\in \mathbb{N} \quad t_n\geq t^*.\]
By Theorem \ref{thm: stability lip sol}, on $[0,t^*]$, $(U^n,\varphi^n)_{n\in \mathbb{N}}$ converges to a Lipschitz solution $(U,\varphi)$ of MFGMP$(T,\psi,g,H_z,L,F,G)$. which ends to show the first statement on local existence. 

As for statements on $T_c$, first if there exist two Lipschitz solutions $(U^1,\varphi^1)$, $(U^2,\varphi^2)$ defined respectively on $[0,T_1]$ and $[0,T_2]$ then by Theorem \ref{lemma: uniqueness lip sol}, they must coincide on the interval $[0,T_1\wedge T_2]$. We define 
\[T_c=\sup \left\{t\in (0,T],\quad \text{there exists a Lipschitz solution } (U,\varphi) \text{ on } [0,t)\right\}.\]
Let us assume that 
\[T_c<T,\]
and that $\liminf_{s\to T_c}\|(U,\varphi)(s,\cdot)\|_{Lip}\leq C_{T_c}<+\infty$. In this case, proceeding as in step 1, we can build a sequence of Lipschitz solutions $(U^n,\varphi^n)$ to MFGMP$(T-T_c+\delta,\varphi(T_c-\delta,\cdot),U(T_c-\delta,\cdot),H_z^n,L^n,F^n,G^n)$ defined on $[0,t^*]$ for some $t^*$ depending only on $C_{T_c},C_{coef},\omega,\sigma^0$ and $t^*>2\delta$ as we did when coefficient were Lipschitz. Then by Theorem \ref{thm: stability lip sol}, there exists a Lipschitz solution to MFGMP$(T,\psi,g,H_z,L,F,G)$ on $[0,t+\frac{t^*}{2}]$. The rest of proof follows by using arguments presented in step 1.
\end{proof}
\begin{remarque}
The third statement is necessary. The fact that there exists a Lipschitz solution on $[0,T)$ does not necessary means it can be extended to the whole interval. This construction is analogous to the Cauchy Lipschitz Theorem for ordinary differential equations. Although there exists a solution to 
\[\frac{dy}{dt}=y^2(t), \quad y(0)=1.\]
on $[0,1)$ given by 
\[y(t)=\frac{1}{1-t},\]
it cannot be extended into a continuous solution to this ODE defined on $[0,1]$.
\end{remarque}
\begin{remarque}
    Essentially, we have proved that so long as the coefficients are Lipschitz and $\sigma^0>0$, not only does there exist a unique solution to the major-minor problem on a sufficiently short time interval, but so long as we have Lipschitz a priori estimates on the control of minor players and the value function of the major player there exists a unique solution. In itself, this Theorem comes from the regularizing effect of common noise on the major player. Our analysis, and even the formulation of the problem relied extensively on the assumption $\sigma^0>0$. Overall, this result is quite general and much like the analogous notion for mean field games \cite{lipschitz-sol}, we believe it should prove useful for the study of major-minor mean field games even outside of the context of this article. 
\end{remarque}

\subsection{Nash equilibrium}
\label{subsection: Nash}
\subsubsection{Formulation of the equilibrium}
We now detail the link between Lipschitz solutions and the Nash equilibrium in MFGs of controls with a major player. First, we need to properly formulate the problem of mean field games with a major player. Given a distribution of players $(m_s)_{s\in [0,T]}:\Omega^0\to \mathcal{P}_2([0,T],\reels^{2d})$ and a control $(\alpha^0_t)_{t\in [0,T]}$ progressively adapted to the filtration of common noise, minor players aim at minimizing
\[
\begin{array}{c}
J(\alpha,\alpha^0,(m_t)_{t\in [0,T]})= \displaystyle \espcond{\int_0^T L(X^\alpha_s,\alpha_s,q^{\alpha^0}_s,m_s)ds+u(T,X^\alpha_T,q^{\alpha^0}_T,\mu_T)}{\mathcal{F}^0_T},\\
dX^\alpha_s=-\alpha_sds+\sqrt{2\sigma_x}dB_s,
\end{array}
\]
where $\mu_T$ indicates the marginal of $m_T$ over the first $d$ variables. To present the problem of the major player, we restrict to minor players with feedback controls of the form 
\[(t,x,q,\mu)\mapsto \phi(t,x,q,\mu),\]
for deterministic and continuous functions $\phi$. In this case the major player pays the cost
\begin{gather*}
J^0(\alpha^0,\phi)=\esp{\int_0^T L^0(q^{\alpha^0}_s,\alpha^0_s,m^\phi_s)ds+\varphi(T,q^\alpha_T,\mu^\phi_T)}\\
dq^{\alpha^0}_t=-\alpha^0_tdt+\sqrt{2\sigma^0}dW^0_t,\\
dX^\phi_t=-\phi(t,X^\phi_t,q^{\alpha^0}_t,m^\phi_t)dt+\sqrt{2\sigma}dB_t,\\
m^\phi_t=\mathcal{L}(X^\phi_t,\phi(t,X^\phi_t,q^{\alpha^0}_t,\mathcal{L}(X^\phi_t|\mathcal{F}^0_t))|\mathcal{F}^0_t).
\end{gather*}
Let us insist that $(m^\varphi_t)_{t\in [0,T]}$ indicates the joint law of the state and controls of the minor players, as we are considering mean field games of controls. In general, the equilibrium condition is formulated by requiring both the major and minor players controls to be in feedback form \cite{mfg-with-common-noise,probabilistic-mfg}. However, even if there exists a Lipschitz solution, the control of the major player cannot, in general, be expressed in terms of such a feedback function. Consequently, we define admissible controls in the following way
\begin{definition}
    \begin{enumerate}
        \item[-]A control $\alpha^0$ is admissible for the major player if $\alpha^0\in \mathcal{H}^T_{\mathcal{F}^0}$
        \item[-] A feedback function $\phi$ is an admissible markovian control for the minor players if for any admissible control $\alpha^0$, there exists a unique strong solution to 
        \[\left\{\begin{array}{c}
            dX_t=-\alpha^*_t(\alpha^0,\phi)dt+\sqrt{2\sigma}dB_t, \quad X_0=X\\
            dq_t=-\alpha^{0}_tdt+\sqrt{2\sigma^0}dW^0_t,\quad q_0=q\\
            \alpha^*_t(\alpha^0,\phi)=\phi(t,X_t,q_t,m^*_t),\\
            m^*_t(\alpha^0,\phi)=\mathcal{L}(X_t,\alpha^*_t|\mathcal{F}^0_t),
        \end{array}\right.\]
which satisfies $\alpha^*(\alpha^0)\in \mathcal{H}^T_\mathcal{F}$.
    \end{enumerate}
\end{definition}
\begin{remarque}
    In particular any Lipschitz function $\phi$ is an admissible feedback control for minor players. 
\end{remarque}

we define a Nash equilibrium as such
\begin{definition}
    \label{def: nash equilibrium}
    A couple $(\phi^*,\alpha^{*,0}_s)_{s\in [0,T]}$ is a Nash equilibrium for the mean field game with a major player if 
    the following holds 
    \begin{enumerate}
        \item[-] The pair $(\phi^*,\alpha^{*,0})$ is admissible 
        \item[-]
       \[\alpha^*(\alpha^{0,*},\phi^*)\in \text{arginf}_\alpha J(\alpha,\alpha^{*,0},(m^*_t(\alpha^{0,*},\phi^*))_{t\in [0,T]}).\]
        \item[-] \[\alpha^{*,0}\in \text{arginf}_{\alpha^0} J^0(\alpha^0,\phi^*).\]
    \end{enumerate}
\end{definition}
The first condition ensures that the optimal control of the major player given the feedback control is well posed. Given a control for the major player $\alpha^0$, the optimality condition for minor players is that of standard mean field games 
\[(\alpha^*_t)_{t\in [0,T]}\in \text{arginf}_{\alpha} J\left(\alpha,\alpha^0,\left(\mathcal{L}(X^\alpha_t,\alpha_t)\right)_{t\in [0,T]}\right),\]
except for the fact that we have forced their controls to be given in feedback form through $\phi^*$. This additional requirement allows for a clear formulation of the equilibrium condition for the major player. This asymmetry reflects the fact that whenever the major player makes a decision it impacts all players and consequently also the mean field quantities. 
\begin{remarque}
    \label{remarque: mfg with dependency on control of major}
    It is also possible for the cost of minor players to depend on the control of the major player through their Lagrangian $L$. In this case, the definition of equilibrium can be adapted easily by considering feedback controls of the form 
    \[(t,x,q,\alpha_0,\mu)\mapsto \phi(t,x,q,\alpha_0,\mu).\]
    The equilibrium condition of the major player is also affected. A natural assumption is that the major player does not predict the change in the dynamics of players by changing its control and the equilibrium condition becomes 
    \begin{equation}
        \label{eq: optimal control problem of major}
    \left\{
        \begin{array}{c}
    \alpha^{*,0}=\text{arginf}_{\alpha^0}\displaystyle \esp{\int_0^T L^0(q^{\alpha^0}_s,\alpha^0_s,m^\phi_s)ds+\varphi(T,q^{\alpha^0}_T,\mu^\phi_T)}\\
dq^{\alpha^0}_t=-\alpha^0_tdt+\sqrt{2\sigma^0}dW^0_t,\\
dX^\phi_t=-\phi(t,X^\phi_t,q^{\alpha^0}_t,\alpha^{*,0}_t,m^\phi_t)dt+\sqrt{2\sigma}dB_t,\\
m^\phi_t=\mathcal{L}(X^\phi_t,\phi(t,X^\phi_t,q^{\alpha^0}_t,\alpha^{*,0}_t,\mathcal{L}(X^\phi_t|\mathcal{F}^0_t))|\mathcal{F}^0_t).
    \end{array}
\right.
\end{equation}
Beyond this technical detail, there is an additional conceptual difficulty in the case of MFGs of controls. Indeed since controls of minors players depends on the control of the major, if the dynamics of the major player also depend on the law of controls of individual players, then there is an additional fixed point in the definition of $\alpha^{*,0}$, which requires additional assumptions on the Lagrangian of the major player for the wellposedness of the problem. Such a phenomenon was already pointed out in \cite{MFG-MAJOR-LIONS}. We do not treat this additional technical point, and will assume instead in this section that controls of minor players do not depend on the control of the major player to avoid it. Although we do not treat this case directly, this problem does not arise if controls of minor players depend on the control of the major player when the state of the major player interact with the minor players only through the law of their state (and not of their control). This suggest that our framework can be extended to such models. 
\end{remarque}
\subsubsection{Link with Lipschitz solutions}
We now make the following assumption
\begin{hyp}
    \label{hyp: mfg with major optimality minor}
    For minor players
    \begin{enumerate}
        \item[-] The function $\alpha \mapsto L(x,\alpha,q,m)$ is convex uniformly in $(x,q,m)\in \reels^{d+d^0}\times \mathcal{P}_2(\reels^{2d})$.
        \item[-] The derivatives $\nabla_x u, \nabla_x L$ exists and define Lipschitz functions.
        \item[-] $u(T,\cdot)$ is bounded from below and there exists constants $c_1,c_2>0$ such that 
\[\forall (x,q,m,\alpha,), L(x,\alpha,q,m)\geq -c_1+c_2|\alpha|^2.\] 
        \item[-] $\nabla_\alpha L$ exists and is continuous, moreover the Hilbertian lift 
\[\alpha \mapsto \nabla_\alpha L(X,q,\alpha,\mathcal{L}(X,\alpha)),\]
has a well defined Lipschitz inverse $(X,q,\alpha)\mapsto \nabla_\alpha \tilde{L}^{-1}(X,q,\alpha)$ and there exists a Lipschitz function 
\[L_{inv}:\reels^d\times \reels^{d^0}\times \reels^d\times \mathcal{P}_2(\reels^{2d})\to \reels^d,\]
such that 
\[\forall (X,\alpha,q,\alpha^0)\in \mathcal{H}^{2d}\times \reels^{2m}, \quad \nabla_\alpha \tilde{L}^{-1}(X,q,\alpha)=L_{inv}(X,q,\alpha,\mathcal{L}(X,\alpha)),\]
    \end{enumerate}
\end{hyp}
\begin{remarque}
The last condition ensures that the additional fixed point condition in MFGs of control is well posed in the sense that for any $X,U\in \mathcal{H}^d, q\in \mathcal{H}^{d^0}$ there exists a unique $\alpha\in \mathcal{H}^d$ such that
\[\alpha=D_p H(X,q,U,\mathcal{L}(X,\alpha)),\]
which is given by
\[\alpha=\nabla_\alpha \tilde{L}^{-1}(X,q,U),\]
where $H$ indicates the Hamiltonian of minor players
\[H(x,q,p,\mu)=\sup_\alpha \left\{ \alpha\cdot p-L(x,q,\alpha,\mu)\right\}.\]
If $L$ does not depend on $\mathcal{L}(\alpha)$ then the last condition is true under standard assumptions (namely the strong convexity of $\alpha\mapsto L(X,q,\alpha,\mathcal{L}(X))$) and the inverse is simply given by $D_pH$. Otherwise, this is true as soon as the following strong monotonicity assumption holds in $\alpha$
\begin{gather*}
\exists c>0, \quad \forall (X,q,\alpha,\alpha')\in \mathcal{H}^{3d+d^0}\\
\esp{\left(\nabla_\alpha L(X,q,\alpha,\mathcal{L}(X,\alpha))-\nabla_\alpha L(X,q,\alpha',\mathcal{L}(X,\alpha'))\right)\cdot \left(\alpha-\alpha'\right)}\geq c\| \alpha-\alpha'\|^2,
\end{gather*}
and $L$ is Lipschitz in all other arguments, for exemple. Although our formulation in terms of an infinite dimensional inverse is somewhat unusual, let us insist that the wellposedness of the additional fixed point inherent to MFG of controls under displacement monotonicity is quite standard \cite{jackson2025quantitativeconvergencedisplacementmonotone}. Our formulation avoids the need for an implicit mapping and seems quite natural from a probabilistic standpoint. In fact, the function $L_{inv}$ is exactly given (see \cite{monotone-sol-meynard}) in terms of $H$ and $\nabla_\alpha \tilde{L}^{-1}$ by 
\[\forall (x,q,u,X,U)\in \reels^{2d+d^0}\times \mathcal{H}^{2d}, L_{inv}(x,q,u,\mathcal{L}(X,U))=D_p H(x,q,u,\mathcal{L}(X,\nabla_\alpha \tilde{L}^{-1}(X,q,U))),\]
meaning the implicit mapping often used in MFGs of controls \cite{ziad-kobeissi,mou2022meanfieldgamescontrols} is exactly given by 
\[(X,q,U)\mapsto \mathcal{L}(X,\nabla_\alpha \tilde{L}^{-1}(X,q,U)).\]
The additional dependency on the variable $q$ comes here from the fact that we are considering MFGs with a major player.
\end{remarque}
Before presenting the main result of this section, let us give this reminder on optimality conditions in optimal control. 
\begin{lemma}
    \label{lemma: optimality of minor}
Under Hypothesis \ref{hyp: mfg with major optimality minor}, for any  $\mathcal{F}^0-$adapted process $(q_t,m_t)_{t\in [0,T]}\subset \reels^{d^0}\times \mathcal{P}_2(\reels^{2d})$ such that 
\[\forall t\in [0,T] \quad \esp{q_t^2+\mathcal{W}^2_2(m_t,\delta_0)}<+\infty,\]
there exists an optimal control $(\alpha^*_t)_{t\in [0,T]}$ solving 
\begin{equation}
    \label{eq: problem of minor}
    \underset{\alpha}{\text{essinf}} \espcond{\int_0^T L(X^\alpha_s,q_s,\alpha_s,m_s)+u(T,X^\alpha_T,q_T,\pi_d m_T)}{\mathcal{F}^0_T}.\end{equation}
with \[X^\alpha_t=X-\int_0^t \alpha_s ds+\sqrt{2\sigma_x}B_t.\]
Moreover, letting $(U^{\alpha^*}_t)_{t\in [0,T]}$ be the unique solution of the BSDE
\begin{equation}
    \label{eq: bsde optimal minor mfgmp}
    U^{\alpha^*}_t=\nabla_x u(T,X^{\alpha^*}_T,q_T,\pi_d m_T)+\int_t^T \nabla_x L(X^{\alpha^*}_s,q_s,\alpha^*_s,m_s)ds+\int_t^T Z^{\alpha^*}_td(B_t,W^0_t),
\end{equation}
the following holds 
\[\|U^{\alpha^*}_\cdot-\nabla_\alpha L(X^{\alpha^*}_\cdot,q_\cdot,\alpha^*_\cdot,m_\cdot)\|_T=0.\]
\end{lemma}
\begin{proof}
The existence of an optimal control can be proven either by working first with relaxed controls as in \cite{Djete} and then showing that this implies the existence of strong optimal controls \cite{Lacker-mimicking}, or by constructing directly an optimal feedback control as in \cite{probabilistic-mfg} (see Theorem 1.57). As for the necessary optimality conditions, those are just a consequence of the fact that any optimal control $\alpha^*$ must be a minimizer of
\[J(\alpha)=\espcond{\int_0^T L(X^\alpha_s,q_s,\alpha_s,m_s)+u(T,X^\alpha_T,q_T,\pi_d m_T)}{F^0_T}.\]
 Fixing $\alpha'\in \mathcal{H}^T_\mathcal{F}$, for any $\lambda\in [0,1]$ 
\[J(\lambda \alpha'+(1-\lambda)\alpha^*)-J(\alpha^*)\geq 0.\]
Dividing by $\lambda>0$ and taking the limit as $\lambda\to 0$ yields the necessary optimality conditions, where we have used the fact that since 
\[\esp{\left|\nabla_x u(T,X^{\alpha^*}_T,q_T,\pi_d m_T)\right|^2+\int_0^T \left|\nabla_x L(X^{\alpha^*}_s,q_s,\alpha^*_s,m_s)\right|^2}<+\infty,\]
there exists a unique solution to the backward SDE \eqref{eq: bsde optimal minor mfgmp}. 
\end{proof}
\begin{remarque}
If there exists a solution to the associated mean field game, that is a control $(\alpha^*_t)_{t\in [0,T]}$ optimal given the flow of measures
\[\forall t\in [0,T], \quad m_t=\mathcal{L}(X^{\alpha^*}_t,\alpha^*_t),\]
then the following holds 
\[\forall t\in [0,T],\quad U^{\alpha^*}_t=\nabla_\alpha L(X^{\alpha^*}_t,q_t,\alpha^*_t,\mathcal{L}(X^{\alpha^*}_t,\alpha^*_t)).\]
In particular using the invertibility of the Hilbertian lift of $\nabla_\alpha L$, this implies that $(X^{\alpha^*}_t,U^{\alpha^*}_t,Z^{\alpha^*}_t)_{t\in [0,T]}$ is a solution to the FBSDE
\begin{equation*}
    \left\{
    \begin{array}{l}
\displaystyle X^{\alpha^*}_t=X-\int_0^t \nabla_\alpha \tilde{L}^{-1}(X^{\alpha^*}_s,q_s,U^{\alpha^*}_s)ds+\sqrt{2\sigma_x}B_t,\\
 \displaystyle U^{\alpha^*}_t=\nabla_x u(T,X^{\alpha^*}_T,q_T,\pi_d m_T)+\int_t^T \nabla_x L(X^{\alpha^*}_s,q_s, \nabla_\alpha \tilde{L}^{-1}(X^{\alpha^*}_s,q_s,U^{\alpha^*}_s),m_s)ds+\int_t^T Z^{\alpha^*}_td(B_t,W^0_t),\\
 m_s=\mathcal{L}(X_s,\nabla_\alpha \tilde{L}^{-1}(X^{\alpha^*}_s,q_s,U^{\alpha^*}_s)).
    \end{array}
    \right.
\end{equation*}
This formulation is only possible thanks to our invertibility assumption on the lift in Hypothesis \ref{hyp: mfg with major optimality minor}, which justifies that this is a natural assumption for MFG of controls. 
\end{remarque}

\begin{hyp}
    \label{hyp: mfg with major optimality major}
For the Major player 
    \begin{enumerate}
        \item[-] its Lagrangian is $c_{L^0}-$ strictly convex in $\alpha^0$.
        \item[-] There exists constants $C^0,C_1,C_2,C_3\geq 0$ $p\geq 0$ such that 
        \begin{gather*}
            \forall (q,q',\mu,\nu,\alpha^0,\alpha^{0'}),\\
            |L(q,\alpha^0,\mu)-L(q',\alpha^{0'},\nu)|\leq C^0(1+|\alpha^0\wedge \alpha^{0'}|^p)\left(|q-q'|+\mathcal{W}_2(\mu,\nu)+|\alpha^0-\alpha^{0'}|\right)\\
            |\nabla_{\alpha^0}L(q,\alpha^0,\mu)-\nabla_{\alpha^0}L(q',\alpha^0,\nu)|\leq C_0\left(|q-q'|+\mathcal{W}_2(\mu,\nu)\right),\\
          -C_1+C_{2}|\alpha^0|^p\leq L(q,\alpha^0,\mu)\leq C_1+C_3|\alpha^0|^p.
        \end{gather*} 
        \item[-]Its terminal cost $\psi(\cdot)$ is Lipschitz.
    \end{enumerate}
\end{hyp}
\begin{remarque}
Let us insist that both Hypothesis \ref{hyp: mfg with major optimality minor} and the current one are not necessary conditions, but rather a convenient setting to highlight the link between Lipschitz solutions and a Nash equilibrium. They also serve as examples of Lagrangians for which the associated Hamiltonians satisfy Hypothesis \ref{hyp: locally lipschitz in z}.
\end{remarque}
We define the Hamiltonian of the major player $H^0$ as follows 
\begin{equation}
    \label{eq: def hamiltonian major}
    \forall (q,z,m)\in \reels^{2m}\times \mathcal{P}_2(\reels^{2d}), \quad H^0(q,z,m)=\sup_\alpha \left(\alpha\cdot z-L^0(q,\alpha,m)\right).\end{equation}
\begin{lemma}
    \label{lemma: optimality of major}
Under Hypothesis \ref{hyp: mfg with major optimality major}, given a Lipschitz function $\phi$, if there exists a strong solution $(q_t,X_t,Z_t)_{t\in [0,T]}$ to the Forward Backward system
\[
\left\{
\begin{array}{l}
  \displaystyle q_t=q_0-\int_0^t D_z H^0(q_s,Z_s,m_s)ds+\sqrt{2\sigma^0}W^0_t\\
   \displaystyle\varphi_t=\psi(q_T,\pi_d m_T)+\int_t^T L^0(q_s,D_zH^0(q_s,Z_s,m_s),m_s)ds-\sqrt{2\sigma^0}\int_t^T Z_sdW^0_s\\
   \displaystyle X_t=X_0+\int_0^t \phi(t,X_t,q_t,m_s)ds+\sqrt{2\sigma}dB_t\\
   \displaystyle m_t=\mathcal{L}(X_t,\phi(t,X_t,q_t,\mathcal{L}(X_s|\mathcal{F}^0_s))|\mathcal{F}^0_t),
\end{array}
\right.
\]
which satisfies 
\[\exists C>0,\quad |Z_t|\leq C \text{ a.e in }[0,T], a.s,\]
then the there exists a unique optimal control for the problem \eqref{eq: optimal control problem of major} in $\mathcal{H}^T_{\mathcal{F}^0}$ which is given by 
\begin{equation}
    \label{eq: optimal feedback major}
    \forall t\in [0,T],\quad \alpha^{0,*}_t=D_z H^0(q_t,Z_t,\mathcal{L}(X_t|\mathcal{F}^0_t)).\end{equation}
\end{lemma}
\begin{proof}
First let us remark that under Hypothesis \ref{hyp: mfg with major optimality major}, the Hamiltonian of the major player satifies 
\[\exists C>0, \forall (q,z,m)\in \reels^{2m}\times \mathcal{P}_2(\reels^{2d}), \quad |D_zH^0(q,z,m)|\leq C(1+|z|).\]
In particular, letting $(\alpha^{0,*}_t)_{t\in [0,T]}$ be defined as in \eqref{eq: optimal feedback major}, it follows from the boundedness of $(Z_t)_{t\in [0,T]}$ that $(\alpha^{0,*}_t)_{t\in [0,T]}\in L^\infty(\Omega,L^\infty([0,T],\reels^{d^0}))$. Moreover by the definition of $H^0$ as the Legendre transform of $L^0$ and the convexity of $L^0$
\[\forall (q,z,m), \quad L^0(q,D_z H^0(q,z,m),m)=D_zH^0(q,z,m)\cdot z-H^0(q,z,m).\]
Consequently 
\[J(\alpha^*)=\varphi_0=\psi(q_T,\pi_d m_T)-\int_0^T H^0(q_t,Z_t,m_t)dt-\sqrt{2\sigma^0}\int_0^T Z_sd\left(W^0_t-\int_0^t D_zH^0(q_s,Z_s,m_s)ds\right) \quad a.s.\]
By the strong convexity of $L^0$ it follows that 
\begin{equation}
    \label{eq: non optimal hamiltonian major}
    \alpha \mapsto L^0(q,\alpha,m)-\alpha\cdot z,
\end{equation}
is also strongly convex, hence 
\[\forall (\alpha,\beta,q,z,m), \quad L^0(q,\beta,m)-\beta\cdot z\geq L(q,\alpha,m)-\alpha \cdot z+(\nabla_\alpha L^0(q,\alpha,m)-z)\cdot (\beta-\alpha)+\frac{1}{2}c_{L^0}|\alpha-\beta|^2.\]
In particular evaluating this expression for $\alpha$ the minimizer of \eqref{eq: def hamiltonian major} yields 
\[\forall (\alpha,\beta,q,z,m), \quad L^0(q,\beta,m)-\beta\cdot z\geq -H(q,z,m)+\frac{1}{2}c_{L^0}|D_zH(q,z,m)-\beta|^2.\]
In light of all those elements, the proof is then a simple adaptation of Theorem 1.57 of \cite{probabilistic-mfg}.
\end{proof}
We now have all the tools at hands to state the main results on the link between Nash equilibrium and Lipschitz solutions. We consider the following version of \eqref{mfg with major}

\begin{equation}
\label{eq: true formulation mfgmp}
\left\{
\begin{array}{l}
     \displaystyle X_t=X_0-\int_0^t  \alpha^*_t ds+\sqrt{2\sigma}B_t, \\
     \displaystyle q_t=q_0-\int_0^t \alpha^{0,*}_tds+\sqrt{2\sigma^0}W^0_t,\\
    \displaystyle  U_t= U_T+\int_t^T \nabla_x L(X_s,q_s,\alpha_s,\mathcal{L}(X_s,\alpha^*_s|\mathcal{F}^0_s))ds-\int_t^T
    Z_sd(B,W^0)_s,\\
\displaystyle \varphi_t= \varphi_T+\int_t^T L^0(q_s,\alpha^0_s,\mathcal{L}(X_s,\alpha^*_s|\mathcal{F}^0_s))ds-\sqrt{2\sigma^0}\int_t^T Z^\varphi_sdW^0_s,\\
U_T=\nabla_x u(T,X_T,q_T,\mathcal{L}(X_T|\mathcal{F}^0_T)), \quad \varphi_T=\psi(q_T,\mathcal{L}(X_T|\mathcal{F}^0_T)),\\
\forall t\in [0,T], \left\{\begin{array}{l}\alpha^*_t=\nabla_\alpha \tilde{L}^{-1}(X_s,q_s,U_s),\\
\alpha^{0,*}_t=D_p H^0(q_s,Z^\varphi_s,\mathcal{L}(X_s,\alpha^*_s|\mathcal{F}^0_s)),\\
\mathcal{F}^0_t=\sigma\left(q_0,(W^0_s)_{s\leq t}\right),
\end{array}\right.
\end{array}
\right.
\end{equation}
Formally, a Lipschitz solution $(U,\varphi)$ to \eqref{eq: true formulation mfgmp} is a solution of the master system defined through the characteristics. If there exists a smooth solution to this system, then it is already known \cite{mfgmp-cardaliaguet} that the optimal controls are given in feedback form as functions of $U,\nabla_p\varphi$. We now prove that even when $\varphi$ is only assumed to be Lipschitz, the control of minor players can still be written in feedback form 
\begin{thm}
    Under Hypotheses \ref{hyp: mfg with major optimality minor} and \ref{hyp: mfg with major optimality major}, for any admissible initial condition $(X_0,q_0)$
\begin{enumerate}
    \item[-] There exists $T_c>0$ such that \eqref{eq: true formulation mfgmp} admits a Lipschitz solution on $[0,T_c)$.
    \item[-] If there exists a Lipschitz solution $(U,\varphi)$ to \eqref{eq: true formulation mfgmp} on $[0,T]$, then letting
\[\phi^U: (t,x,q,\mu)\mapsto L_{inv}(x,q,U(T-t,x,q,\mu),\left(id_{\reels^d},U(T-t,\cdot,q,\mu)\right)_\#\mu),\]
the pair $(\phi^U,(\alpha^0_t)_{t\in [0,T]})$ is a Nash equilibrium in the sense of Definition \ref{def: nash equilibrium} for the major-minor problem. Moreover, it is the unique Nash equilibrium given by a Lipschitz markovian feedback $\phi$ for minor players, in the sense that if there exists another Nash equilibrium $(\phi',\alpha^{0,'})$ for a Lipschitz function $\phi'$ then 
\[\forall t\in [0,T], \quad (\alpha^0_t,\alpha(\alpha^0,\phi^U)_t)=(\alpha^{0,'}_t,\alpha(\alpha^{0,'},\phi')_t) \quad a.s.\]
\end{enumerate}
\end{thm}
\begin{proof}

    \noindent\textit{Step 1: $(\phi,\alpha^{0,*})$ is a Nash equilibrium}\\
    Under Hypotheses \eqref{hyp: mfg with major optimality minor} and \eqref{hyp: mfg with major optimality major}, 
    by the strong convexity in $\alpha^0$ of $L^0$,
    Hypothesis \eqref{hyp: locally lipschitz in z} holds for the coefficients $(\psi,\nabla_x u(T,\cdot,),\nabla_\alpha \tilde{L}^{-1},\nabla_x L\circ \nabla_\alpha \tilde{L}^{-1},D_pH^0,L^0\circ D_p H^0)$ and the first claim follows from Theorem \ref{thm: existence mfgmp}.
    
    We now fix $T>0$ and assume that there exists a Lipschitz solution to \linebreak MFGMP$(\psi,\nabla_x u(T,\cdot,),\nabla_\alpha \tilde{L}^{-1},\nabla_x L\circ \nabla_\alpha \tilde{L}^{-1},D_pH^0,L^0\circ D_p H^0)$ on $[0,T]$. By Lemma \ref{lemma: optimality of major}, the control $(\alpha^{0,*}_t)_{t\in [0,T]}$ as defined in \eqref{eq: true formulation mfgmp} is optimal for the Lipschitz feedback $\phi^U$. Moreover, letting 
    \[\forall t\in [0,T], \quad m_t=\mathcal{L}(X_t,\alpha^*_t|\mathcal{F}^0_t),\]
    it follows from Lemma \ref{lemma: optimality of minor}, that given $(q^{\alpha^{0,*}}_t,m_t)_{t\in [0,T]}$ there exists at least one optimal control to the problem \eqref{eq: problem of minor} and any optimal control must satisfy
    \[
    \left\{
        \begin{array}{l}
     \displaystyle X^\alpha_t=X_0-\int_0^t  \alpha_t ds+\sqrt{2\sigma}dB_s, \\
    \displaystyle  U^\alpha_t= \nabla_x u(T,X_T,q_T,\pi_d m_T)+\int_t^T \nabla_x L(X_s,q_s,\alpha_s,m_s)ds-\int_t^T
    Z_sd(B,W^0)_s,\\
    \forall s\in [0,T],\quad \alpha_s=D_p H(X^\alpha_s,q_s,U^\alpha_s,m_s)
        \end{array}
    \right.
    \]
    where $H$ is the Hamiltonian of minor players defined by 
    \[H(x,q,p,m)=\sup_\alpha \left(\alpha\cdot p-L(x,q,\alpha,m)\right).\]
    We now remind that for any random variable $(X,\alpha)\in \mathcal{H}^{2d}$, $L_{inv}$ admits the following representation \cite{monotone-sol-meynard}
    \[L_{inv}(x,q,a,\mathcal{L}(X,\alpha))=D_p H(x,q,a,\mathcal{L}(X,\nabla_\alpha \tilde{L}^{-1}(X,q,\alpha))).\]
    In particular 
    \[\forall (t,x,a)\in [0,T]\times \reels^{d+d^0+d}, \quad D_pH(x,q_t,a,m_t)=L_{inv}(x,q,a,\mathcal{L}(X_t,\alpha_t|\mathcal{F}^0_t)), \quad a.s.\]
    In light of this representation, the necessary conditions for the optimality condition of a control for minor players become
    \[
    \left\{
        \begin{array}{l}
     \displaystyle X^\alpha_t=X_0-\int_0^t  \alpha_t ds+\sqrt{2\sigma}B_t, \\
    \displaystyle  U^\alpha_t= \nabla_x u(T,X_T,q_T,\pi_d m_T)+\int_t^T \nabla_x L(X_s,q_s,\alpha_s,m_s)ds-\int_t^T
    Z^\alpha_sd(B,W^0)_s,\\
    \forall s\in [0,T],\quad \alpha_s=L_{inv}(X^\alpha_s,q_s,U^\alpha_s,\mathcal{L}(X_t,\alpha^*_t|\mathcal{F}^0_t))
        \end{array}
    \right.
    \]
    By Theorem \ref{lemma: uniqueness lip sol}, there exists a unique solution to this forward backward system which is given by $(X_t,U_t,Z_t)_{t\in [0,T]}$. This ends to show that $(\alpha^*_t)_{t\in [0,T]}$ is the unique optimal control for the minor problem given $(\alpha^{0,*}_t,m_t)_{t\in [0,T]}$ and consequently that the couple $\phi^U,(\alpha^{*,0}_t)_{t\in [0,T]}$ is a Nash equilibrium for the minor-major problem. \\

    \noindent\textit{Step 2: uniqueness of Nash equilibrium given by a Lipschitz feedback}\\
    Let us assume that there exists another Nash equilibrium $(\phi',(\alpha^{0,'}_t)_{t\in [0,T]})$, with a Lipschitz feedback $\phi'$. For any admissible control $\alpha$, and any pair $q_0,\mu\in \reels^{d^0}\times \mathcal{P}_2(\reels^d),$ we define 
    \[J(t,q,\mu,\alpha)=\esp{\psi(q^{\alpha,q_0}_T,\pi_dm^{\mu,q_0}_T)+\int^T_{T-t}L^0(q^{\alpha,q_0}_s,\alpha_s,m^{\mu,q_0}_s)ds},\] 
    where for $u\in [T-t,T]$
        \[
    \left\{
        \begin{array}{l}
     \displaystyle q^{\alpha,q^0}_{u}=q_0-\int_{T-t}^{u}  \alpha_s ds+\sqrt{2\sigma}\int_{T-t}^{u}dW^0_s, \\
\displaystyle X^{\mu,q_0}_{u}=X^\mu_0-\int_{T-t}^{u} \phi'(s,X^{\mu,q_0}_s,q^{\alpha^0}_s,m^{\mu,q_0}_s)ds+\sqrt{2\sigma}\int_{T-t}^{u}dB_s,\\
m^{\mu,q_0}_u=\mathcal{L}(X^{\mu,q_0}_u,\phi(u,X^{\mu,q_0}_u,q^{\alpha,q_0}_u,\mathcal{L}(X^{\mu,q_0}_u|\mathcal{F}^0_u))|\mathcal{F}^0_u),
        \end{array}
    \right.
    \]
    for a random variable $X^\mu_0$ independent of $(B_s-B_{T-t},W^0_s-W^0_{T-t})_{s\geq T-t}$ distributed along $\mu$. By standard Gronwall type estimates, there exists a constant $C_{Lip}>0$ such that for any admissible control $\alpha$ 
    \begin{gather*}\forall (q_0,q_0',\mu,\nu)\in \reels^{2m}\times \mathcal{P}_2(\reels^{2d}), \forall u\in [T-t,T]\\
         |q^{\alpha,q_0}_u-q^{\alpha,q_0'}_u|^2+\espcond{|X^{\mu,q_0}_u-X^{\nu,q_0'}_u|^2}{\mathcal{F}^0_u}\leq C_{Lip}(|q_0-q_0'|^2+\mathcal{W}^2_2(\mu,\nu)) \quad a.s. \end{gather*}
    We now remark that independently of the chosen initial condition $(q,\mu)$ 
    \[\inf_\alpha J(t,q,\mu,\alpha)\leq J(t,q,\mu,0)\leq C_1,\]
    consequently for a given $(q,\mu)$ the infimum is reached on the set 
    \[K_\alpha=\left\{\alpha\in \mathcal{H}^T_{\mathcal{F}^0}, \quad \esp{\int_{0}^T|\alpha_s|^p ds}\leq 2C_1\right\},\]
    by Hypothesis \ref{hyp: mfg with major optimality major}. However fixing $(q,q',\mu,\nu)$ and $\alpha\in K_\alpha$, 
    \begin{align*} 
        |J(t,q,\mu,\alpha)-J(t,q',\nu,\alpha)|&\leq \|\psi\|_{Lip}C_{Lip}\left(|q-q'|+\mathcal{W}_2(\mu,\nu)\right)\\
        &+\esp{\int_{T-t}^t\left|L^0(q^{\alpha,q_0}_s,\alpha_s,m^{\mu,q}_s)-L^0(q^{\alpha,q'}_s,\alpha_s,m^{\nu,q'_0}_s)\right|ds },
    \end{align*}
    and 
    \begin{gather*}
        \esp{\int_{T-t}^T\left|L^0(q^{\alpha,q_0}_s,\alpha_s,m^{\mu,q}_s)-L^0(q^{\alpha,q'}_s,\alpha_s,m^{\nu,q'_0}_s)\right|ds }\\
        \leq  \esp{\int_{T-t}^T\espcond{C^0(1+|\alpha_s|^p)(|q^{\alpha,q}_s-q^{\alpha,q'}_s|+\mathcal{W}_2(m^{\mu,q}_s,m^{\nu,q'}_s))|}{\mathcal{F}^0_s}ds }\\
        \leq 2(1+C^0)C_1)C_{Lip}(1+\|\phi'\|_{Lip})(T-t)\left(|q-q'|+\mathcal{W}_2(\mu,\nu)\right)
    \end{gather*}
    Letting 
    \[\varphi'(t,q,\mu)=\inf_\alpha J(t,q,\mu,\alpha)=\inf_{\alpha \in K_\alpha}J(t,q,\mu,\alpha),\]
    we deduce that $\varphi':[0,T]\times \reels^{d^0}\times \mathcal{P}_2(\reels^d)\to \reels$ is a Lispchitz function in $(q,\mu)$ uniformly on $[0,T]$. In particular, it follows naturally from Lemma \eqref{lemma: optimality of major} and Theorem \ref{thm: existence mfgmp}, that there exists a Lipschitz solution on $[0,T]$ to the forward backward system
\[
\left\{
\begin{array}{l}
  \displaystyle q'_t=q_0-\int_{0}^t D_z H^0(q'_s,Z^{0,'}_s,m'_s)ds+\sqrt{2\sigma^0}W^0_t\\
   \displaystyle\varphi'_t=\psi(q'_T,\pi_d m'_T)+\int_t^T L^0(q'_s,D_zH^0(q'_s,Z^{0,'}_s,m'_s),m'_s)ds-\sqrt{2\sigma^0}\int_t^T Z^{0,'}_sdW^0_s\\
   \displaystyle X'_t=X_0-\int_{0}^t \phi'(s,X'_s,q'_s,\mathcal{L}(X'_s|\mathcal{F}^0_s)),m'_s)ds+\sqrt{2\sigma}B_t\\
   \displaystyle m'_t=\mathcal{L}(X'_t,\phi'(t,X'_t,q'_t,\mathcal{L}(X'_t|\mathcal{F}^0_t))|\mathcal{F}^0_t),
\end{array}
\right.
\]
which is exactly given by $\varphi'$. In particular
\[\forall t\in [0,T],\quad  \alpha^{0,'}_t=D_zH^0(q'_s,Z^{0,'}_s,\mathcal{L}(X'_s|\mathcal{F}^0_s))|^2 \quad a.s.\]
Naturally, it follows that 
\[\forall t\in [0,T], \quad (X'_t,m'_t)=(X^{\phi'}_t,m^{\phi'}_t) \quad a.s. \]
and that for minor players  
\[\forall t\in [0,T], \quad \alpha^*(\alpha^{0,*})_t=\phi'(t,X'_t,q'_t,\mathcal{L}(X'_s|\mathcal{F}^0_s)) \quad a.s.\]
Using the necessary optimality condition for minor players on $[0,T]$ we deduce that the following must hold 
 \[
\left\{
\begin{array}{l}
  \displaystyle q'_t=q_0-\int_{0}^t D_z H^0(q'_s,Z^{0,'}_s,m'_s)ds+\sqrt{2\sigma^0}W^0_t\\
   \displaystyle\varphi'_t=\psi(q'_T,\pi_d m'_T)+\int_t^T L^0(q'_s,D_zH^0(q'_s,Z^{0,'}_s,m'_s),m'_s)ds-\sqrt{2\sigma^0}\int_t^T Z^{0,'}_sdW^0_s\\
   \displaystyle X'_t=X_0-\int_{0}^t \phi'(s,X'_s,q'_s,\mathcal{L}(X'_s|\mathcal{F}^0_s),m'_s)ds+\sqrt{2\sigma}B_t\\
    \displaystyle  U'_t= U_T+\int_t^T \nabla_x L(X'_s,q'_s,\phi'(s,X'_s,q'_s,\mathcal{L}(X'_s|\mathcal{F}^0_s)),\mathcal{L}(X'_s,\alpha_s|\mathcal{F}^0_s))ds-\int_t^T
    Z'_sd(B,W^0)_s,\\
   \displaystyle m'_t=\mathcal{L}(X'_t,\phi'(t,X'_t,q'_t,\mathcal{L}(X'_t|\mathcal{F}^0_t))|\mathcal{F}^0_t),\\
   \phi'(s,X'_s,q'_s,\mathcal{L}(X'_s|\mathcal{F}^0_s))=\nabla_\alpha \tilde{L}^{-1}(X'_s,q'_s,U'_s),\\
   U_T=\nabla_x u(T,X'_T,q'_T,\mathcal{L}(X_T|\mathcal{F}^0_T)).
\end{array}
\right.
\]
This means that $(X'_t,q'_t,\varphi'_t,U'_t,Z'_t,Z^{0,'}_t)_{t\in [0,T]}$ is a solution of the FBSDE \eqref{eq: true formulation mfgmp}. By Theorem \eqref{lemma: uniqueness lip sol} this implies that $(\phi,\alpha^0)$ and $(\phi',\alpha^{0,'})$ yield the same controls. 
\end{proof}
\begin{remarque}
Our uniqueness result may appear quite weak. Although we prove that two Lipschitz feedbacks $\phi,\phi'$ must yield the same controls for minor players we do not show that they are equal. This is a direct consequence of our definition of a Nash equilibrium, and in general we cannot expect a better result. Indeed if $(\phi,\alpha^0)$ is a Nash equilibrium, then any Lipschitz function $\phi'$ such that
\[\phi(t,x,q^{\alpha^0}_t,m^*_t(\alpha^0,\phi))=\phi'(t,x,q^{\alpha^0}_t,m^*_t(\alpha^0,\phi)) \quad a.e \text{ in } [0,T]\times \reels^d, \quad a.s\]
yield the same Nash equilibrium. Let us formally explain how to recover the uniqueness of a Lipschitz feedback whenever there exists a Lipschitz solution. Instead of fixing an initial condition $(X_0,q_0)$, we may define a pair of function $(\phi,\Phi^0)$ as a Nash equilibrium if $\Phi^0$ is measurable, $\phi$ is Lipschitz, and for any admissible initial condition $(X_0,q_0)$ the pair
\begin{gather*}
    (\alpha^0)_{t\in [0,T]}=\Phi^0(\mathcal{L}(X_0),q_0,W^0),\\
    (\alpha_t)_{t\in [0,T]}=(\alpha^*(\alpha^0)_t)_{t\in [0,T]},
\end{gather*}
is an optimal feedback for the major-minor system with initial condition $(X_0,q_0)$. Given the existence of a Lipschitz solution to \eqref{eq: true formulation mfgmp}, a variant of Theorem 1.33 of \cite{probabilistic-mfg} ensures the existence of such a mapping. The uniqueness of a pair $(\phi,\Phi^0)$ then follows from the fact that the argument we presented in the above proof now holds along any initial condition. 
\end{remarque}
\subsubsection{The case of additive common noise}
\label{subsection: additive common noise}
Suppose now that there is an additional additive common noise $(B^0_t)_{t\geq 0}$ a $\mathcal{F}^0-$Brownian motion. Given a correlation matrix 
\[\forall t>0, \quad \Gamma_{cor}=\text{Cor}(B^0_t,W^0_t)=Cov(B^0_1,W^0_1),\]
we may always assume that there exist $C_{cor}\in \mathcal{M}_d(\reels)$ such that 
\[C_{cor} C_{cor}^T=I_d-\Gamma_{cor}\Gamma_{cor}^T,\]
and a Brownian motion $(\tilde{B}^0_t)_{t\geq 0}$ independent of $(W^0_t)_{t\geq 0}$ such that 
\[\forall t\geq 0, \quad B^0_t=\Gamma_{cov}W^0_t+C_{cor}\tilde{B}^0_t.\]
This follows from the fact that 
\[I_d-\Gamma_{cor}\Gamma_{cor}^T,\]
is positive semi-definite and Cholesky decomposition. Let $v_0\geq 0$, we consider systems of the form 
 \begin{equation}
\label{eq: system common noise}
\left\{
\begin{array}{l}
     \displaystyle X_t=X_0-\int_0^t  F(X_s,q_s,U_s,\varphi_s,Z^\varphi_s,\mathcal{L}(X_s,U_s|\mathcal{F}^0_s))ds+\sqrt{2\sigma}B_t+\sqrt{2v^0}B^0_t, \\
     \displaystyle q_t=q_0-\int_0^t H_z(q_s,\varphi_s,Z^\varphi_s,\mathcal{L}(X_s,U_s|\mathcal{F}^0_s))ds+\sqrt{2\sigma^0}W^0_t,\\
    \displaystyle  U_t= U_T+\int_t^T G(X_s,q_s,U_s,\varphi_s,Z^\varphi_s,\mathcal{L}(X_s,U_s|\mathcal{F}^0_s))ds-\int_t^T
    Z_sd(B,W^0,B^0)_s,\\
\displaystyle \varphi_t= \varphi_T+\int_t^T L_H(q_s,\varphi_s,Z^\varphi_s,\mathcal{L}(X_s,U_s|\mathcal{F}^0_s))ds-\int_t^T Z^\varphi_sd(W^0,B^0)_s,\\
U_T=g(X_T,q_T,\mathcal{L}(X_T|\mathcal{F}^0_T)), \quad \varphi_T=\psi(q_T,\mathcal{L}(X_T|\mathcal{F}^0_T)).
\end{array}
\right.
\end{equation}
This is a more general version of MFGs with a major player and additive common noise. Introducing the change of variable 
\[q^0_t=\sqrt{2v^0}W^0_t, \quad \tilde{X}_t=W_t-q^0_t,\]
we consider the system
 \begin{equation}
\label{system with additive common noise}
\left\{
\begin{array}{l}
     \displaystyle \tilde{X}_t=X_0-\int_0^t  F'(\tilde{X}_s,q^0_s,q_s,U_s,\varphi_s,\tilde{Z}^\varphi_s,\mathcal{L}(\tilde{X}_s,U_s|\mathcal{F}^0_s))ds+\sqrt{2\sigma}B_t, \\
     \displaystyle q_t=q^0_0-\int_0^t H'_z(q_s,q^0_s,\varphi_s,\tilde{Z}^\varphi_s,\mathcal{L}(\tilde{X}_s,U_s|\mathcal{F}^0_s))ds+\sqrt{2\sigma^0}W^0_t,\\
     \displaystyle q^0_t=\sqrt{2v^0}\left(\Gamma_{cor} W^0_t+C_{cor}\tilde{B}^0_t\right)\\
    \displaystyle  U_t= U_T+\int_t^T G'(\tilde{X}_s,q^0_s,q_s,U_s,\varphi_s,\tilde{Z}^\varphi_s,\mathcal{L}(\tilde{X}_s,U_s|\mathcal{F}^0_s))ds-\int_t^T
    \tilde{Z}_sd(B,W^0,\tilde{B}^0)_s,\\
\displaystyle \varphi_t= \varphi'_T+\int_t^T L'_H(q_s,q^0_s,\varphi_s,\tilde{Z}^\varphi_s,\mathcal{L}(\tilde{X}_s,U_s|\mathcal{F}^0_s))ds-\int_t^T \tilde{Z}^\varphi_sd(W^0,\tilde{B}^0)_s,\\
U'_T=g'(X_T,q^0_T,q_T,\mathcal{L}(\tilde{X}_T|\mathcal{F}^0_T)), \quad \varphi'_T=\psi'(q_T,q^0_T,\mathcal{L}(\tilde{X}_T|\mathcal{F}^0_T)),
\end{array}
\right.
\end{equation}
with the notation 
\[F'(x,q^0,q,\varphi,m)=F(x+q^0,q,\varphi,f(q^0)_\#m ),\]
for the translation
\[f(q^0):(x,u)\mapsto (x+q^0,u),\]
and a similar definition for $g',\psi',L'_H,G',H'_z$. For $q^0_0=0$, this is exactly the system \eqref{eq: system common noise}.Consequently if there exists a decoupling field associated to \eqref{system with additive common noise}, then it can be related directly to the decoupling field of \eqref{eq: system common noise}, see \cite{common-noise-in-MFG} Lemma 5.3 for a rigorous argument. Letting $H^{extended}_z=(H'_z,0_{\reels^d})$, For $\Gamma_{cor}=0_{\mathcal{M}_{d\times d^0}(\reels)}$, this falls direcly into the class we have already considered, otherwise, the only difference with this system and the class we treated with Lipschitz solutions is the fact that the equation satisfied by $(\bar{q}_t)_{t\geq 0}=(q_t,q^0_t)_{t\geq 0}$ is of the form 
\[d\bar{q}_s=-H^{extended}_z(\bar{q}_s,\varphi_s,\tilde{Z}_s,\mathcal{L}(\tilde{X}_s,U_s))ds+\Sigma d(W^0,\tilde{B}^0),\]
for a non diagonal matrix 
\[\Sigma=\left(\begin{array}{cc}
    \sqrt{2\sigma^0}I_{d^0} &0\\
    \sqrt{2v^0} \Gamma_{cor} & \sqrt{2v^0}C_{cor}
\end{array}\right).\]
There is no particular difficulty in extending our result to \eqref{system with additive common noise}, since under the transformation we have introduced, coefficients are Lipschitz in $q^0$. We now make the following non degeneracy assumption 
\begin{hyp}
    \label{hyp: non degeneracy}
At least one of the following holds
\begin{enumerate}
    \item[-] $v^0=0$,
    \item[-] $I_d-\Gamma_{cor}\Gamma_{cor}^T$ is invertible
\end{enumerate}
\end{hyp}
This last assumption on the invertibility of $I_d-\Gamma_{cor}\Gamma_{cor}^T$ is truly a non degeneracy assumption. Consider the simple example for $d^0=d=1$, for which $\Gamma_{cor}=\rho\in [-1,1]$, this is equivalent to 
\[|\rho|<1.\]

\begin{lemma}
    Under Hypothesis \ref{hyp: non degeneracy} and if $\sigma^0>0$, there exists a matrix $\Sigma_{inv}\in \mathcal{M}_{d^0\times (d^0+d)}(\reels)$ depending only on $\Gamma_{cor}$ such that 
    \[\forall  x\in \reels^{d^0},y\in \reels^d,\quad \frac{1}{\sqrt{2\sigma^0}}\Sigma_{inv}\cdot \Sigma^T \left(\begin{array}{cc}
        x\\
        y
    \end{array}\right) =x\]
\end{lemma}
\begin{proof}
If $v^0=0$, then $\Sigma_{inv}=\left(\begin{array}{cc}
     I_{d^0} & 0
\end{array}
\right)$. Otherwise, the invertibility of 
\[I_d-\Gamma_{cor}\Gamma_{cor}^T\]
is equivalent to the invertibility of its square root $C_{cor}$, consequently we can take 
\[\Sigma_{inv}=\left(\begin{array}{cc}
    I_{d^0} & -\Gamma_{cor}C_{cor}^{-1}
\end{array}\right)\]
\end{proof}
Let us now observe that if \eqref{system with additive common noise}, admits a smooth decoupling field $(U,\varphi)$ then by Ito's lemma
\[\forall t\in [0,T], \quad \tilde{Z}^\varphi_t=\Sigma^T\left(\begin{array}{cc}
        \nabla_q \varphi\\
        \nabla_{q^0} \varphi
    \end{array}\right) (t,q_t,q^0_t,\mathcal{L}(X_t|\mathcal{F}^0_t)).\]
    Consequently as soon as the non degeneracy assumption, Hypothesis \eqref{hyp: non degeneracy} holds, the case of mean field games corresponds to taking 
    \[\bar{H}(q,q^0,\varphi,z,\mu)=H'(q,q^0,\varphi,\frac{1}{\sqrt{2\sigma^0}}\Sigma_{inv}\cdot z,\mu),\]
    up to a rescaling by $\sqrt{2\sigma^0}$, we can get back exactly a system of the form \eqref{mfg with major} and since $z\mapsto \Sigma_{inv}\cdot z$ is a Lipschitz function, this fits into the framework we have already developed. Formally this argument is a consequence of the fact that as soon as the covariance matrix of the common noise is not degenerate, up to a change of variables, the additive common noise can be seen as an uncontrolled part of the state of the major player. 
    \begin{remarque}
        It is possible to consider a non constant correlation through time 
        \[\forall t>0, \quad \Gamma^t_{cor}= Cor (B^0_t,W^0_t)=\frac{1}{t}Cov(B^0_t,W^0_t),\]
        so long as
        \[\exists \kappa>0, \quad \forall t\in (0,T), \quad \|\left(C_{cor}^t\right)^{-1}\|\leq \kappa ,\]
        where $\|\cdot \|$ indicates any norm on $\mathcal{M}_d(\reels)$ and 
        \[C_{cor}^t=\sqrt{I_d-\Gamma^t_{cor}\left(\Gamma^t_{cor}\right)^T}\]
        . Indeed in this case 
        \[(t,q,q^0,\varphi,z,\mu)\mapsto H'(q,q^0,\varphi,\Sigma_{inv}^t \cdot z,\mu),\]
        is Lipschitz uniformly in time and the extension is straighforward. 
    \end{remarque}

\section{Global wellposedness of Lipschitz solutions to the master system}
\subsection{General study under Hilbertian monotonicity}
The long time analysis of \eqref{mfg with major} relies on a regularization by common noise argument. Altough the setting and method of proof are vastly different, the result is analogous to \cite{minor-major-delarue}, namely if the volatility associated to the common noise is sufficiently large compared to the horizon of the game, we can obtain estimates which are sufficient to conclude to the existence and uniqueness of a strong solution to the system \eqref{mfg with major}. Under sufficiently strong assumptions we also show that there exists a volatility threshold independent of the time interval considered, such that if the volatility of the major player is above this threshold, existence and uniqueness hold over arbitrarily long time intervals.

Letting $H_{major}$ indicates the Hamiltonian of the major player, throughout this section the following assumption will be in force.
\begin{hyp}
    \label{hyp: coefficients mfgmp}
\begin{enumerate}
    \item[-] $F,G$ do not depend on $\varphi$ and there exists a constant $\lambda \geq 0$ and a function $H$ such that 
\[
\forall (\mu,q,\varphi,z)\in \mathcal{P}_2(\reels^{2d})\times \reels^{2d^0+1}, \quad H_{major}(q,\varphi,z,\mu)=\lambda \varphi+H^0(q,z,\mu).
\]
    \item[-] Hypothesis \ref{hyp: locally lipschitz in z} holds for $(\psi,g,D_zH^0,L_{major},F,G)$, with
    \[\forall (q,z,m)\in \reels^{2m}\times \mathcal{P}_2(\reels^{2d}), \quad  L_{major}(q,\varphi,z,m)=H_{major}(q,\varphi,z,m)-D_zH_{major}(q,\varphi,z,m)\cdot z.\]
\item[-] there exists a constant $C_{H}$ and a power $\delta \in [0,1]$ such that 
\begin{gather*}
 \forall (\mu,q),(\nu,q')\in \mathcal{P}_2(\reels^{2d})\times \reels^{m}, \forall z\in \reels^{d^0}\\
 |(H^0-D_z H^0\cdot z)(q,z,\mu)-(H^0-D_z H^0\cdot z)(q',z,\nu)|\leq C_H(1+|z|^\delta)\left(|q-q'|+\mathcal{W}_2(\mu,\nu) \right).
\end{gather*}
\end{enumerate}
\end{hyp}
Under this assumption, the function $L_{major}$ is such that 
\[L_{major}(q,\varphi,z,\mu)=\lambda \varphi+L_H(q,z,\mu),\]
with 
\[L_H(q,z,\mu)= H^0(q,z,\mu)-D_z H^0(q,z,\mu)\cdot z.\]
While the assumptions on the Hamiltonian of the major player $H^0$ can be relaxed, it makes the statement of our main result easier to read. The fact that coefficients depend on $\varphi$ through a discount term only is also quite natural in the context of optimal control. The second assumption on the regularity of $H^0$ is more restrictive but still holds for a wide variety of Hamiltonians.
\begin{exemple}
Let $H^0$ be of the form 
\[H^0(q,z,\mu)= b(q,\mu)\cdot z+h(q,\mu,z),\]
for a Lipschitz function $b$, and a function $h$ such that 
\[\exists C>0,\quad \forall z\in \reels^{d^0},\quad \|h(\cdot,z)\|_{Lip}\leq C(1+|z|^\delta),\quad  \|\nabla_z h(\cdot,z)\|_\infty\leq C(1+\omega(|z|)),\]
and 
\[\|\nabla_z h(\cdot,z)\|_{Lip}\leq \frac{C}{1+|z|^{1-\delta}},\]
for some $\delta\in [0,1]$ and an increasing function $\omega:\reels^+\to \reels^+$. Then $H^0$ satisfies Hypothesis \ref{hyp: coefficients mfgmp} for this particular $\delta$. This includes any power Hamiltonian 
\[h\equiv |z|^p,\]
for some $p\geq 1$, or more generally any Hamiltonian such that both $H^0,D_zH^0$ are Lipschitz in $(q,\mu)$ uniformly in $z$. 
\end{exemple}
Now that we have introduced our regularity assumption, we also make an assumption on the monotonicity of coefficients
\begin{hyp}
\label{hyp: monotonicity mfgmp}
There exists a symmetric matrix $A\in M_m(\reels)$, a constant $\beta_0>0$ such that 
\begin{gather}
\nonumber \forall (X,q),(Y,q')\in \mathcal{H}^{d}\times \reels^{d^0},\\
\label{eq: monotonicity on initial condition} \langle g(X,q,\mathcal{L}(X))-g(Y,q',\mathcal{L}(Y)),X-Y\rangle+ \frac{1}{2}(q-q')\cdot A(q-q')\\
\nonumber \geq \beta_0 \|\psi(q,\mathcal{L}(X))-\psi(q',\mathcal{L}(Y))\|^2.
\end{gather}
Moreover there also exists a constant $\kappa>0$ such that
\begin{gather}
   \nonumber \forall (X,U,q,\varphi),(Y,V,q',\varphi')\in \mathcal{H}^{2d}\times \reels^{d^0+1}, \quad z\in \reels^{d^0}\\
   \label{eq: monotonicity on coefficients}\langle \Delta (\tilde{G},\tilde{F},AD_zH^0),\Delta \bar{X}\rangle\geq \kappa \|\Delta \bar{X}\|^2 
\end{gather}
where 
\[\Delta (\tilde{G},\tilde{F},AD_zH^0)=\left(\begin{array}{c}
\tilde{G}(X,U,q,z)-\tilde{G}(Y,V,q',z)\\
\tilde{F}(X,U,q,z)-\tilde{F}(Y,V,q',z)\\
AD_z\tilde{H}^0(X,U,q,z)-AD_z\tilde{H}^0(Y,V,q',z)
\end{array}\right) \text{ and }  \Delta \bar{X}=\left(\begin{array}{c}
X-Y\\
U-V\\
q-q'
\end{array}\right). \]
\end{hyp}
This monotonicity assumption is inspired by previous works on mean field games with an external source of noise \cite{noise-add-variable,common-noise-in-MFG}, though those works did not address controlled dynamics (i.e. a major player). In fact it is the exact same monotonicity assumption if we ignore the dependency on the variable $z$. By itself \eqref{eq: monotonicity on initial condition} is not such a strong assumption. For example it is satisfied for $A=cI_m$ for some $c\geq 0$ as soon as $g$ is strongly monotone in $X\in H$, which is not an unusual assumption in the literature on displacement monotone MFGs \cite{disp-monotone-1,disp-monotone-2} and related transport equations \cite{monotone-sol-meynard}. The second part however \ref{eq: monotonicity on coefficients} has strong implications for the dynamics. Indeed it requires the drift of the major player $D_z H^0$ to be monotone in its state $q$. 
\begin{exemple}
A typical exemple in which the drift of the major player is monotone in $q$ is whenever its Hamiltonian is of the form 
\[H^0(q,z,m)=b(q,m)\cdot z+ H^0(z,m),\]
with a function $b$ monotone in $q$. This corresponds to a controlled process of the form 
\[dq^\alpha_t=-(b(q^\alpha_t,m_t)+\alpha_t)dt+\sqrt{2\sigma^0}dW^0_t.\]

\end{exemple}
In particular, if \eqref{eq: monotonicity on initial condition} holds for $\beta_0=0$ and the triple $(F,G,b)$ satisfies \eqref{eq: monotonicity on coefficients}, then for any admissible initial condition $(X,q)$ and any adapted process $(Z^0_t)_{t\in [0,T]}$ such that 
\[\esp{\int_0^T\left(|F(0,Z^0_s)|^2+|G(0,Z^0_s)|^2+|D_zH^0(0,Z^0_s)|^2\right)ds}<+\infty,\]
there exists a unique strong solution to the uncontroled dynamics
\[
\left\{
\begin{array}{l}
   dX_t=-F(X_t,q_t,U_t,Z^0_t,\mathcal{L}(X_t,U_t|\mathcal{F}^0_t))dt+\sqrt{2\sigma}dB_t, \quad X_0=X\\
   dq_t= -D_zH^0(q_t,Z^0_t,\mathcal{L}(X_t,U_t|\mathcal{F}^0_t))dt+\sqrt{2\sigma^0}dW^0_t, \quad q_0=q\\
   dU_t= G(X_t,q_t,U_t,Z^0_t,\mathcal{L}(X_t,U_t|\mathcal{F}^0_t))dt+Z_td(B_t,W^0_t), \quad U_T=g(X_T,q_T,\mathcal{L}(X_t\mathcal{F}^0_t))\\
   \mathcal{F}^0_t=\sigma(q,(W^0_s)_{s\leq t}),
\end{array}
\right.
\]
where 
\[F(0,z)=F(0,0,0,z,\delta_{0_{\reels^{2d}}}).\]
This follows easily by adapting results from \cite{common-noise-in-MFG}. Formally this ensures the wellposedness of mean field games with a major player for open loop controls.
\begin{lemma} 
    Under Hypotheses \ref{hyp: coefficients mfgmp} and \ref{hyp: monotonicity mfgmp}, there exists a matrix $A\in \mathcal{M}_{d^0}(\reels)$, constants $\kappa,\beta>0, C_M\geq 0$ and an increasing function $K:\reels^+\to \reels^+$ such that the following monotonicity assumption is satisfied
\begin{gather}
   \nonumber \forall (X,U,q,z),(Y,V,q',z')\in \mathcal{H}^{2d}\times \reels^{2d^0+1},\\
   \label{eq: monotonicity on coefficients with z}\langle \Delta (\tilde{G},\tilde{F},AD_zH^0),\Delta \bar{X}\rangle+ (C_M+K(|z|\wedge |z'|))|z-z'|^2\geq \kappa \|\Delta \bar{X}\|^2 
\end{gather}
where 
\[\Delta (\tilde{G},\tilde{F},AD_zH^0)=\left(\begin{array}{c}
\tilde{G}(X,U,q,z)-\tilde{G}(Y,V,q',z')\\
\tilde{F}(X,U,q,z)-\tilde{F}(Y,V,q',z')\\
AD_z\tilde{H}^0(X,U,q,z)-AD_z\tilde{H}^0(Y,V,q',z')
\end{array}\right) \text{ and }  \Delta \bar{X}=\left(\begin{array}{c}
X-Y\\
U-V\\
q-q'
\end{array}\right), \]
\end{lemma}
\begin{proof}
Since coefficients are locally Lipschitz in the variable $z$ and Lipschitz in the other variables, it suffices to take any $A$ for which Hypothesis \ref{hyp: monotonicity mfgmp} holds, and use Cauchy Schwarz inequality. 
\end{proof}
This is the notion of monotonicity that we will use for the long time analysis of \eqref{mfg with major}. 
\begin{thm}
\label{thm: existence for sigma T mfgmp}
Under Hypotheses \ref{hyp: coefficients mfgmp} and \ref{hyp: monotonicity mfgmp}, for any $T>0$ there exists a $\sigma^0_T$ depending on $T,\kappa,C_{coef},\omega,\beta_0,C_H,\delta$ such that for any $\sigma^0>\sigma^0_T$ there exists a unique Lipschitz solution to \linebreak MFGMP$(T,\psi,g,D_z H^0,L_{major},F,G)$ on $[0,T]$. 
\end{thm}
\begin{proof}
We already know that under Hypothesis \ref{hyp: coefficients mfgmp}, there exits a maximal time $T_c\leq T$ such that there exists a Lipschitz solution $(U,\varphi)$ to MFGMP$(T,\psi,g,D_z H^0,L_{major},F,G)$ on $[0,T_c]$ by Theorem \ref{thm: existence mfgmp}. Let us assume by contradiction that $T_c<T$. 

\noindent\textit{Step 1: propagatation of monotonicity on small time intervals}

Let $t\leq T_c$, we fix $(X^1,q^1)$, $(X^2,q^2)$ two set of admissible initial conditions and for $i=1,2$, let $(X^i_s,q^i_s,U^i_s,\varphi^i_s,Z^i_s,Z^{\varphi,i}_s)_{s\in [0,t]}$ be the associated solution to \eqref{eq: def lip sol system}. We now introduce the process $(V_s)_{s\in [T-t,T]}$ defined by
\begin{equation}
    \label{eq: prop monotonicity mfgmp}
    \forall s\in [T-t,T], \quad V_s= \langle U^1_s-U^2_s,X^1_s-X^2_s\rangle+\frac{1}{2}\langle q^1_s-q^2_s,A(q^1_s-q^2_s)\rangle-\beta(T-s)|\varphi^1_s-\varphi^2_s|^2, \end{equation}
for $\beta:\reels\to \reels$ a smooth function of time satisfying $\beta(0)\leq \beta_0$. In particular, let us remark that by \eqref{eq: monotonicity on initial condition} of Hypothesis \ref{hyp: monotonicity mfgmp}, 
\[\espcond{V_T}{\mathcal{F}^0_T}\geq 0 \quad a.s.\]
We now apply Ito's Lemma to $(V_s)_{s\in [T-t,T]}$.
\begin{align*}
\esp{V_T}=&\esp{V_0-\int_{T-t}^T  (\Delta (G,F,AD_zH^0)_s)\cdot(\Delta (X,U,q)_s)ds-\int_{T-t}^T \sigma^0 \beta(t-s)|Z^{\varphi,1}_s-Z^{\varphi,2}_s|^2ds}\\
&+\esp{\int_{T-t}^T \left(\frac{d\beta}{ds}-2\lambda \beta\right)(T-s)|\varphi^1_s-\varphi^2_s|^2ds-\int_{T-t}^T 2\beta(T-s)({L_H}^1_s-{L_H}^2_s)(\varphi^1_s-\varphi^2_s)ds},
\end{align*}
where we have used the notation $\Delta (X,U,q)_s=(X^1_s-X^2_s,U^1_s-U^2_s,q^1_s-q^2_s)$, the definition of $\Delta (G,F,AD_zH^0)_s$ being analogous. Also for $i=1,2$, \[{L_H}^i_s=L_H(q^i_s,\varphi^i_s,Z^{\varphi,i}_s,\mathcal{L}(X^i_s,U^i_s|\mathcal{F}^{0,i}_s)).\]
Using \eqref{eq: monotonicity on coefficients with z}, and taking the expectation of conditional expectation with respect to $\mathcal{F}^0_{T-t}$, we deduce that
\begin{gather*}-\esp{\int_{T-t}^T  (\Delta (G,F,AD_zH^0)_s)\cdot(\Delta (X,U,q)_s)ds}\\
    \leq \esp{\int_{T-t}^T\left(C_M+K(|Z^{\varphi,1}_s|\wedge |Z^{\varphi,2}_s|)|Z^{\varphi,1}_s-Z^{\varphi,2}_s|^2-\kappa |\Delta (X,U,q)_s|^2 \right)ds }.
\end{gather*}
On the other hand using Hypothesis \ref{hyp: coefficients mfgmp} on $H^0$, for any $\gamma>0$
\begin{gather*}\esp{\int_{T-t}^T \left(\frac{d\beta}{ds}-2\lambda \beta\right)(t-s)|\varphi^1_s-\varphi^2_s|^2ds-\int_{T-t}^T \beta(T-s)({L_H}^1_s-{L_H}^2_s)(\varphi^1_s-\varphi^2_s)ds}\\
    \leq \esp{\int_{T-t}^T \left(\frac{d\beta}{ds}(T-s)-(2\lambda-\gamma) \beta(T-s) \right)|\varphi^1_s-\varphi^2_s|^2ds+\int_{T-t}^T \frac{1}{\gamma}\beta(T-s)\left(C^2_H(1+|Z^{\varphi,1}_s|^{2\delta})|\Delta (X,U,q)_s|^2\right)ds}\\
    +\esp{\int_{T-t}^T \frac{1}{\gamma}\beta(T-s)\omega^2(|Z^{\varphi,1}_s|\wedge |Z^{\varphi,2}_s|)|Z^{\varphi,1}_s-Z^{\varphi,2}_s|^2ds }.
\end{gather*}
Fixing $\gamma$ and letting $\beta(s)=\beta_0 e^{(2\lambda-\gamma)s}$, it follows that 
\begin{align*}
    \esp{V_T}\leq& \esp{V_{T-t}}+\esp{\int_{T-t}^T \left(\frac{1}{\gamma}\beta(T-s)\left(C^2_H(1+|Z^{\varphi,1}_s|^{2\delta})-\kappa \right)|\Delta (X,U,q)_s|^2\right)ds}\\
    &\esp{\int_{T-t}^T \left(\frac{1}{\gamma}\beta(T-s)\omega^2(|Z^{\varphi,1}_s|\wedge |Z^{\varphi,2}_s|)+C_M+K(|Z^{\varphi,1}_s|\wedge |Z^{\varphi,2}_s|)-\sigma^0 \beta(t-s)\right)|Z^{\varphi,1}_s-Z^{\varphi,2}_s|^2ds}.
\end{align*}
Let us assume that 
\begin{equation}
    \label{eq: local in time sigma greater}
    \sigma^0>\frac{1}{\gamma}\omega^2(\|\nabla_q \psi\|_\infty )+\frac{1}{\beta_0}\left(C_M+K((\|\nabla_q \psi\|_\infty )\right).
\end{equation}
By Lemma \ref{lemma: local lipschitz estimate}, on a sufficiently small interval of time $[0,t^*]$ there exists a constant $C$ depending on the coefficients such that 
\[\forall s\leq t^*, \quad \|\nabla_q \varphi(s,\cdot)\|_\infty \leq \|\nabla_q \psi(\cdot)\|_\infty+Cs.\]
By Lemma \eqref{lemma: Z bound by lip}, it follows that for any $\gamma>0$, there exists a time $t^*_\gamma$ depending on the coefficients $\beta_0$ and $\gamma$ such that for $t\leq t^*_\gamma$
\[\esp{\int_{T-t}^T \left(\frac{1}{\gamma}\beta(T-s)\omega^2(|Z^{\varphi,1}_s|\wedge |Z^{\varphi,2}_s|)+C_M+K(|Z^{\varphi,1}_s|\wedge |Z^{\varphi,2}_s|)-\sigma^0 \beta(t-s)\right)|Z^{\varphi,1}_s-Z^{\varphi,2}_s|^2ds}\leq 0.\]
Choosing $\gamma$ sufficiently large, and $t^*_\gamma$ sufficiently small, we may also assume that for $t\leq t^*_\gamma$
\[\esp{\int_{T-t}^T \left(\frac{1}{\gamma}\beta(T-s)\left(C^2_H(1+|Z^{\varphi,1}_s|^{2\delta})-\kappa \right)|\Delta (X,U,q)_s|^2\right)ds}\leq 0.\]
For any $t\in [0,t^*_\gamma]$, 
\[\esp{V_{T-t}}\geq 0.\]
Since this is true for any initial condition. Taking $X^1=X^2$ and a deterministic initial condition for the major player
\[\forall (t,q^1,q^2,\mu)\in[0,t^*_h]\times \reels^{2m}\times \mathcal{P}_2(\reels^d),\quad \beta(t)|\varphi(t,q^1,\mu)-\varphi(t,q^2,\mu)|^2\leq (q^1-q^2)\cdot A(q^1-q^2).\]
Naturally, 
\[\forall t\in [0,t^*_\gamma], \quad \| \nabla_q\varphi(t,\cdot)\|_\infty\leq \sqrt{\frac{\|A\|}{\beta(t)}}.\]
Combining this estimate with Lemma \ref{lemma: Z bound by lip}, for any $t\leq t^*_\gamma$ and any initial conditions $(X^1,q^1), (X^2,q^2)$

\begin{gather*}\esp{\int_{T-t}^T \left(\frac{1}{\gamma}\beta(t-s)\left(C^2_H(1+|Z^{\varphi,1}_s|^{2\delta})-\kappa \right)|\Delta (X,U,q)_s|^2\right)ds}\\
    \leq \esp{\int_{T-t}^T \left(\frac{1}{\gamma}\left(C^2_H(\|A\|+\beta(T-s))-\kappa \right)|\Delta (X,U,q)_s|^2\right)ds}.\end{gather*}
In light of this new estimate, we may choose
\[\gamma^*= \frac{2}{\kappa}C^2_H(\|A\|+\beta_0),\]
and 
\[\beta^*(t)=\beta_0 e^{(2\lambda-\gamma^*)t},\]
independent of $\sup_{[0,t^*_\gamma]}\|(U,\varphi)(t,\cdot)\|_{Lip}$.

\noindent\textit{Step 2: extension to the interval of definition}

We now define 
\[\sigma^0_T=\frac{1}{4\gamma^*}\omega^2\left(\sqrt{\frac{\|A\|}{\beta^*(T)}}\right)+\frac{1}{\beta^*(T)}\left(C_M+K\left(\sqrt{\frac{\|A\|}{\beta^*(T)}}\right)\right),\]
and we assume that $\sigma^0>\sigma^0_T$. First we notice that $(U,\varphi)(t^*_\gamma,\cdot)$ satisfy \eqref{eq: monotonicity on initial condition} for $\tilde{\beta}_0=\beta^*(t^*_h)$. Since $\sigma^0>\sigma^0_T$, in particular we also have $\sigma^0>\sigma^0_t$ and \eqref{eq: local in time sigma greater} holds for $\sigma^0,\varphi(t^*_\gamma,\cdot)$. Clearly, we can repeat Step 1, and choose the same $\gamma^*$ since it was independent of $(U,\varphi)$. Since this is true as long as $\|(U,\varphi)(t,\cdot)\|_{Lip}<+\infty$, we deduce that for any $t\leq T_c$ and any initial conditions $(X^1,q^1), (X^2,q^2)$, 
\begin{equation}
\label{eq: prop monotonicity mfgmp 2}
    \forall s\in [0,t],\quad \esp{V_{T-s}}\geq \esp{V_{T}}\geq 0.\end{equation}

\noindent\textit{Step 3: regularity from monotonicity}
From \eqref{eq: prop monotonicity mfgmp 2}, we already have the following estimate 
\[\forall s\in [T-t,T],\quad \|\varphi^1_s-\varphi^2_s\|^2\leq \frac{\|A\|}{\beta(T-s)}\|q^1_s-q^2_s\|^2+\frac{1}{\beta(T-s)}\|X^1_s-X^2_s\|\|U^1_s-U^2_s\|.\]
Consequently, it is sufficient to show that we can get Lipschitz estimates uniformly on $[0,T]$. To that end, we now make use of the strong monotonicity of coefficients. Indeed since $\sigma^0>\sigma^0_T$ and thanks to our choice of $\gamma^*$ there exists a constant $c_\kappa$ such that 
\[\esp{V_T}\leq \esp{V_{T-t}-c_\kappa\int_{T-t}^T |\Delta (X,U,q,Z^{\varphi})_s|^2ds }.\] 
It follows that 
\begin{equation}
\label{eq: estimate from strong monotonicity mfgmp}
    c_\kappa\esp{\int_{T-t}^T |\Delta (X,U,q,Z^{\varphi})_s|^2ds}\leq \|A\|\|q^1_0-q^2_0\|^2+\|X^1_0-X^2_0\|\|U^1_0-U^2_0\|.
\end{equation}
Using the regularity of coefficients as well as the bound we already have on $\|\nabla_q\varphi\|_\infty$, we deduce that 
\[\forall s\in [T-t,T], \quad \|(X^1_s,q^1_s)-(X^2_s,q^2_s)\|^2\leq \left(1+C_{coef}+\omega\left(\sqrt{\frac{\|A\|}{\beta(T-s)}}\right)\right)\esp{\int_{T-t}^s |\Delta (X,U,q,Z^\varphi)_u|^2 du},\]
which in turns implies 
\[\|(X^1_{T-t},q^1_{T-t})-(X^2_{T-t},q^2_{T-t})\|^2\leq C\left(\|q^1_0-q^2_0\|^2+\|X^1_0-X^2_0\|\|U^1_0-U^2_0\|\right),\]
for some constant $C$ independent of $t\in [0,T]$.
We may now estimate the difference of the backward processes in a similar fashion and conclude that 
\[\| U_0^1-U_0^2\|^2\leq C'(\|q^1_0-q^2_0\|^2+\|X^1_0-X^2_0\|^2),\]
for another constant independent of $t$. It follows that the pair $(U,\varphi)$ is lift Lipschitz, that is for any $t<T_c$, the function 
\[(X,q)\mapsto (U,\varphi)(t,X,q,\mathcal{L}(X)),\]
is Lipschitz for the Hilbertian norm $\|\cdot\|$. Then since we have the estimate \eqref{eq: estimate from strong monotonicity mfgmp} on the stability of $Z^\varphi$ with respect to initial conditions, we may conclude that $U$ is also Lipschitz on $\mathcal{P}_2(\reels^d)$ with a constant independent of $t\in [0,T]$ by going back to the system \eqref{eq: caracteristics in x fbsde with major} as we did in the proof of Lemma \ref{lemma: local lipschitz estimate}. 

Since the Lipschitz estimates we presented are valid on $[0,T_c)$ with an uniform constant, this contradict 
\[\limsup_{t\to T_c}\|(U,\varphi)(t,\cdot)\|_{Lip}=+\infty. \]
and we deduce that $T_c=T$. Finally since this estimate is valid up until $T$ included, naturally
\[\limsup_{t\to T}\|(U,\varphi)(t,\cdot)\|_{Lip}<+\infty, \]
and a Lipschitz solution can be constructed on the closed interval $[0,T]$ by Theorem \ref{thm: existence mfgmp}.
\end{proof}
\begin{remarque}
    Although we do not treat the case of a dynamics with an additive common noise $(B^0_t)_{t\geq 0}$ directly, there is no particular difficulty in extending our result to this setting. Indeed, under Hypothesis \ref{hyp: non degeneracy}, the coefficients still satisfy Hypotheses \ref{hyp: coefficients mfgmp} and \ref{hyp: monotonicity mfgmp} under the change of variables introduced in Section \ref{subsection: additive common noise}, uniformly in the new variable $q^0$. The above proof can be adapted by taking the same initial condition \[q^{0,1}_0=q^{0,2}_0=q^0_0,\]
    when comparing two solutions with initial conditions $(X^i,q^i,q^{0,i})$ for $i=1,2$. Since
    \[\forall t\geq 0, \quad q^{0,1}_t=q^{0,2}_t \quad a.s,\]
    for any such initial condition, the only difference is that $\sigma^0_T$ now also depends on $\left\|\left(I_d-\Gamma_{cor}\Gamma_{cor}^T\right)^{-1}\right\|$ for $\Gamma_{cor}$ the correlation matrix between the two common noises.
\end{remarque}
In the above proof, $\sigma^0_T$ depends on $T$ only through the function $t\mapsto \beta(t)$ such that 
\begin{gather}
\nonumber\forall t\in [0,T],\quad  \forall (X,q),(Y,q')\in \mathcal{H}^{d}\times \reels^{d^0},\\
\label{eq: monotonicity with beta t} \langle U(t,X,q,\mathcal{L}(X))-U(t,Y,q',\mathcal{L}(Y)),X-Y\rangle+ \frac{1}{2}(q-q')\cdot A(q-q')\\
\nonumber \geq \beta(t) \|\varphi(t,q,\mathcal{L}(X))-\varphi(t,q',\mathcal{L}(Y))\|^2.
\end{gather}
Consequently, if we can find a constant function $\beta$ independent of $T$ such that \eqref{eq: monotonicity with beta t} holds on $[0,T]$, then $\sigma^0_T$ can also be chosen independently of $T$. 
\begin{thm}
    \label{thm sigma independent of T}
Under Hypotheses \ref{hyp: coefficients mfgmp} and \ref{hyp: monotonicity mfgmp}, if either 
\begin{enumerate}
\item[-] the Hamiltonian is such that $\lambda\geq \frac{\gamma^*}{2}$,
\item[-] the Hamiltonian is such that $\lambda\geq 0$ and $\delta<1$,
\end{enumerate}
then there exists a $\sigma^0_*$ depending on $\kappa,C_{coef},\omega,\beta_0,C_H,C_M,\delta,\lambda$ but independent of a time horizon, such that for any $\sigma^0>\sigma^0_*$ and for any $T>0$ there exists a unique Lipschitz solution to \linebreak   MFGMP$(T,\psi,g,D_z H^0,L_{major},F,G)$ on $[0,T]$
\end{thm}
\begin{proof}
First if $\lambda\geq \frac{\gamma^*}{2}$ then following the proof of Theorem \ref{thm: existence for sigma T mfgmp}, we may take $\beta\equiv \beta_0$ independent of time in which case, for any $\sigma^0>\sigma^0_*$ defined by
\[\sigma^0_*=\frac{1}{2\lambda }\omega^2\left(\sqrt{\frac{\|A\|}{\beta_0}}\right)+\frac{1}{\beta_0}\left(C_M+K\left(\sqrt{\frac{\|A\|}{\beta_0}}\right)\right),\]
there exists a solution to MFGMP$(T,\psi,g,D_z H^0,L_{major},F,G)$ on $[0,T]$ for any $T>0$.

On the other hand, if we assume that $\delta<1$ and that $\lambda>0$, we fix $\gamma=\lambda$ and observe that
\[ \frac{1}{2\lambda}\beta(0)\left(C^2_H(1+\left|\sqrt{\frac{\|A\|}{\beta(0)}}\right|^{2\delta}\right)-\kappa \leq \frac{C^2_H}{2\lambda}(\beta(0)+\beta(0)^{1-\delta}\sqrt{\|A\|})-\kappa. \]
In particular there exists a $\beta(0)=\beta_\kappa$ sufficiently small such that 
\[ \frac{1}{4\lambda }\beta_\kappa \left(C^2_H(1+\left|\sqrt{\frac{\|A\|}{\beta_\kappa }}\right|^{2\delta}\right)\leq \frac{\kappa}{2}.\]
Letting $\beta^*=\min(\beta_\kappa,\beta_0)$, we deduce that for any $\sigma^0>\sigma^0_*$ defined by 
\[\sigma^0_*=\frac{1}{2\lambda }\omega^2\left(\sqrt{\frac{\|A\|}{\beta^*}}\right)+\frac{1}{\beta^*}\left(C_M+K\left(\sqrt{\frac{\|A\|}{\beta^*}}\right)\right),\]
there exists a Lipschitz solution on $[0,T]$ for any $T>0$.

Finally we treat the case $\lambda=0$. Let us first assume that there exists a constant function $\beta(t)\equiv \beta$ and a time horizon $T_\beta$ such that there exists a Lipschitz solution to MFGMP$(T,\psi,g,D_z H^0,L_{major},F,G)$ on $[0,T_\beta]$ with
\[\forall t\in [T-T_\beta,T], \quad \esp{V_t}\geq 0,\]
where $(V_t)_{t\in [T-T_\beta,T]}$ is defined as in \eqref{eq: prop monotonicity mfgmp} for any set of initial conditions. Fixing two such initial conditions $(X^1,q^1), (X^2,q^2)$ and defining the associated processes as in the proof of Theorem \ref{thm: existence for sigma T mfgmp}, by definition of $(V_t)_{t\in [T-T_\beta,T]}$, 
\[\forall t\in [T-T_\beta,T], \quad \beta(t)\|\varphi^1_t-\varphi^2_t\|^2\leq \|U^1_t-U^2_t\|\|X^1_t-X^2_t\|+\|A\|\|q^1_t-q^2_t\|^2.\]
From this estimate, it follows that for any $\beta'\leq \beta$
\begin{gather*}
2\!\beta' \esp{\int_{T-t}^T \!(L^1_{H_s}\!-\!L^2_{H_s})(\varphi^1_s\!-\!\varphi^2_s)ds}\\
\leq 2\sqrt{\beta'} \displaystyle \int_{T-t}^T\|L^1_{H_s}-L^2_{H_s}\|\left(\|U^1_s-U^2_s\|+\|X^1_s-X^2_s\|+\sqrt{\|A\|}\|q^1_s-q^2_s\|\right)ds\\
\leq 2\sqrt{\beta'}\displaystyle \int_{T-t}^T\left(\!\underbrace{C_H\left(2+\!\frac{\|A\|^\frac{\delta}{2}}{\sqrt{\beta'}^\delta}\right)\|\Delta(X,U,q)_s\|^2}_I\!+\!\underbrace{\left(1+2\omega^2\left(\frac{\|A\|^\frac{\delta}{2}}{\sqrt{\beta'}^\delta}\right)\right)\|Z^{\varphi,1}_s-Z^{\varphi,2}_s\|^2}_{II}\right)ds.
\end{gather*}
For $\beta'$ sufficiently small, the term $I$ can be controlled by the strong monotonicity of $(F,G,AD_zH^0)$, and $II$ by $\sigma^0$ provided it is sufficiently large before said $\beta'$, as in the previous case. Then this a priori estimate is made rigorous by following the proof of Theorem \ref{thm: existence for sigma T mfgmp}.
\end{proof}
\begin{remarque}
    Clearly, in both cases, this corollary relies on the fact that under those conditions $\|\nabla_q\varphi(t,\cdot)\|_\infty$ is bounded uniformly in time. The first condition 
    \[\lambda>\frac{\gamma^*}{2},\]
    is less interesting in our opinion since it implies that we can take 
    \[\beta(t)=e^{(2\lambda-\gamma^*)t}.\]
    Letting $(U^T,\varphi^T)$ be the unique solution to MFGMP$(T,\psi,g,D_z H^0,L_{major},F,G)$ on $[0,T]$, we deduce that 
    \[\lim_{T\to \infty}\|\nabla_p\varphi^T(0,\cdot)\|_{\infty}=0.\]
    In optimal control, the presence of a term $\lambda \varphi$ in the Hamiltonian corresponds to a control problem with discounting.
    In this case the discount is so significant that the major player is not interested in optimizing over long intervals of time. 
    This does not seem to be the case with the second condition, since $\beta$ must stay sufficiently small. Since we have presented estimates independent of $\lambda\geq 0$ it also becomes possible to study the ergodic limit of the system in this case, although we leave this extension to future studies. 
    \end{remarque}
    
    \begin{exemple}
The second condition holds for any Hamiltoninan of the form 
    \begin{equation}
        \label{eq: exemple of hamiltonian}
        H_{major}(q,\varphi,\mu,z)=\lambda \varphi+b(q,\mu)\cdot z+c(q,\mu)+f(z),
    \end{equation}
    for Lipschitz functions $b,c$ a $\lambda \geq 0$ and a locally Lipschitz function $f$. Although the growth condition
    \[\|L_H(\cdot,z)\|_{Lip}\leq C(1+|z|^\delta),\]
    for some $\delta<1$ is quite restrictive, it is still possible to have a non linear interaction term $h$ depending on $(q,\mu)$. For example it is still satisfied if we add a term of the form 
    \[h(q,\mu,z)=f(q,\mu)(1+|z|)^\eta,\]
    with a Lipschitz function $f$ and $\eta\in (0,1)$.
    Moreover, if the Hamiltonian is of the form \eqref{eq: exemple of hamiltonian}, and the triple $(F,G,b)$ satisfies \eqref{eq: monotonicity on coefficients}, then Hypothesis \ref{hyp: monotonicity mfgmp} holds for this Hamiltonian. We refer to \cite{common-noise-in-MFG} for examples of triples $(F,G,b)$ satisfying this monotonicity condition. Finally, let us also mention that this corollary can be adapted to the case $\lambda=0$ and $\delta=1$ although this leads to a very strong monotonicity condition. Indeed, in addition to $\sigma^0_*$ there exists a $\kappa_*$ such that \eqref{eq: monotonicity on coefficients} must hold with $\kappa>\kappa^0$. Since we do not expect this condition to be often satisfied, we do not delve too much into it. 
        \end{exemple}
Let us end this section by insisting that it is unclear to us whether further regularity of the solution $(U,\varphi)$ follows naturally in this setting, even should coefficient be smooth. Consequently, the framework of Lipschitz solutions is not just convenient but essential in our approach. 
\subsection{Application to mean field games}
We just presented a general result on the existence of a Lipschitz solution associated to FBSDEs of the form \eqref{mfg with major}. We now explain how it applies to mean field games, going back to our original motivation of solving \eqref{eq: master system}. Clearly, for mean field games without dependence in law on the controls, the above result can be applied directly with $(g,F,G)=(\nabla_x u(T,\cdot),D_p H,-D_xH)$ for $H$ the Hamiltonian of minor players. In light of the analysis carried out in \ref{subsection: Nash}, the case of mean field games of controls is more technical and requires an additional assumption, namely the invertibility of the Hilbertian lift of $\nabla_\alpha L$ which is equivalent to the wellposedness of the fixed point mapping formulation of MFG of controls \cite{extragradient-in-mfg,jackson2025quantitativeconvergencedisplacementmonotone}. We now make the following assumption
\begin{hyp}
    \label{hyp: motonicity mfgmp lagrangian minor}
The Lagrangian of minor players $(x,q,a,\mu)\mapsto L(x,q,a,\mu)$ is such that 
\begin{enumerate}
    \item[-] $\nabla_x L$ is Lipschitz in all arguments
    \item[-] $(x,q,\mu)\mapsto \nabla_a L(x,q,a,\mu)$ is Lipschitz uniformly in $a\in \reels^d$ 
\end{enumerate}
Moreover there exists a symmetric matrix $A\in \mathcal{M}_{d^0}(\reels)$ and constants $\beta_0,c_L>0$ 
\begin{enumerate}
    \item[-] the inequality \eqref{eq: monotonicity on initial condition} holds for $g=\nabla_xu(T,\cdot)$ and $\psi$ for the matrix $A$ and with $\beta_0$
    \item[-]  \begin{gather}
   \nonumber \forall (X,\alpha,q),(Y,\alpha',q')\in \mathcal{H}^{2d}\times \reels^{m}, \quad z\in \reels^{d^0}\\
\label{eq: ineq monotonicity mfgmp lagrangian minor}\langle \Delta (\nabla_x \tilde{L},\nabla_\alpha \tilde{L},AD_z\tilde{H}^0),\Delta \bar{X}\rangle\geq c_L \|\Delta \bar{X}\|^2 
\end{gather}
where 
\[
  \Delta \bar{X}=\left(\begin{array}{c}
X-Y\\
\alpha-\alpha'\\
q-q'
\end{array}\right). \]
and 
\[\Delta (\nabla_x \tilde{L},\nabla_\alpha \tilde{L},AD_z\tilde{H}^0)=\left(\begin{array}{c}
\nabla_x L(X,\alpha,q,\mathcal{L}(X,\alpha))-\nabla_x L(Y,\alpha',q',z,\mathcal{L}(Y,\alpha'))\\
\nabla_\alpha L(X,\alpha,q,\mathcal{L}(X,\alpha))-\nabla_\alpha L(Y,\alpha',q',z,\mathcal{L}(Y,\alpha'))\\
AD_z H^0(q,z,\mathcal{L}(X,\alpha))-AD_zH^0(q',z,\mathcal{L}(Y,\alpha'))
\end{array}\right),\]
\end{enumerate}
\end{hyp}
This monotonicity condition is formulated directly at the level of the Lagrangian of minor players $L$, which is natural in mean field games of controls. 
\begin{lemma}
    \label{lemma: main result existence for mfg}
Under Hypothesis \ref{hyp: motonicity mfgmp lagrangian minor}, the Hilbertian lift 
\[\alpha \mapsto \nabla_\alpha L(X,q,\alpha,\mathcal{L}(X,\alpha)),\]
has a well defined Lispchitz inverse \[(X,q,\alpha)\mapsto \nabla_\alpha \tilde{L}^{-1}(X,q,\alpha).\]
Moreover there exists a Lipschitz function \[L_{inv}:\reels^d\times \reels^{d^0}\times \reels^d\times\mathcal{P}_2(\reels^{2d})\to \reels^d,\]
such that 
\[\forall (X,\alpha,q)\in \mathcal{H}^{2d}\times \reels^{d^0},\quad \nabla_\alpha \tilde{L}^{-1}(X,q,\alpha)=L_{inv}(X,q,\alpha,\mathcal{L}(X,\alpha)).\]
Finally, the functions 
\[
\begin{array}{l}
F:(x,q,u,\mathcal{L}(X,U))\mapsto L_{inv}(x,q,u,\mathcal{L}(X,U)),\\
G:(x,q,u,\mathcal{L}(X,U))\mapsto \nabla_x L(x,q,L_{inv}(x,q,u,\mathcal{L}(X,U)),\mathcal{L}(X,L_{inv}(X,q,U,\mathcal{L}(X,U))))
\end{array}
\]
are Lipschitz in all arguments and letting
\[H^{0'}:(q,z,\mathcal{L}(X,U))\mapsto H^0(q,z,\mathcal{L}(X,\nabla_\alpha \tilde{L}^{-1}(X,q,U))),\]
the following monotonicity assumption holds for the triple $(F,G,H^{0'})$
\begin{gather}
   \nonumber \forall (X,U,q),(Y,V,q')\in \mathcal{H}^{2d}\times \reels^{m}, \quad z\in \reels^{d^0}\\
\label{eq: ineq monotonicity coef mfg control with major}\langle \Delta (\tilde{G},\tilde{F},AD_z\tilde{H}^{',0}),\Delta \bar{X}\rangle\\
\nonumber \geq c_L \left(\|X-Y\|^2+|q-q'|^2+\|\Delta \tilde{F}\|^2\right)
\end{gather}
\end{lemma}
\begin{proof}
    Let us first observe that under Hypothesis \ref{hyp: motonicity mfgmp lagrangian minor}, there exist a constant $C_{Lip}$ depending only on $c_L$ and the Lipschitz constant of $(x,q,\mu)\mapsto \nabla_a L(x,q,a,\mu)$ such that
    \begin{gather*}
        \forall (X,Y,\alpha,\alpha',q,q')\in \mathcal{H}^{4d}\times \reels^{2m}\\
        C_{Lip}\left(\|X-Y\|^2+|q-q'|^2\right)+\langle \nabla_\alpha L(X,q,\alpha,\mathcal{L}(X,\alpha))-\nabla_\alpha L(Y,q',\alpha',\mathcal{L}(X',\alpha')),\alpha-\alpha'\rangle \geq \frac{c_L}{2}\|\alpha-\alpha'\|^2.
    \end{gather*}
    This implies that the function 
    \[\FuncDef{\mathcal{H}^d\times \reels^{d^0}\times \mathcal{H}^d}{\mathcal{H}^d}{(X,q,\alpha)}{\nabla_\alpha \tilde{L}^{-1}(X,q,\alpha)},\]
    is well defined and Lipschitz. Moreover by \cite{monotone-sol-meynard} Lemma 3.10, this Hilbertian inverse can be represented as the lift of a Lipschitz function $L_{inv}$. Finally evaluating the expression \eqref{eq: ineq monotonicity mfgmp lagrangian minor} for 
    \[(\alpha,\alpha')=\left(\nabla_\alpha \tilde{L}^{-1}(X,q,U),\nabla_\alpha \tilde{L}^{-1}(X,q,V)\right),\]
    for $U,V\in \mathcal{H}^d$ yields the joint monotonicity condition on $(F,G,H^0)$. 
\end{proof}
\begin{remarque}
    A direct consequence is that under Hypothesis \eqref{hyp: motonicity mfgmp lagrangian minor}, the assumption we made on the Lagrangian of minor players to make the link between Lispchitz solution and Nash equilibrium, Hypothesis \eqref{hyp: mfg with major optimality minor} holds for $L$.
\end{remarque}
The monotonicity condition obtained in this lemma is slightly different from \eqref{hyp: monotonicity mfgmp}. In fact, under the additional assumption that $\nabla_\alpha L$ is globally Lipschitz (meaning also in the controls), it is a straighforward extension to show that the monotonicity condition \eqref{hyp: monotonicity mfgmp} holds. However, this condition limit us to the study of Lagrangians with quadratic growth and is not required. Instead, let us remark that in this setting, the associated forward backward system writes
\begin{equation*}
\left\{
\begin{array}{l}
     \displaystyle X_t=X_0-\int_0^t  F_sds+\sqrt{2\sigma}B_t, \\
     \displaystyle q_t=q_0-\int_0^t D_zH^0(q_s,\varphi_s,Z^\varphi_s,\mathcal{L}(X_s,F_s|\mathcal{F}^0_s))ds+\sqrt{2\sigma^0}W^0_t,\\
    \displaystyle  U_t= g(X_T,q_T,\mathcal{L}(X_T|\mathcal{F}^0_T))+\int_t^T G(X_s,q_s,F_s,\mathcal{L}(X_s,F_s|\mathcal{F}^0_s))ds-\int_t^T
    Z_sd(B,W^0)_s,\\
\displaystyle \varphi_t= \psi(q_T,\mathcal{L}(X_T|\mathcal{F}^0_T))+\int_t^T L_{major}(q_s,\varphi_s,Z^\varphi_s,\mathcal{L}(X_s,F_s|\mathcal{F}^0_s))ds-\sqrt{2\sigma^0}\int_t^T Z^\varphi_sdW^0_s,\\
F_s=F(X_s,q_s,U_s,\mathcal{L}(X_s,U_s|\mathcal{F}^0_s)).
\end{array}
\right.
\end{equation*}
Since all coefficients depend on the backward process $(U_t)_{t\in [0,T]}$, only through $(F_t)_{t\in [0,T]}$, all results presented in Section 3.1 can be adapted under the monotonicity condition \eqref{eq: ineq monotonicity coef mfg control with major} for this particular class of systems. This concludes the demonstration that under Hypothesis \ref{hyp: motonicity mfgmp lagrangian minor}, the long time existence theory presented in Section 3.1 applies to MFGs of controls for which the coefficients are exactly the previously defined functions $(F,G,H')$. 

Finally, let us end this section by insisting that an analogue to Lemma \ref{lemma: main result existence for mfg} can be proved when the Lagrangian of minor players $L$ depends on $z$ so long as the Hamiltonian of the major player does not depend on the law of controls of minors. This shows that Theorem \ref{thm: existence for sigma T mfgmp} and its variants hold when minor players depend on the control of the major player under the equilibrium condition presented in Remark \ref{remarque: mfg with dependency on control of major}. Without presenting the result in full generality, let us give a simple example which fits into our framework
\begin{exemple}
The Hamiltonian of the major player is given by 
\[(q,z,\mathcal{L}(X))\mapsto b(q,\mathcal{L}(X))\cdot z+\frac{1}{2}|z|^2,\]
and the Lagrangian of minor players 
\[(x,q,z,a,\mathcal{L}(X,\alpha))\mapsto L(x,q,z,a,\mathcal{L}(X,\alpha)),\]
satisfies Hypothesis \ref{hyp: motonicity mfgmp lagrangian minor} uniformly in $z\in \reels^{d^0}$. In this setting, this last condition is equivalent to the strong monotonicity of $(\nabla_xL,\nabla_\alpha L,A b)$ uniformly in $z$ for some $A\in \mathcal{M}_{d^0}(\reels)$. Given a control $(\alpha^0_t)_{t\in [0,T]}$, the dynamics of the major player follows
\[dq_t=(b(q_t,\mu_t)+\alpha^0_t)dt+\sigma^0dW^0_t.\]
The optimal control for the major player at the Nash equilibrium is exactly given by the backward variable $(Z^\varphi_t)_{t\in [0,T]}$ in the associated stochastic formulation \eqref{mfg with major}. In particular the Lagrangian of minor players depends on $z$ directly through the control of the major player, which is natural in MFGs with a major player. 
\end{exemple}

\section{Decoupling algorithm}
In this section, we present a numerical method to solve the system \eqref{mfg with major}. By a decoupling algorithm, we here indicate a method which iterates on a sequence of decoupled FBSDEs to approximate the solution of the system \eqref{mfg with major}. A typical example consists in schemes based on Picard iterations, as in \cite{numerical-FBSDE-II}. The scheme we propose here is instead inspired by the recently developed extragradient methods in mean field games \cite{extragradient-in-mfg}. Most notably, it does not rely on a contraction in small time but instead uses the monotonicity inherent to the problem to build a solution directly on the whole interval $[0,T]$.
\subsection{On extragradient methods}
Let us briefly present the main idea behind extragradient methods \cite{extra-gradient,Nesterov}. This class of methods was introduced as a means of solving monotone variational inequalities of the form 
\begin{equation}
    \label{eq: varitional inequality extra gradient}
    \text{find }x^*\in \reels^d, \quad \forall x\in \reels^d, \quad v(x^*)\cdot (x-x^*)\geq 0,
\end{equation}
for continuous and monotone functions $v:\reels^d\to \reels^d$. Starting from an initial point $x_0$, the key change from standard gradient methods is to first do a trial step 
\[x_{n+\frac{1}{2}}=x_n-\gamma v(x_n),\]
before updating the direction of descent using the extrapolated point 
\[x_{n+1}=x_n-\gamma v(x_{n+\frac{1}{2}}),\]
hence the terminology of extragradient methods. If the function $v$ is monotone and Lipschitz, it is proven in \cite{extra-gradient} that the sequence $(x_n)_{n\in \mathbb{N}}$ converges to a solution $x^*$ of \eqref{eq: varitional inequality extra gradient} for step sizes $\gamma$ sufficiently small. 

In \cite{extragradient-in-mfg} we showed that under monotonicity assumptions, some mean field games and a class of FBSDEs can be equivalently represented as solutions to variational inequalities on the Hilbert spaces of square integrable random processes. These results do not extend directly to the system \eqref{mfg with major} we consider in this paper. However, since the wellposedness theory we presented in the previous section is also based on a notion of monotonicity, a natural question is whether \eqref{mfg with major} can also be written as the solution of a variational inequality. We show that this is indeed the case, allowing us to write a natural algorithm converging to the solution of \eqref{mfg with major}.
\subsection{Construction and convergence of a decoupling algorithm}
In this section, Hypotheses \ref{hyp: coefficients mfgmp} and \ref{hyp: monotonicity mfgmp} are in force. We assume that we are in the conditions of application of Theorem \ref{thm: existence mfgmp}, that is to say $\sigma^0>\sigma^0_T$. Consequently, there exists a unique Lipschitz solution of MFGMP$(T,\psi,g,D_zH^0,L_{major},F,G)$ on $[0,T]$. Before presenting our method, let us insist that coefficients may have superlinear growth in the backward variable $Z^\varphi$. This must be taken into account to ensure the wellposedness of numerical methods. Let us remark that by the existence of a Lipschitz solution, and previous estimates, we already now that there exists a constant $M$ such that for the solution $(X_t,q_t,U_t,\varphi_t,Z_t,Z^\varphi_t)_{t\in [0,T]}$ of \eqref{mfg with major} the following holds, independently of the initial condition $(X_0,q_0)$
\[  |Z^\varphi_t|\leq M \quad a.e \text{ for } t\in [0,T] \quad a.s. \]
In particular this means that $(X_t,q_t,U_t,\varphi_t,Z_t,Z^\varphi_t)_{t\in [0,T]}$ is also a strong solution to 
\begin{equation}
\label{eq: mfg with major lipschitz coefficients}
\left\{
\begin{array}{l}
     \displaystyle X_t=X_0-\int_0^t  F^M(X_s,q_s,U_s,Z^\varphi_s,\mathcal{L}(X_s,U_s|\mathcal{F}^0_s))ds+\sqrt{2\sigma}B_t, \\
     \displaystyle q_t=q_0-\int_0^t D_zH^M(q_s,Z^\varphi_s,\mathcal{L}(X_s,U_s|\mathcal{F}^0_s))ds+\sqrt{2\sigma^0}W^0_t,\\
    \displaystyle  U_t= U_T+\int_t^T G^M(X_s,q_s,U_s,Z^\varphi_s,\mathcal{L}(X_s,U_s|\mathcal{F}^0_s))ds-\int_t^T
    Z_sd(B,W^0)_s,\\
\displaystyle \varphi_t= \varphi_T+\int_t^T L^M(q_s,\varphi_s,Z^\varphi_s,\mathcal{L}(X_s,U_s|\mathcal{F}^0_s))ds-\sqrt{2\sigma^0}\int_t^T Z^\varphi_sdW^0_s,\\
U_T=g(X_T,q_T,\mathcal{L}(X_T|\mathcal{F}^0_T)), \quad \varphi_T=\psi(q_T,\mathcal{L}(X_T|\mathcal{F}^0_T)),
\end{array}
\right.
\end{equation}
where $H^M$ is defined by 
\[H^M:(q,u,\varphi,z,\mu)\mapsto H_{major}(q,u,\varphi,z\wedge M\vee (-M),\mu),\]
and $F^M,G^M,L^M$ are defined in a similar fashion. Moreover, by Theorem \ref{lemma: uniqueness lip sol}, it is the only strong solution of \eqref{eq: mfg with major lipschitz coefficients}. In what follows, we focus on the system \eqref{eq: mfg with major lipschitz coefficients} which is much easier to work with since its coefficients are Lipschitz. Although this constant $M$ is not exogenous but very much dependent on the constants of the problem, and in particular $\sigma^0$, we ignore this dependency in our notation for the sake of clarity. We now introduce the coefficients $(F',H',G',L_H')$ defined by 
\[(F',H',G',L_H')(x,q^f,q^b,u,\varphi,z,\mu)=(F^M,H^M,G^M,L^M)(x,\frac{q^f+q^b}{2},u,\varphi,z,\mu).\]
The system \eqref{eq: mfg with major lipschitz coefficients} is equivalent to 
\begin{equation}
\label{eq: q forward backward}
\left\{
\begin{array}{l}
     \displaystyle X_t=X_0-\int_0^t  F'(X_s,q^f_s,,q^b_s,U_s,Z^\varphi_s,\mathcal{L}(X_s,U_s|\mathcal{F}^0_s))ds+\sqrt{2\sigma}B_t, \\
     \displaystyle q^f_t=q_0-\int_0^t D_zH'(q^f_s,q^b_s,Z^\varphi_s,\mathcal{L}(X_s,U_s|\mathcal{F}^0_s))ds+\sqrt{2\sigma^0}W^0_t,\\
     \displaystyle q^b_t=q^f_T+\int_0^t D_zH'(q^f_s,q^b_s,Z^\varphi_s,\mathcal{L}(X_s,U_s|\mathcal{F}^0_s))ds+Z^q_sdW^0_s,\\
    \displaystyle  U_t= U_T+\int_t^T G'(X_s,q^f_s,q^b_s,U_s,Z^\varphi_s,\mathcal{L}(X_s,U_s|\mathcal{F}^0_s))ds-\int_t^T
    Z_sd(B,W^0)_s,\\
\displaystyle \varphi_t= \varphi_T+\int_t^T L'_H(q^f_s,q^b_s,\varphi_s,Z^\varphi_s,\mathcal{L}(X_s,U_s|\mathcal{F}^0_s))ds-\sqrt{2\sigma^0}\int_t^T Z^\varphi_sdW^0_s,\\
U_T=g(X_T,q_T,\mathcal{L}(X_T|\mathcal{F}^0_T)), \quad \varphi_T=\psi(q_T,\mathcal{L}(X_T|\mathcal{F}^0_T)),
\end{array}
\right.
\end{equation}
Clearly, the function $(t,x,q,\mu)\mapsto (q,U(t,x,q,\mu),\varphi(t,q,\mu))$ is a Lipschitz solution associated to \eqref{eq: q forward backward}, and since coefficients are Lipschitz this shows by Theorem \ref{lemma: uniqueness lip sol} that the unique strong solution is indeed given for $(q^b_t)_{t\in [0,T]}\equiv (q^f_t)_{t\in [0,T]}$.

\begin{lemma}
    \label{lemma: coef invertible mfgmp}
    Under Hypotheses \ref{hyp: coefficients mfgmp} and \ref{hyp: monotonicity mfgmp}, for any $(X,z,p)\in \mathcal{H}^d\times \reels^{2dm}$, the function 
    \[(F',D_zH'):(U,q)\mapsto (F',D_zH')(X,U,p,q,z,\mathcal{L}(X,U)),\]
    has a Lipschitz inverse $\theta(X,p,z):\mathcal{H}^d\times \reels^{d^0}\to \mathcal{H}^{d+d^0}$, we denote by $\theta_F$ its first $d-$coordinates and by $\theta_H$ its last $m$ coordinates. 
\end{lemma}
\begin{proof}
We first fix a couple $(X,z,p)\in \mathcal{H}^d\times \reels^{2m}$. Let us observe that for all $(U,q^1),(V,q^2)\in \mathcal{H}^d\times \reels^{d^0}d$
\begin{gather*}\langle F'(X,U,p,q^1,z,\mathcal{L}(X,U))-F'(X,V,p,q^2,z,\mathcal{L}(X,V)),U-V\rangle\\
    + \langle D_zH('p,q^1,z,\mathcal{L}(X,U))-D_z H'(p,q^2,z,\mathcal{L}(X,V)),\frac{q^1-q^2}{2}\rangle\\
    \geq \alpha \left(\|U-V\|^2+\frac{1}{4}\|q^1-q^2\|^2\right)
\end{gather*}
By the strong monotonicity of  $(U,q)\mapsto (F',\frac{1}{2}D_zH')(X,U,p,q,z,\mathcal{L}(X,U))$, this function has a Lipschitz inverse defined on $\mathcal{H}^d\times \reels^{d^0}$. As for the regularity with respect to the other variables, it follows naturally using the fact that $F',D_zH'$ are Lipschitz functions with respect to all of their arguments. 
\end{proof}
\begin{remarque}
    \label{remarque: inverse for mfg}
    Observe that in the case of mean field games this inverse function is quite simple to compute. Indeed $\theta$ is exactly given by 
    \[\theta(X,p,z): (\alpha,q)\mapsto \left(\nabla_\alpha L(X,\alpha,\frac{p+q}{2},\mathcal{L}(X,\alpha)),D_z H'^{-1}(p,z,\mathcal{L}(X,\alpha))(q))\right),\]
    for $L$ the Lagrangian of minor players and where $q\mapsto D_zH'^{-1}(p,z,\mu)(q)$ indicates the inverse function of 
    \[q\mapsto D_zH'(p,q,z,\mu).\]
\end{remarque}
Introducing the space of adapted controls
\[\mathcal{H}^T_{mp}=\left(\mathcal{H}^T_\mathcal{F}\right)^d\times \left(\mathcal{H}^T_{\mathcal{F}^0}\right)^{d^0}.\]
For any control $(\alpha^x_s,\alpha^q_s)_{s\in [0,T]}\in \mathcal{H}^T_{mp}$, we introduce the following parametrized system 
\begin{equation}
\label{eq: parametrized system mfgmp}
\left\{
\begin{array}{l}
     \displaystyle X^\alpha_t=X_0-\int_0^t  \alpha^x_sds+\sqrt{2\sigma}B_t, \\
     \displaystyle q^{f,\alpha}_t=q_0-\int_0^t \alpha^q_sds+\sqrt{2\sigma^0}W^0_t,\\
     \displaystyle q^{b,\alpha}_t=q^{f,\alpha}_T+\int_t^T D_zH'(q^{f,\alpha}_s,\theta^H_s,\varphi^\alpha_s,Z^{\varphi,\alpha}_s,\mathcal{L}(X^\alpha_s,\theta^F_s|\mathcal{F}^0_s))ds-\int_t^TZ^{q,\alpha}_sdW^0_s,\\
    \displaystyle  U^\alpha_t= g(X^\alpha_T,q^{f,\alpha}_T,\mathcal{L}(X^\alpha_T|\mathcal{F}^0_T))+\int_t^T G'(X^\alpha_s,q^{f,\alpha}_s,\theta^H_s,\theta^F_s,\varphi^\alpha_s,Z^{\varphi,\alpha}_s,\mathcal{L}(X^\alpha_s,\theta^F_s|\mathcal{F}^0_s))ds-\int_t^T
    Z^\alpha_sd(B,W^0)_s,\\
\displaystyle \varphi^\alpha_t= \psi(q^{f,\alpha}_T,\mathcal{L}(X^\alpha_T|\mathcal{F}^0_T))+\int_t^T L'_H(q^{f,\alpha}_s,\theta^H_s,\varphi^\alpha_s,Z^{\varphi,\alpha}_s,\mathcal{L}(X^\alpha_s,\theta^F_s|\mathcal{F}^0_s))ds-\sqrt{2\sigma^0}\int_t^T Z^{\varphi,\alpha}_sdW^0_s,\\
(\theta^F_s,\theta^H_s)=\theta(X^\alpha_s,q^{f,\alpha_s},Z^\alpha_s)(\alpha^x_s,\alpha^q_s).
\end{array}
\right.
\end{equation}
By standard results on BSDEs \cite{zhang2017backward} (Theorem 4.3.1), for any control $(\alpha^x_s,\alpha^q_s)_{s\in [0,T]}$ there exists a unique strong solution to the above system. For the forward part of this system, this follows from standard results on SDEs. As for the backward SDEs they can be written in the form 
\[(U^\alpha_t,\varphi^\alpha_t)=\xi^\alpha_T-\int_t^T f(s,\omega,U^\alpha_s,\varphi^\alpha_s,Z^\alpha_s)ds+\int_t^T Z^\alpha_sd(W^0_s,B_s),\]
with a $\mathcal{F}-$adapted Lipschitz driver such that 
\[\esp{|\xi^\alpha_T|^2+\int_0^T |f(s,\omega,0)|^2}<+\infty.\]
 We now introduce the associated monotone operator
\begin{lemma}
    \label{lemma: monotone functional mfgmp}
 Under Hypotheses \ref{hyp: monotonicity mfgmp}, \ref{hyp: coefficients mfgmp}, we define the operator $v:\mathcal{H}^T_{mp}\to \mathcal{H}^T_{mp}$ by
 \[\forall \alpha\in \mathcal{H}^T_{mp}, \quad v(\alpha)=\left(\begin{array}{c}
 I \quad 0\\
 0 \quad  \frac{1}{2}A
 \end{array}\right)   \left(\theta(X^\alpha_s,q^{f,\alpha}_s,Z^{\varphi,\alpha}_s)(\alpha^x_s,\alpha^q_s)-(U^\alpha_s,q^{b,\alpha}_s)
    \right)_{s\in [0,T]}.\]
If $\sigma^0>\sigma^0_T$ then 
    \begin{enumerate}
        \item[-]$v$ is Lipschitz on $\mathcal{H}^T_{mp}$
        \item[-] for any $\alpha,\alpha'\in \mathcal{H}^T_{mp}$
        \[\langle v(\alpha)-v(\alpha'),\alpha-\alpha'\rangle \geq 0.\]
        \item[-] Let $\alpha^*=(F^M,D_zH^M)(X_s,q_s,U_s,Z^\varphi_s,\mathcal{L}(X_s,U_s|\mathcal{F}^0_s))$ for $(X_s,q_s,U_s,\varphi_s,Z_s,Z^\varphi_s)_{s\in [0,T]}$ the solution of \eqref{eq: mfg with major lipschitz coefficients}. Then, for any $\alpha,\alpha' \in \mathcal{H}^T_{mp}$, 
        \[\langle v(\alpha^*),\alpha-\alpha'\rangle^T=0.\]
    \end{enumerate}
\end{lemma}
\begin{proof}
    The first claim is a direct consequence of Lemma \ref{lemma: coef invertible mfgmp}, and standard estimates on backward Lipschitz SDEs. The last claim is also straightforward: by the uniqueness of a strong solution to \eqref{eq: q forward backward}, it has already been established that for this choice $v(\alpha^*)=0$.
    It remains to show the second claim on the monotonicity of $v$. 
    By Ito's lemma for any $\alpha,\alpha'\in \mathcal{H}^T_{mp}$, 
    \begin{gather*}\esp{\int_0^T (-U^\alpha_s\cdot(\alpha^x_s-\alpha'^x_s)-\frac{1}{2}Aq^{b,\alpha}_s\cdot (\alpha^q_s-\alpha'^q_s)ds}\\
        =   \esp{g(X^\alpha_T,q^{f,\alpha}_T,\mathcal{L}(X^\alpha_T|\mathcal{F}^0))\cdot (X^\alpha_T-X^{\alpha'}_T)+\frac{1}{2}Aq^{f,\alpha}_T\cdot (q^{f,\alpha}_T-q^{f,\alpha'}_T)}\\
        +\esp{\int_0^T G'(X^\alpha_s,q^{f,\alpha}_s,\theta^H_s,\theta^F_s,\varphi^\alpha_s,Z^{\varphi,\alpha}_s,\mathcal{L}(X^\alpha_s,\theta^F_s|\mathcal{F}^0_s))\cdot (X^\alpha_s-X^{\alpha'}_s)ds}\\
         +\esp{\int_0^T \frac{1}{2}AD_zH'(q^{f,\alpha}_s,\theta^H_s,\varphi^\alpha_s,Z^{\varphi,\alpha}_s,\mathcal{L}(X^\alpha_s,\theta^F_s|\mathcal{F}^0_s))\cdot (q^{f,\alpha}_s-q^{f,\alpha'}_s)ds}.
    \end{gather*}
Letting $(V^\alpha_s,p^\alpha_s)=\theta(X^\alpha_s,q^{f,\alpha},Z^{\varphi,\alpha}_s)(\alpha)$, it follows that for any smooth function of time $\beta:[0,T]\to \reels$, 
\begin{gather*}
\langle v(\alpha)-v(\alpha'),\alpha-\alpha'\rangle^T=\\
 \esp{\left( g(X^\alpha_T,q^{f,\alpha}_T,\mathcal{L}(X^\alpha_T|\mathcal{F}^0))-g(X^{\alpha'}_T,q^{f,\alpha'}_T,\mathcal{L}(X^{\alpha'}_T|\mathcal{F}^0))\right)\cdot (X^\alpha_T-X^{\alpha'}_T)+\frac{1}{2}A(q^{f,\alpha}-q^{f,\alpha'}_T)\cdot (q^{f,\alpha}_T-q^{f,\alpha'}_T)}\\
 +\esp{+\beta(0)|\varphi^\alpha_T-\varphi^{\alpha'}_T|^2-\beta(T)|\varphi^\alpha_0-\varphi^{\alpha'}_0|^2}\\
 +\esp{\int_0^T \left(G'^{\alpha}_s-G'^{\alpha'}_s\right)\cdot (X^\alpha_s-X^{\alpha'}_s)ds+\int_0^T \left(F'^{\alpha}_s-F'^{\alpha'}_s\right)\cdot (V^\alpha_s-V^{\alpha'}_s)ds }\\
 +\esp{\int_0^T A\left(D_zH'^{\alpha}_s-D_zH'^{\alpha'}_s\right)\cdot (\frac{q^{f,\alpha}_s+p^\alpha_s}{2}-\frac{q^{f,\alpha'}_s+p^{\alpha'}_s}{2})ds+\int_0^T\sigma^0\beta(T-s)|Z^\alpha_s-Z^{\alpha'}_s|^2 ds }\\
 +\esp{\int_0^T\left(( {L_H'}^{\alpha}_s-{L_H'}^{\alpha'}_s)(\varphi^\alpha_s-\varphi^{\alpha'}_s)+\frac{d\beta}{ds}(T-s)|\varphi^\alpha_s-\varphi^{\alpha'}_s|^2\right)ds},
\end{gather*}
where we have used the notation 
\[G'^{\alpha}_s=G'(X^\alpha_s,q^{f,\alpha}_s,p^\alpha_s,V^\alpha_s,\varphi^\alpha_s,Z^{\varphi,\alpha}_s,\mathcal{L}(X^\alpha_s,V^\alpha_s|\mathcal{F}^0)),\]
and so on for the different coefficients. Since the Hypotheses of Theorem \ref{thm: existence for sigma T mfgmp} are satisfied, we may choose $\beta$ as we did in its proof and it follows naturally that for this choice, there exists a constant $c_v$ depending on the coefficients but independent of $(\alpha,\alpha')$
\begin{gather}
 \label{strong monotonicity of functional mfgmp}   \langle v(\alpha)-v(\alpha'),\alpha-\alpha'\rangle^T \\
\nonumber \geq c_v\esp{\int_0^T \left(|V^{\alpha}-V^{\alpha'}|^2+|Z^{\varphi,\alpha}_s-Z^{\varphi,\alpha'}_s|^2+|X^\alpha_s-X^{\alpha'}_s|^2+\left|\frac{q^{f,\alpha}_s+p^\alpha}{2}-\frac{q^{f,\alpha'}_s+p^{\alpha'}}{2}\right|^2\right)ds }.
\end{gather}
\end{proof}
At this point, let us explain the reasoning behind the parametrization \eqref{eq: parametrized system mfgmp} we presented, and why it works. It might seem very unnatural that we first introduced a fictitious backward version of the forward process $(q_t)_{t\in [0,T]}$ before parametrizing the problem. At an extremely formal level, this is because in the theory of monotone FBSDEs, comparison is obtained by studying the inner product between the forward and the backward processes. Adding a quadratic term 
\[\langle q_T-q'_T,A(q_T-q'_T)\rangle,\]
corresponds to interpreting $(q_t)_{t\in [0,T]}$ both as a forward and backward process. For a more rigorous explanation of the logic behind this idea, we refer to \cite{noise-add-variable} Section 3.2, which also hints at how the notion of monotonicity used in Hypothesis \ref{hyp: monotonicity mfgmp} can be generalized. For example it can be adapted to G-monotonicity \cite{G-monotonicity,noise-add-variable}, though in this case a different parametrization must be used for the system \eqref{eq: mfg with major lipschitz coefficients}. Before we state the main convergence result for \eqref{mfg with major}, let us also insist once again, that this is only one possible parametrization of the problem. It is unclear to us whether other parametrizations of \eqref{eq: mfg with major lipschitz coefficients} exist in this setting, possibly leading to better convergence results. 
\begin{thm}
 Let $v$ be defined as in Lemma \ref{lemma: monotone functional mfgmp}, for a given initial condition $\alpha_1)\in (\mathcal{H}^T_{mp})$ and $n\geq 1$ we introduce the sequences
\[\left\{\begin{array}{c}
     \alpha^{n+\frac{1}{2}}=\alpha^n-\gamma v(\alpha^n),  \\
     \alpha^{n+1}=\alpha^n-\gamma v(\alpha^{n+\frac{1}{2}}),
\end{array}\right. \]
under Hypotheses \ref{hyp: coefficients mfgmp}, \ref{hyp: monotonicity mfgmp}, if 
\[\gamma\leq \frac{1}{\|v\|_{Lip,\mathcal{H}^T}},\]
and $\sigma^0>\sigma^0_T$, then letting 
\[\forall t\in [0,T],\quad (\bar{U}^n_t,\bar{p}^n_t,\bar{q}^n_t,\bar{X}^n_t,\bar{Z}^{\varphi,n}_t)=\frac{1}{n}\sum_{i=1}^n \left(\theta(X^{\alpha^{i+\frac{1}{2}}}_t,q^{\alpha^{i+\frac{1}{2}}}_t,Z^{\varphi,\alpha^{i+\frac{1}{2}}}_t)(\alpha^{i+\frac{1}{2}}_t),q^{\alpha^{i+\frac{1}{2}}}_t,X^{\alpha^{i+\frac{1}{2}}}_t,Z^{\varphi,\alpha^{i+\frac{1}{2}}}_t\right),\]
and $(X_s,U_s,q_s,\varphi_s,Z_s,Z^\varphi_s)_{s\in [0,T]}$ be the strong solution of \eqref{eq: mfg with major lipschitz coefficients},
the following holds 
\begin{enumerate}
    \item[-] \[\forall n\geq 1,\quad \|(U,X,q,Z^\varphi)-(\bar{U}^n,\bar{X}^n,\frac{\bar{q}^n+\bar{p}^n}{2},\bar{Z}^{\varphi,n})\|_T^2\leq \frac{\| (F,D_zH)(X_\cdot,U_\cdot,q_\cdot,Z^\varphi_\cdot)-\alpha_1\|^2_T}{2\gamma c_v n}.\]
    \item[-] Letting  $(\bar{\varphi}^n_t)_{t\in [0,T]}$ be the unique strong solution to the backward SDE
    \[\forall t\in [0,T],\quad \bar{\varphi}^n_t=\psi(\frac{\bar{q}^n_T+\bar{p}^n_T}{2},\mathcal{L}(\bar{X}^n_T))-\int_t^T L_H'(\frac{\bar{q}^n_T+\bar{p}^n_T}{2},\bar{\varphi}^n_s,\bar{Z}^{\varphi,n}_s,\mathcal{L}(\bar{X}^n_s,\bar{U}^n_s|\mathcal{F}^0_s))ds-\int_t^T \sqrt{2\sigma^0}Z^{\varphi,n}_s dW^0_s,\] there exists a constant $C$ depending only on $T$ and the coefficients such that 
    \[\| \bar{\varphi}^n_\cdot-\varphi_\cdot\|^2_T\leq C \frac{\| (F,D_zH)(X_\cdot,U_\cdot,q_\cdot,Z^\varphi_\cdot)-\alpha_1\|^2_T}{2\gamma c_L n}.\] 
\end{enumerate}
\end{thm}
\begin{proof}
    Letting $\alpha^*$ be defined as in Lemma \ref{lemma: monotone functional mfgmp} (i.e. the control associated to the solution of \eqref{eq: def lip sol system}). We first observe by applying Proposition 2 of \cite{NEURIPS2021_6d65b5ac}, that for any $\alpha\in \mathcal{H}^T_{mp}$, the following holds 
\[\sum_{i=1}^n \langle v(\alpha^{i+\frac{1}{2}}),\alpha^{i+\frac{1}{2}}-\alpha \rangle^T\leq \frac{\|\alpha-\alpha^1\|^2_T}{2\gamma}+\frac{1}{2}\sum_{i=1}^n \left(\gamma \|v(\alpha^{i+\frac{1}{2}})-v(\alpha^i)\|^2_T-\frac{1}{\gamma}\|\alpha^{i+\frac{1}{2}}-\alpha^i\|^2_T\right).\]
In particular, if $\alpha=\alpha^*$ and $\gamma<\|v\|_{Lip,\mathcal{H}^T}$ then it follows that 
\[\sum_{i=1}^n \langle v(\alpha^{i+\frac{1}{2}})-v(\alpha^*),\alpha^{i+\frac{1}{2}}-\alpha^* \rangle^T\leq \frac{\|\alpha^*-\alpha^1\|^2_T}{2\gamma}.\]
Using the strong monotonicity of $v$ \eqref{strong monotonicity of functional mfgmp}, we deduce that 
\[\forall n\geq 1,\quad \|(U,X,q,Z^\varphi)-(\bar{U}^n,\bar{X}^n,\frac{\bar{q}^n+\bar{p}^n}{2},\bar{Z}^{\varphi,n})\|_T^2\leq \frac{\| (F,D_zH)(X_\cdot,U_\cdot,q_\cdot,Z^\varphi_\cdot)-\alpha_1\|^2_T}{2\gamma c_v n}.\]
The second claim follows from standard gronwall type estimates on Lipschitz backward SDEs. 
\end{proof}
\begin{remarque}
This procedure can also be applied to systems with a common noise process
\begin{equation}
\label{eq: mfg with common noise process}
\left\{
\begin{array}{l}
     \displaystyle X_t=X_0-\int_0^t  F(X_s,q_s,U_s,\mathcal{L}(X_s,U_s|\mathcal{F}^0_s))ds+\sqrt{2\sigma}B_t, \\
     \displaystyle q_t=q_0-\int_0^t b(q_s,\mathcal{L}(X_s,U_s|\mathcal{F}^0_s))ds+\sqrt{2\sigma^0}W^0_t,\\
    \displaystyle  U_t= g(X_T,q_T,\mathcal{L}(X_T|\mathcal{F}^0_T))+\int_t^T G(X_s,q_s,U_s,\mathcal{L}(X_s,U_s|\mathcal{F}^0_s))ds-\int_t^T
    Z_sd(B,W^0)_s.
\end{array}
\right.
\end{equation}
In mean field games, this class of systems arises naturally whenever the dynamics of players depend on a common process
\[dq_t=b(q_t,\mu_t)dt+\sqrt{2\sigma^0}dW^0_t.\]
For a study of these types of mean field games we refer to \cite{noise-add-variable,common-noise-in-MFG}, let us mention that Theorem \ref{thm: existence for sigma T mfgmp} still applies in the framework, only since there is no major player $\beta$ can be taken equal to 0 in the assumptions and there is no need for any assumption on $\sigma^0$. For more general results, we refer to \cite{common-noise-in-MFG} Theorem 4.30. 
\end{remarque}
Although we believe this result to be of interest on its own, the convergence rate is only of order $\frac{1}{\sqrt{n}}$. This is because the monotone functional $v$ we have defined is not strongly monotone in the following sense 
\begin{equation}
    \label{eq: strong monotonicity v}
    \exists c>0\quad \forall (\alpha,\alpha')\in \left(H^T_{mp}\right)^2, \quad \langle v(\alpha)-v(\alpha'),\alpha-\alpha'\rangle^T\geq \eta \| \alpha-\alpha'\|^2_T.\end{equation}
Otherwise, it is quite standard that extragradient methods achieve a geometric convergence rate \cite{extra-gradient}. We now show that in the particular case of mean field games, the strong monotonicity of $v$ is a natural consequence of previous assumptions 
\begin{thm}
    \label{thm: exponential convergence mfg}
Under Hypothesis \ref{hyp: motonicity mfgmp lagrangian minor}, let $F,G,H$ be given as in Lemma \ref{lemma: main result existence for mfg} and suppose Hypothesis \ref{hyp: coefficients mfgmp} holds for $H$. Let $v$ be defined as in Lemma \ref{lemma: monotone functional mfgmp}, if $\sigma^0>\sigma^0_T$ then there exists a constant $\eta>0$ such that \eqref{eq: strong monotonicity v} holds. In particular, if 
\[\gamma <\min \left(\frac{1}{2\|v\|_{Lip}},\frac{\eta}{\|v\|^2_{Lip}}\right),\]
there exists $\lambda\in (0,1)$ depending only on $\|v\|_{Lip}$ and $\eta$ such that, starting from any initial condition $\alpha_1\in H^T_{mp}$ the extragradient procedure 
\[\left\{\begin{array}{c}
     \alpha^{n+\frac{1}{2}}=\alpha^n-\gamma v(\alpha^n),  \\
     \alpha^{n+1}=\alpha^n-\gamma v(\alpha^{n+\frac{1}{2}}),
\end{array}\right. \]
satifies 
\[\| \alpha^*_{mp}-\alpha^n\|_T\leq \lambda^n \| \alpha^*_{mp}-\alpha_1\|_T,\]
for 
\[\alpha^*_{mp}=(\alpha^*,\alpha^{0,*}),\]
where $\alpha^*$ and $\alpha^{0,*}$ are respectively the optimal control for a representative minor player and the optimal control for the major player at the mean field equilibrium given by \eqref{eq: true formulation mfgmp}.
\end{thm}
\begin{proof}
Using the same notation as in the proof of Lemma \ref{lemma: monotone functional mfgmp}, and following its proof, we find that there exists a constant $c_v>0$ such that for any two controls $\alpha,\alpha'$ the following holds 
\begin{gather}
     \langle v(\alpha)-v(\alpha'),\alpha-\alpha'\rangle^T \\
 \nonumber \geq c_v\esp{\int_0^T \left(|\Delta F'_s|^2+|Z^{\varphi,\alpha}_s-Z^{\varphi,\alpha'}_s|^2+|X^\alpha_s-X^{\alpha'}_s|^2+\left|\frac{q^{f,\alpha}_s+p^\alpha}{2}-\frac{q^{f,\alpha'}_s+p^{\alpha'}}{2}\right|^2\right)ds },
\end{gather}
with 
\[\Delta F'_s=F'(X^\alpha_s,q^{f,\alpha}_s,p^\alpha_s,V^\alpha_s,\mathcal{L}(X^\alpha_s,V^\alpha_s|\mathcal{F}^0_s))-F'(X^{\alpha'}_s,q^{f,\alpha'}_s,p^{\alpha'}_s,V^{\alpha'}_s,\mathcal{L}(X^\alpha_s,V^{\alpha'}_s|\mathcal{F}^0_s)).\]
Using the fact that $D_z H^M$ is a Lipschitz function and 
\[D_zH'^{\alpha}_s=D_z H^M (\frac{q^{f,\alpha}_s+p^\alpha}{2},Z^\alpha_s,\mathcal{L}(X^\alpha_s,F'^\alpha_s|\mathcal{F}^0_s)),\]
we deduce that 
\begin{gather}
     \forall \alpha,\alpha'\in H^T_{mp}, \quad \langle v(\alpha)-v(\alpha'),\alpha-\alpha'\rangle^T \\
 \nonumber \geq \frac{c_v}{1+\|D_z H^M\|_{Lip}}\esp{\int_0^T \left(|\Delta F'_s|^2+|\Delta D_zH'_s|^2\right)ds },
\end{gather}
where $\Delta D_zH'_s$ is defined as $\Delta F'_s$ for the function $D_zH'$. Finally by remark \ref{remarque: inverse for mfg} we get exactly 
\[\forall \alpha,\alpha'\in H^T_{mp}, \quad \langle v(\alpha)-v(\alpha'),\alpha-\alpha'\rangle^T\geq \frac{c_v}{1+\|D_zH^M\|_{Lip}}\|\alpha-\alpha'\|^2_T.\]
The result is then a direct adaptation of \cite{extragradient-in-mfg} Theorem 2.11. 
\end{proof}

\section{Perspectives}
We introduce a notion of weak solution for MFGs with a major player. This notion does not require any assumption of differentiability with respect to the measure argument and works as soon as the noise of the major player is not degenerate. We believe it will prove useful for future studies of mean field games with a major player. We show that Lipschitz regularity of the value function of the major player is sufficient for the wellposedness of major-minor Nash equilibrium. These general results are also valid in a wider a class of extended mean field games and in particular for mean field games of controls with a major player. 

We also present a framework for long time existence of mean field games with a major player, which rely on a coupled monotonicity assumption between the major and minor players. Under sufficiently strong assumptions, this yields uniform in time estimates for both value functions. An interesting direction of research would be to study the convergence to the corresponding ergodic limit and possible turnpike properties for this problem. 

Finally we present a numerical method for MFGs with a major player with a theoretical exponential convergence rate. We hope that this will prove useful for simulations of MFG with a major player, both for applications and to further our understanding of these problems in general.

\section*{acknowledgements}
\begin{center}
The author is thankful to François Delarue for helpful discussions and comments on the manuscript.

\quad

Charles Meynard acknowledge the financial support of the European Research Council (ERC) under the European Union’s Horizon Europe research and innovation program (ELISA project, Grant agreement No. 101054746). Views and opinions expressed are however those of the author only and do not necessarily reflect those of the European Union or the European Research Council Executive Agency. Neither the European Union nor the granting authority can be held responsible for them
\end{center}

\bibliographystyle{plain}
\bibliography{source}
\end{document}

%% file: librairies_and_commands.tex
\usepackage[utf8]{inputenc}  
\usepackage{csquotes}
\usepackage[T1]{fontenc}
\usepackage{graphicx}
\usepackage{tikz}
\usepackage{pgfplots}
\usepackage{pgfplotstable}
\pgfplotsset{compat=1.18}
\pgfplotsset{compat=1.18}
\usetikzlibrary{positioning}
\usepackage{float}
\usepackage{algorithm}
\usepackage{algorithmic}
\usepackage[utf8]{inputenc}
\usepackage[export]{adjustbox}
\usepackage{wrapfig}
\usepackage{amsfonts}
\usepackage{amsthm}
\usepackage{array}
\usepackage[top=2cm, bottom=2cm, left=3cm, right=3cm]{geometry}
\usepackage{amssymb}
\usepackage{xcolor}
\usepackage{mathtools}
\usepackage{comment}
\usepackage{algorithm}
\usepackage{algorithmic}
\usepackage[main=english]{babel}
\usepackage{subcaption}
\usepackage{dsfont}
\usepackage{lmodern}
\usepackage{indentfirst}
\usepackage{tabto}
\usepackage{color}
\usepackage{systeme}
\usepackage{eso-pic}
\usepackage{float}
\usepackage{fancyhdr}
\usepackage{appendix}
\usepackage{listings}
\usepackage{pgf, tikz}
\usepackage[utf8]{inputenc}
\usepackage{lastpage}

\definecolor{bf}{rgb}{0,0,0.6} 
\definecolor{mygray}{gray}{0.85}
\definecolor{darkWhite}{rgb}{0.94,0.94,0.94}

\graphicspath{{images/}}

\newcommand{\espcond}[2]{\mathbb{E}\mathopen{}\left[#1\middle|#2\right]}
\newcommand{\esp}[1]{\mathbb{E}\mathopen{}\left[#1\right]}

\newcommand{\reels}{\mathbb{R}}

\usepackage{thmtools,thm-restate}

\newcommand{\FuncDef}[4]{\ensuremath{\left\{\begin{array}{c} #1\longrightarrow #2\\ #3\mapsto #4\\\end{array}\right.}}

\usepackage{hyperref}

\numberwithin{equation}{section}

\usepackage{amsthm, amssymb}

\newtheorem{thm}{Theorem}[section]
\newtheorem{definition}[thm]{Definition}

\newtheorem{lemma}[thm]{Lemma}
\newtheorem{hyp}[thm]{Hypothesis}

\theoremstyle{definition}
\newtheorem{remarque}[thm]{Remark}
\newtheorem{exemple}[thm]{Example}